\documentclass[11pt]{amsart}
\usepackage{amsmath,amsthm,amssymb,enumerate}
\usepackage{setspace}
\newtheorem{theorem}{Theorem}[section]
\newtheorem*{theorem*}{Theorem}
\newtheorem{lemma}[theorem]{Lemma}
\newtheorem*{lemma*}{Lemma}
\newtheorem{example}[theorem]{Example}
\newtheorem*{example*}{Example}
\newtheorem{corollary}[theorem]{Corollary}
\newtheorem*{corollary*}{Corollary}
\newtheorem{definition}[theorem]{Definition}
\newtheorem*{definition*}{Definition}
\newtheorem{lemma-definition}[theorem]{Lemma-Definition}
\newtheorem*{lemma-definition*}{Lemma-Definition}
\newtheorem{definition-remark}[theorem]{Definition-Remark}
\newtheorem*{definition-remark*}{Definition-Remark}
\newtheorem{remark}[theorem]{Remark}
\newtheorem*{remark*}{Remark}
\newtheorem{notation}[theorem]{Notation}
\newtheorem*{notation*}{Notation}
\newtheorem{proposition}[theorem]{Proposition}
\newtheorem*{proposition*}{Proposition}

\def \g  {\mathfrak{g}}   
\def \h  {\mathfrak{h}}

\def \n  {\mathfrak{n}}

\def \b  {\mathfrak{b}}

\def \a  {\mathfrak{a}}
\def \sl {\mathfrak{sl}}

\def \t  {\mathfrak{t}}
\def \l  {\mathfrak{l}}

\def \Cset {{\mathbb C}}

\def \CC {{\mathbb C}}
\def \ZZ {{\mathbb Z}}

\def \dw {{\dot{w}}}

\def \bfu {{\bf u}}

\def \sA {{\scriptscriptstyle{A}}}
\def \piA {{\pi_{\sA}}}
\def \piAs {{\pi_{\sA^*}}}

\def \sG {{\scriptscriptstyle{G}}}

\def \sZ {{\scriptscriptstyle{Z}}}
\def \piZ {\pi_{\sZ}}

\def \sX {{\scriptscriptstyle{X}}}
\def \piX {\pi_{\sX}}
\def \sY {{\scriptscriptstyle{Y}}}
\def \piY {\pi_{\sY}}

\def \sB {{\scriptscriptstyle{B}}}

\def \pist {\pi_{\rm st}}
\def \Pist {\Pi_{\rm st}}
\def \Zu {Z_{\bfu}}
\def \la {\langle}
\def \ra {\rangle}
\def \lara {\la \, , \, \ra}
\def \O {\mathcal{O}}
\def \Upu {\Upsilon_{\bfu}}

\def \hs {\hspace{.2in}}
\def \hhs {\hspace{.12in}}
\def \bs {\bar{s}}

\def \lrw {\longrightarrow}

\def \sT {\scriptscriptstyle{T}}

\def \al {\alpha}
\def \lam {\lambda}

\def \calA {{\mathcal{A}}}

\def \calY {{\mathcal{Y}}}

\def \calU {{\mathcal{U}}}
\def \calP {{\mathcal{P}}}
\def \calQ {{\mathcal{Q}}}

\def \calG {{\mathcal{G}}}

\def \Guv {G^{u, v}}

\def \bu {{\overline{u}}}

\def \bs {{\overline{s}}}
\def \al {{\alpha}}

\def \bbv {{\overline{\overline{v}}}}
\def \calO {{\mathcal{O}}}
\def \Ovmu {\O^{(\bfv^{-1}, \bfu)}}
\def \bfv  {{\bf v}}
\def \piGu {\pi_{({\scriptscriptstyle G}, \bfu)}}

\def \bbv {{\overline{\overline{v}}}}
\def \Hol {{\mathcal{H}}}

\def \cYGu {{\calY^{\bfu}}}
\def \tcYGu {{\widetilde{\calY}^{\bfu}}}
\def \yij {y_{[i, j]}^{\lam}}
\def \bpi {0 \bowtie \pi_n}

\begin{document}
\setlength{\baselineskip}{1.2\baselineskip}

\title[Generalized Bruhat Cells and
Kogan-Zelevinsky Integrable Systems]{Generalized Bruhat Cells and Completeness of Hamiltonian Flows of
Kogan-Zelevinsky Integrable Systems}
\author{Jiang-Hua Lu}
\address{
Department of Mathematics   \\
The University of Hong Kong \\
Pokfulam Road, Hong Kong}
\email{jhlu@maths.hku.hk}
\author{Yipeng Mi}
\address{
Department of Mathematics   \\
The University of Hong Kong\\
Pokfulam Road, Hong Kong}
\email{littlemi@connect.hku.hk}
\date{}
\begin{abstract} Let $G$ be any connected and simply connected complex semisimple Lie group, equipped with a standard
holomorphic multiplicative Poisson structure. We show that the Hamiltonian
flows of all the Fomin-Zelevinsky twisted generalized minors on every double Bruhat cell of $G$
are complete in the sense that all the integral curves of their Hamiltonian vector fields
are defined on ${\mathbb{C}}$. It follows that the
Kogan-Zelevinsky integrable systems on $G$ have complete Hamiltonian flows, generalizing the result of Gekhtman and Yakimov
for the case of $SL(n, {\mathbb{C}})$. We in fact construct a class of integrable systems with complete
Hamiltonian flows associated to {\it generalized Bruhat cells} which are defined using
arbitrary sequences of elements in the Weyl group of $G$, and we
obtain the results for double Bruhat cells
through the so-called open {\it Fomin-Zelevinsky embeddings} of (reduced) double Bruhat cells in generalized Bruhat cells.
The Fomin-Zelevinsky embeddings are proved to be Poisson, and they provide global coordinates
on double Bruhat cells, called
{\it Bott-Samelson coordinates}, in which all the Fomin-Zelevinsky minors become polynomials
and the Poisson structure can be computed explicitly.

\end{abstract}
\maketitle
\hfill
In memory of Bertram Kostant

\section{Introduction}\label{sec-intro}
\subsection{Introduction}\label{subsec-intro}
Complex Poisson manifolds provide a natural setting for complex integrable systems. Here
recall that a complex Poisson manifold is a pair $(X, \pi)$, where $X$ is a complex manifold and
$\pi$, the {\it Poisson structure}, is a holomorphic section of $\wedge^2 T_X$ such that the bilinear map
$\Hol_X \otimes_\CC \Hol_X\to \Hol_X$ given by
\begin{equation}\label{eq-poi-fg}
\{f, \, g\} \; \stackrel{{\rm def}}{=} \; \pi (df, \, dg), \hs f, g \in \Hol_X,
\end{equation}
makes the sheaf $\Hol_X$ of local holomorphic functions on $X$ into a sheaf of Poisson algebras.
In such a case, the operation $\{\, , \, \}$  is called a {\it Poisson bracket} on ${\Hol}_X$, and $f, g \in \Hol_X$
are said to Poisson commute if $\{f, g\} = 0$. Given a complex Poisson manifold $(X, \pi)$
and a (local) holomorphic function $f$ on $X$, the
{\it Hamiltonian vector field} $H_f$ of $f$ is
the holomorphic vector field on $X$ defined by
\[
H_f(g) = \{f, \, g\}, \hs g \in \Hol_X,
\]
and $f$ is said to have {\it complete Hamiltonian flow}
if all the integral curves of $H_f$ on $X$ are defined on the whole of $\CC$. If the generic symplectic leaves
of $\pi$ in $X$ have dimension $2m$, a set $(y_1, \ldots, y_m)$ of $m$ holomorphic functions on $X$ is said to
define an {\it integrable
system on $X$ with complete Hamiltonian flows} if they are functionally independent when restricted to generic symplectic leaves, pairwise Poisson commute, and each $y_j$ has
complete Hamiltonian flow on $X$. An integrable system on $X$ with complete Hamiltonian flows
thus defines an analytic action of $\CC^m$ on $X$.

In this paper, we are mainly concerned with Hamiltonian flows and integrable
systems defined by regular functions on a smooth complex affine variety, equipped with
a Poisson structure that is algebraic.
We refer to \cite{Camille-Anne-Pol} for a systematic study of complex Poisson manifolds, algebraic Poisson varieties,
as well as their relations to integrable systems.

For an integer $m \geq 1$, let $M_m(\CC)$ be the vector space of all $m \times m$ matrices with complex entries.
In \cite{KW1, KW2}, B. Kostant and N. Wallach proved the remarkable result that the
complex Gelfand-Zeitlin
integrable system on $M_m(\CC)$ has complete Hamiltonian flows, thus giving rise to an analytic action of
$\CC^{m(m-1)/2}$ on $M_m(\CC)$. Here the Hamiltonian vector fields are defined with respect to the linear Kostant-Kirillov-Souriau
Poisson structure on $M_m(\CC)$, by identifying $M_m(\CC)$ with its dual vector space via the trace form and by using the
natural Lie algebra structure on $M_m(\CC)$.
Various aspects of this (complex) Gelfand-Zeitlin system on $M_m(\CC)$, such as its relation to representation theory and
the geometry of the orbits of the action by
$\CC^{m(m-1)/2}$, have been studied in detail by Kostant and Wallach \cite{KW1, KW2} and by M. Colarusso and S. Evens
\cite{Co-1, Co-Sam-1, Co-Sam-2, Co-Sam-3, Co-Sam-4}.

The vector space
$M_m(\CC)$ also has a well-known quadratic Poisson structure with its origin in quantum groups \cite{Misha-Milen}.
In \cite{Misha-Milen},  M. Gekhtman and M. Yakimov gave an analog of the Gelfand-Zeitlin integrable system
on $M_m(\CC)$ with respect to the quadratic Poisson structure and proved its completeness in a certain sense.
In the same paper \cite{Misha-Milen}, they also showed that the Hamiltonian flows
(with respect to the quadratic Poisson structure) of all minors on $M_m(\CC)$ are complete, and that the
Kogan-Zelevinsky integrable systems,
originally introduced by M. Kogan and A. Zelevinsky
\cite{KZ:integrable} for any connected and simply connected complex semisimple Lie group (see Example 
\ref{ex-KZ-system} for more detail),
have complete Hamiltonian flows for the case of $SL_m(\CC)$.

Let $\calQ$ be the algebra of all quasi-polynomials in one complex variable \cite[$\S$26]{Arnold:ODE}, i.e.,
$\calQ$ consists of all holomorphic functions on $\CC$ of the form
\begin{equation}\label{eq-gamma-t}
\gamma(c) = \sum_{k=1}^N q_k(c)e^{a_kc}, \hs c \in \CC,
\end{equation}
with each $q_k(c) \in \CC[c]$ and the $a_k$'s pairwise distinct complex numbers. It is in fact shown in \cite{Misha-Milen} that all the integral curves of the Hamiltonian vector field
(with respect to the quadratic Poisson structure on $M_n(\CC)$)
of every minor on $M_m(\CC)$ are of the form
\[
c \longmapsto (\gamma_{jk}(c))_{j,k = 1, \ldots, m}, \hs c \in \CC,
\]
where $\gamma_{jk} \in \calQ$ for all $j, k = 1, \ldots, m$. Motivated by the results in \cite{Misha-Milen}, we make the
following definition.


\begin{definition}\label{de-calS}
{\rm Let $X$ be a smooth affine variety  with an algebraic Poisson structure, and assume that $X \subset \CC^n$.
A regular function $y$ on $X$ is said to have {\it complete Hamiltonian flow with property $\calQ$} (with respect to the given embedding of $X$ in $\CC^n$) if
all the integral curves of the Hamiltonian vector field of $y$ are defined on $\CC$ and are of the form
\[
c \longmapsto (\gamma_1(c), \ldots, \gamma_n(c)), \hs c \in \CC,
\]
with $\gamma_j \in \calQ$ for every $j = 1, 2, \ldots, n$. An integrable system on $X$ defined by a set
$(y_1, \ldots, y_m)$ of regular functions on $X$ is said to have {\it complete Hamiltonian flows with property $\calQ$} if each $y_j$
has complete Hamiltonian flow with property $\calQ$.
\hfill $\diamond$
}
\end{definition}

As $\calQ$ is an algebra, the property for a regular function, and thus also for an integrable system,
to have complete Hamiltonian flow with property $\calQ$ is independent of the
embedding of $X$ into affine spaces, and is thus also invariant under biregular Poisson isomorphisms between
affine Poisson varieties (see Lemma \ref{le-simple-1}). In view of Definition \ref{de-calS},
the results in \cite{Misha-Milen} on the minors on $M_m(\CC)$ can be rephrased as saying that, with respect to the quadratic Poisson
structure on $M_m(\CC)$, all the minors on $M_m(\CC)$ have complete Hamiltonian flows with property $\calQ$, and that
the Kogan-Zelevinsky integrable systems for $SL(m, \CC) \subset M_m(\CC)$ have complete Hamiltonian flows
with property $\calQ$.

In this paper, we generalize the above result of \cite{Misha-Milen} for $SL(m, \CC)$ to any
connected and simply connected complex semisimple Lie group $G$.
More precisely, recall that double Bruhat cells in $G$ are defined as
\[
\Guv = (B u B) \cap (B_- v B_-),
\]
where $(B, B_-)$ is a pair of opposite Borel subgroups of $G$, and $u, v \in W$,
the Weyl group of $G$ with respect to
$T = B \cap B_-$. By \cite[Proposition 2.8]{BFZ}, each double Bruhat cell $G^{u, v}$ is an affine variety
(see also \cite[Proposition 3.1]{FZ:total} and \cite[Corollary 2.5]{FZ:total}, which imply that $G^{u, v}$
is biregularly isomorphic to
a principal Zariski open subset of an affine space).
We show that all the
{\it Fomin-Zelevinsky twisted generalized minors} (see Definition
\ref{de-FZ-minors}) on every double Bruhat cell $\Guv$ in $G$
have complete Hamiltonian flows with property $\calQ$. Consequently, the Kogan-Zelevinsky integrable
systems on every $G^{u, u}$, defined by certain Fomin-Zelevinsky twisted generalized minors and using
reduced words of $u$, have
complete Hamiltonian flows with property $\calQ$.
Here the group $G$ is equipped with the so-called
{\it standard multiplicative Poisson structure $\pist$}, the definition of which depends on the
choice of $(B, B_-)$ and the additional choice of a non-degenerate symmetric bilinear form $\lara_\g$
on the Lie algebra $\g$ of $G$ (see $\S$\ref{subsec-pist} for its precise definition). For $G = SL(m, \CC)$
and the standard choices of $B, B_-$, and $\lara_\g$,
$\pist$ extends to the
quadratic Poisson structure on $M_m(\CC)$ used in \cite{Misha-Milen}.

We remark that the Fomin-Zelevinsky twisted generalized minors on double Bruhat cells
have been studied extensively in the literature
(see, for example,
\cite{BZ, BFZ, FZ:total, GY,  GY:PNAS, GY:AMSbook, GY:Poi} and references therein),
specially because of their important
roles in the theories of total positivity, crystal bases,  and (classical and quantum)
cluster algebras.
In particular,
certain Fomin-Zelevinsky minors on $\Guv$
form an initial cluster of
the Berenstein-Fomin-Zelevinsky upper cluster algebra structure on the coordinate ring of
$\Guv$ (see \cite{BFZ}), which is compatible with the Poisson structure $\pist$ on
$\Guv$ in the sense \cite{GSV:book} that any two functions $f_1, f_2$ in the same cluster have {\it log-canonical} Poisson bracket, i.e.
\[
\{f_1, f_2\} = a f_1f_2
\]
for some constant $a$. We also refer to \cite{GY:Poi} for a remarkable proof, {\it using the Poisson structure $\pist$},
that the Berenstein-Fomin-Zelevinsky upper cluster algebra structure on the coordinate ring of $\Guv$
 coincides with
the corresponding cluster algebra structure.

Given a double Bruhat cell $\Guv$, our main tool for studying the Poisson manifold $(\Guv, \pist)$ is
a modification of the map defined by Fomin and Zelevinsky in \cite[Proposition 3.1]{FZ:total},
which  we call the (open) {\it Fomin-Zelevinsky embedding} (see Definition \ref{de-FZ-map})
\[
F^{u, v}: \;\; \Guv \lrw T \times \O^{(v^{-1}, u)},
\]
and the resulting global {\it Bott-Samelson coordinates} on $\Guv$. Here
$\O^{(v^{-1}, u)}$ is the so-called {\it generalized Bruhat cell} associated to
the elements $v^{-1}$ and $u$ of the Weyl group $W$.
The generalized Bruhat cells $\O^{\bfu}$, where $\bfu = (u_1, \ldots, u_n)$ is {\it any} sequence
of elements in $W$, is defined in \cite[$\S$1.3]{LM:flags}  as
\[
\O^\bfu = B u_1B \times_B \cdots \times_B B u_n B/B \;\subset \;F_n \; \stackrel{{\rm def}}{=}\;
G \times_B \cdots \times_B G/B,
\]
where $F_n$ is the quotient space of $G^n$ by a certain right action of the product group $B^n$
(see $\S$\ref{subsec-gBC} for detail).
The product Poisson structure $\pi_{{\rm st}}^n$ on $G^n$
projects to a well-defined Poisson structure $\pi_n$ on $F_n$, with respect to which every generalized Bruhat cell $\O^\bfu$
is a Poisson submanifold.
Fixing a pinning for $(G, T)$ (see $\S$\ref{subsec-nota-intro}),
the choice of a reduced word for each $u_j$ then gives rise to a parametrization of
$\O^{(u_1, \ldots, u_n)}$ by $\CC^{l(u_1) + \cdots + l(u_n)}$, where $l$ is the length function on $W$,
and we refer to the resulting coordinates on
$\O^{(u_1, \ldots, u_n)}$ as {\it Bott-Samelson coordinates}  (see Definition \ref{de-BS-coor-1}).
Explicit formulas  for the Poisson
structure $\pi_n$ on $\O^{(u_1, \ldots, u_n)}$ in Bott-Samelson coordinates have been studied in \cite{Balazs:thesis, EL:BS}.
By {\it Bott-Samelson coordinates on a double Bruhat cell $\Guv$} we mean the coordinates
on $\Guv$ obtained through the open Fomin-Zelevinsky embedding $F^{u, v}$ and the combination of coordinates on $T$ and
Bott-Samelson coordinates on  $\O^{(v^{-1}, u)}$.
We prove in Theorem \ref{thm-FZ-Poi} that the open Fomin-Zelevinsky embedding
\[
F^{u, v}: \;\; (\Guv, \; \pist) \lrw (T \times \O^{(v^{-1}, u)}, \; 0 \bowtie \pi_2)
\]
is Poisson, its image being a single $T$-orbit of symplectic leaves of
$(T \times \O^{(v^{-1}, u)}, \; 0 \bowtie \pi_2)$. Here $0 \bowtie \pi_2$ is the sum of $0 \times \pi_2$, the
product of the zero Poisson structure on $T$ with the Poisson structure $\pi_2$
on $\O^{(v^{-1}, u)}$, and a certain {\it mixed term} defined using the $T$-action on $\O^{(v^{-1}, u)}$
(see $\S$\ref{subsec-T-O} for detail). We then show that,
in Bott-Samelson coordinates on $\Guv$,
all the Fomin-Zelevinsky twisted generalized minors on $\Guv$ belong to a special class of polynomials
whose Hamiltonian flows are,
{\it manifestly},
complete with property $\calQ$.

We remark that for any sequence $(u_1, \ldots, u_n) \in W^n$, it is shown in \cite[$\S$5]{EL:BS} that,
through the identification of
the algebra of regular functions
on $\O^{(u_1, \ldots, u_n)}$ with the polynomial algebra $\CC[x_1, \ldots, x_l]$ via Bott-Samelson coordinates,
where $l=l(u_1) + \cdots + l(u_n)$,
the Poisson structure $\pi_n$ on $\O^{(u_1, \ldots, u_n)}$ makes $\CC[x_1, \ldots, x_l]$ into
a so-called {\it symmetric nilpotent semi-quadratic Poisson algebra} as defined in \cite[Definition 4]{GY:PNAS},
also called a {\it Poisson CGL extension} in \cite{GY:Poi}.
The Fomin-Zelevinsky embedding,
being an open Poisson embedding, thus gives an explicit identification of the coordinate ring
of $\Guv$, as a Poisson algebra, with a certain explicit localization of a Poisson CGL extension. Identifications
of this kind,
and their quantum analogs, have been used and play a crucial role in
the work of K. Goodearl and M. Yakimov in their proof that the
Berenstein-Fomin-Zelevinsky upper cluster algebra structure on the coordinate ring of $\Guv$ coincides with
the corresponding cluster algebra structure (both at the classical and quantum levels). See
\cite{GY, GY:PNAS, GY:AMSbook, GY:Poi}.
We thank M. Yakimov for pointing out to us that their identification in \cite{GY:Poi} of the coordinate ring of $\Guv$
with the localization
of a Poisson CGL extension  is through a Poisson analog of the method in
\cite[Proposition 4.4]{GY} to prove \cite[Theorem 4.1]{GY}.

We also point out that, based on the results in \cite{Balazs:thesis, EL:BS}, B. Elek has
written a computer program in the language of GAP that computes the Poisson structure $\pi_n$ 
on $\O^{(u_1, \ldots, u_n)}$ in the 
Bott-Samelson coordinates for any $G$ and any sequence $(u_1, \ldots, u_n) \in W^n$. This, in principle, allows
one to compute explicitly the Hamiltonian flows of the Fomin-Zelevinsky minors in Bott-Samelson coordinates.
We give examples for $G = SL(2, \CC)$, $G = SL(3, \CC)$, and $G = G_2$ in
Example \ref{ex-ss-0}, Example \ref{ex-A2-121} and Example \ref{ex-G2-212}, respectively.

In the first part of the paper, $\S$\ref{sec-GBC}, we develop a general theory on
complete Hamiltonian flows associated to an {\it arbitrary} generalized Bruhat cell
$\O^\bfu$ and the Poisson structure $\pi_n$, where $\bfu = (u_1, \ldots, u_n) \in W^n$.
We introduce a collection
$\calY^\bfu$ (see Definition \ref{de-class-Y}) of regular functions on $\O^\bfu$ and prove in
Theorem \ref{thm-Y-calS}, using Bott-Samelson coordinates
on $\O^\bfu$, that every  $y \in \calY^\bfu$ has complete Hamiltonian flow with property $\calQ$.
We also study a Poisson structure $0 \bowtie \pi_n$ on
$T \times \O^\bfu$ and introduce a collection $\tcYGu$ (see Definition \ref{nota-class-tY}) of regular functions
on $T \times \O^\bfu$, which are shown in Theorem \ref{thm-tcYGu} to all
have complete Hamiltonian flows with property $\calQ$.
Moreover, in $\S$\ref{subsec-int-O} we construct on $(\O^{(\bfu^{-1}, \bfu)}, \pi_{2n})$
an integrable system that has complete Hamiltonian flows with property $\calQ$, where
$\bfu^{-1} = (u_n^{-1}, \ldots, u_1^{-1})$, and $(\O^{(\bfu^{-1}, \bfu)}, \pi_{2n})$
can be regarded as a {\it double} of $(\O^\bfu, \pi_n)$.

In $\S$\ref{sec-KZ}, the second part of the paper, we relate double Bruhat cells and generalized Bruhat cells through the
Fomin-Zelevinsky embeddings.
Given a double Bruhat cell $\Guv$ and respective reduced words $\bfu$ and $\bfv$ of $u$ and $v$, we introduce in $\S$\ref{subsec-FZ-map} the
Fomin-Zelevinsky embeddings
\begin{align*}
&\hat{F}^{\bfu, \bfv}:\;\; \Guv/T \lrw \O^{(\bfv^{-1}, \bfu)} \cong \O^{(v^{-1}, u)},\\
&F^{\bfu, \bfv}: \;\; \Guv \lrw T \times \O^{(\bfv^{-1}, \bfu)} \cong T \times \O^{(v^{-1}, u)},
\end{align*}
where recall that $\Guv/T$ is the {\it reduced double Bruhat cell} corresponding to $\Guv$.
We prove in Proposition
\ref{pr-M-ij-y} that all the Fomin-Zelevinsky twisted generalized minors on $\Guv$
 are pullbacks by $F^{\bfu, \bfv}$ of
regular functions on
$T \times \O^{(\bfv^{-1}, \bfu)}$ from the collection $\widetilde{\calY}^{(\bfv^{-1}, \bfu)}$.
After proving in Theorem \ref{thm-bf-Fuv-Poi} that
 the Fomin-Zelevinsky embeddings are Poisson, we establish in Corollary \ref{co-FZ-complete}
and Corollary \ref{co-KZ-complete} that all the Fomin-Zelevinsky minors on every $\Guv$ and all the
Kogan-Zelevinsky integrable systems on every  $G^{u, u}$
have complete Hamiltonian flows with property $\calQ$. For each reduced word $\bfu$ of $u$,
by pulling back
the integrable system on $\O^{(\bfu^{-1}, \bfu)}$ constructed in
$\S$\ref{subsec-int-O} using $\hat{F}^{\bfu, \bfu}$, we also obtain an integrable system on $G^{u, u}/T$ that
has complete Hamiltonian flows
with property $\calQ$.
See $\S$\ref{subsec-KZ-system}.

The proof of the fact that the Fomin-Zelevinsky embeddings are Poisson, as given in $\S$\ref{subsec-FZ-Poi-1},
uses certain formulas for Poisson brackets between Fomin-Zelevinsky minors
by Kogan and Zelevinsky \cite{KZ:integrable} and their explicit expressions in
Bott-Samelson coordinates. As such a proof depends on many computations in different coordinates,
we give a more conceptual proof in the Appendix, which
involves some general theory on Poisson Lie groups and certain other interesting facts about the Poisson Lie group $(G, \pist)$.

\subsection{Further studies}\label{subsec-further}
We expect most of the results in $\S$\ref{sec-GBC} on the standard Poisson structure on 
generalized Bruhat cells to
hold for arbitrary symmetric Poisson CGL extensions. In particular, it
would be very interesting to construct a {\it double} of a symmetric Poisson CGL extension, as we have done 
from $(\O^\bfu,\! \pi_n)$
to $(\O^{(\bfu^{\!-1},\! \bfu)}, \pi_{2n})$, and produce an
integrable system on the double that has complete Hamiltonian flows with property $\calQ$.

Consider now a double Bruhat cell $G^{u, u}$, where $u \in W$. In addition to carrying Kogan-Zelevinsky
integrable systems, which we have now proved to have complete Hamiltonian flows with property $\calQ$, it is shown in
\cite{LM:groupoids} that $G^{u, u}$ also has the natural structure of a {\it Poisson groupoid} over the
Bruhat cell $BuB/B \subset G/B$. In fact, each symplectic leaf of $\pist$ in $G^{u, u}$
is shown in \cite{LM:groupoids} to be a symplectic groupoid over $BuB/B$.
It is very natural, then, to formulate any compatibility between the Kogan-Zelevinsky
integrable systems and the groupoid structure on $G^{u, u}$, and to ask whether there is any groupoid-theoretical
interpretation of the $\CC^{l(u)}$-action on $G^{u, u}$ defined by the Kogan-Zelevinsky systems.

Finally, as with the complex Gelfand-Zeitlin integrable system on the matrix space $M_m(\CC)$ pioneered by Kostant
and Wallach, it would be very interesting to see when the analytic action of $\CC^{l(u)}$ on $G^{u, u}$ descends to an algebraic action of $(\CC^\times)^{l(u)}$
and what roles
the property $\calQ$ of the Hamiltonian flows plays in understanding the geometry of the orbits of the action.

\subsection{Why property $\calQ$?}\label{subsec-simple-lemma}
Recall Definition \ref{de-calS} of property $\calQ$
of Hamiltonian flows. We first make a simple observation which explains the appearance of property $\calQ$
in this paper.
Let $(x_1, \ldots, x_n)$ be the standard coordinates on $\CC^n$ and let
$A = \CC[x_1, \ldots, x_n]$.
We identify an algebraic
Poisson structure on $\CC^n$ with the corresponding Poisson bracket $\{ \, , \, \}$ on the polynomial
algebra $A$.

\begin{lemma}\label{le-simple}
Let $\{\, , \}$ be any algebraic Poisson structure on $\CC^n$ and
let $y \in A$. Assume that there is a set
$\calG = (\tilde{x}_1, \ldots, \tilde{x}_n)$ of
generators for $A$ such that
for each $j = 1, 2, \ldots, n$, there exist $\kappa_j \in \CC$
and $f_j \in \CC[\tilde{x}_1, \ldots, \tilde{x}_{j-1}]$ ($f_j$ is a constant when $j =1$) such that
\begin{equation}\label{eq-y-tildex}
\{y, \, \tilde{x}_j\} = \kappa_j \tilde{x}_jy + f_j.
\end{equation}
Then $y$ has complete Hamiltonian flow with property $\calQ$.
\end{lemma}

\begin{proof}
As the change of coordinates $(x_1, \ldots, x_n) \to (\tilde{x}_1, \ldots, \tilde{x}_n)$ is biregular,
we may assume that
$(\tilde{x}_1, \ldots, \tilde{x}_n) = (x_1, \ldots, x_n)$. The Hamiltonian vector field of $y$
with respect to $\{ \, , \}$ is then given by
\[
H_y = \{y, x_1\} \frac{\partial}{\partial x_1} +
\{y, x_2\} \frac{\partial}{\partial x_2} + \cdots +
\{y, x_n\} \frac{\partial}{\partial x_n}.
\]
As $y$ is constant on any integral curve of $H_y$, solving for
an integral curve of $H_y$ is the same as solving for the functions
$(x_1(c), \ldots, x_n(c))$ from the system of ODEs
\begin{align*}
\frac{d x_1}{dc} &= \kappa_1 x_1 y_0 + f_1,\\
\frac{d x_2}{dc} &= \kappa_2 x_2 y_0 + f_2(x_1),\hs \cdots, \\
\frac{d x_n}{dc} &= \kappa_n x_n y_0 + f_n(x_1, \ldots, x_{n-1}),
\end{align*}
where $y_0$ and $f_1$ are constants. Solving for the $x_j$'s successively, one sees that
each $x_j$ is defined on  $\CC$ and is in $\calQ$.
\end{proof}

\begin{remark}\label{rm-shuffle}
{\rm In our applications of Lemma \ref{le-simple}, for example,
in the proofs of Theorem \ref{thm-Y-calS} and Theorem \ref{thm-tcYGu}, the set
$(\tilde{x}_1, \ldots, \tilde{x}_n)$ of generators of $\CC[x_1, \ldots, x_n]$ is often a shuffle of $(x_1, \ldots, x_n)$.
\hfill $\diamond$
}
\end{remark}

\begin{lemma}\label{le-simple-2}
With the same assumptions as in Lemma \ref{le-simple},
suppose that $X$ is a  smooth
affine Poisson subvariety of $\CC^n$ and
$g \in A$
is such that  $g|_X \neq 0$. Consider
${X'} = \{x \in X: g(x) \neq 0\} \subset X$,
regarded as an affine variety in $\CC^{n+1}$ via the embedding
\begin{equation}\label{eq-X-embed}
{X'} \lrw \CC^{n+1}, \;\; x \longmapsto (x, \;1/g(x)), \hs x \in {X'}.
\end{equation}
Equip ${X'}$ with the (algebraic) Poisson structure from $X$ by restriction. If
$\{y, g\} = \kappa yg$ for some $\kappa \in \CC$, then $y|_{X'}$
has complete Hamiltonian flow in $X'$ with property $\calQ$ (with respect to the embedding
in \eqref{eq-X-embed}).
\end{lemma}

\begin{proof} Let $I=\la a_1, \ldots, a_m\ra$ be the radical ideal of $A$ defining $X$.
Then there exist $b_{j, k} \in A$, $j, k = 1, \ldots, m$, such
that $\{y, \, a_k\} = \sum_{j=1}^m b_{j, k} a_j$ for each $k = 1, \ldots, m$.
Let $\gamma: \CC \to \CC^n$ be an integral curve of $H_y$
with $\gamma(0) \in X$. For $j, k = 1, \ldots, m$, let $f_k = a_k\circ \gamma$ and $g_{j,k} = b_{j, k} \circ \gamma$.
Then the functions $f_1, \ldots, f_m$ on $\CC$
satisfy the system of linear ODEs
\[
\frac{df_k}{dc} = \sum_{j=1}^m g_{j, k}(c) f_j, \hs k = 1, \ldots, m,
\]
with the initial conditions $f_1(0) = \cdots = f_m(0) = 0$. By the uniqueness of solutions of ODEs,
$f_k=0$ for $k = 1, \ldots, m$. Thus $\gamma(\CC) \subset X$.

Identifying the algebra $\O_{X'}$ of regular functions on ${X'}$ with the localization
$\O_X[(g|_X\!)^{-1}\!]$, the extension of $\{\, ,\, \}$ from $\O_X$ to $\O_{X'}$ indeed makes ${X'}$ into a (smooth) affine
Poisson variety.
Let $p \in {X'}$ and let $\gamma: \CC \to X$ be the integral curve of $H_y$ through $p$
in $X$.
Consider the function $g(\gamma(c))$, $c \in \CC$. It follows from $\{y, g\} = \kappa y g$ that
\[
\frac{d}{dc}(g(\gamma(c))) = \{y, g\}(\gamma(c)) = \kappa y(p) g(\gamma(c)), \hs c \in \CC.
\]
Thus $g(\gamma(c)) = g(p) e^{\kappa y(p) c}$ for all $c \in \CC$. In particular, $g(\gamma(c)) \neq 0$ for all $c \in \CC$ and thus
$\gamma(\CC) \subset {X'}$. Moreover,
as $1/g(\gamma(c)) =  e^{-\kappa y(p) c}/g(p)$ for all $c \in \CC$, it
follows by definition that $y|_{X'} \in \O_{X'}$ has complete Hamiltonian flow in ${X'}$ with property $\calQ$.
\end{proof}

We now show that property $\calQ$ for a regular function on an affine Poisson variety $X$ is independent of
the embedding of $X$ into an affine space.

\begin{lemma}\label{le-simple-1}
Let $X$ be a smooth affine variety with an algebraic Poisson structure $\pi$, and let
$\phi: X \to \CC^n$ and $\psi: X \to \CC^m$ be embeddings of $X$ as affine varieties. Suppose that a regular
function $y$ on $X$ has complete Hamiltonian flow with property $\calQ$ with respect to $\phi$. Then
$y$ also has complete Hamiltonian flow with property $\calQ$ with respect to $\psi$.
\end{lemma}

\begin{proof}
Let $\gamma: \CC \to X, c \mapsto \gamma(c)$, be an integral curve of the Hamiltonian vector field of $y$.
Writing $\phi \circ \gamma: \CC \to \CC^n$ as
\[
\phi(\gamma(c)) = (\gamma_1(c), \, \ldots, \gamma_n(c)), \hs c \in \CC,
\]
we know by assumption that $\gamma_j \in \calQ$ for each $j = 1, \ldots, n$. Let $z_1, \ldots, z_m$ be
the coordinates on $\CC^m$. Then for each $k = 1, \ldots, m$, $\psi^*(z_k)$ is a regular function on $X$, so
there exists $f_k \in \CC[x_1, \ldots, x_n]$ such that $\psi^*(z_k) = \phi^*(f_k)$
as regular functions on $X$. Define
\[
F: \; \; \; \CC^n \lrw \CC^m, \; \; F(x) = (f_1(x), \ldots, f_m(x)), \hs x \in \CC^n.
\]
Then $F \circ \phi = \psi: X \to \CC^m$, and thus
\[
\psi(\gamma(c)) = F(\phi(\gamma(c)) = F(\gamma_1(c), \ldots, \gamma_j(c)), \hs c \in \CC.
\]
As $\calQ$ is a subalgebra of the algebra of all $\CC$-valued functions on $\CC$, one has $\psi \circ \gamma \in
\calQ^m$.
\end{proof}

\subsection{Notation}\label{subsec-nota-intro}
We now fix some notation that will be used throughout the paper.

We will fix the connected and simply connected complex semisimple Lie group $G$ and also
fix a pair $(B, B_-)$ of opposite Borel subgroups of $G$. Let $N$ and $N_-$ be the respective uniradicals
of $B$ and $B_-$.
Let $T = B \cap B_-$, and let
$\t$ be the Lie algebra of $T$. Let $\Delta^+ \subset \t^*$ be the system of positive roots determined by $B$, and let
\[
\g = \t + \sum_{\alpha \in \Delta^+} (\g_\al + \g_{-\al})
\]
be the root decomposition of $\g$.
Let $N_G(T)$ be the normalizer subgroup of $T$ in $G$, and let $W = N_G(T)/T$
be the Weyl group. Define the right action of $W$ on $T$ by
\[
t^w = \dw^{-1} t \dw,  \hs t \in T,
\]
where for $w \in W$, $\dw$ is any representative of $w$ in $N_G(T)$.

Let $\Gamma \subset \Delta^+$ be
the set of all simple roots in $\Delta^+$.
For each $\alpha \in \Gamma$,
we will fix root vectors $e_\alpha \in \g_\al$  and $e_{-\alpha} \in \g_{-\al}$ such that
$h_\alpha \stackrel{{\rm def}}{=}
[e_\alpha, e_{-\alpha}]$ satisfies $\alpha(h_\alpha) = 2$. Recall that the collection
$(T, \, B, \; \{e_\al, \, e_{-\al}: \; \al \in \Gamma\})$ is called a {\it pinning} of $G$.
For $\alpha \in \Gamma$, let
\[
u_{\alpha}, \, u_{-\alpha}: \;\;\CC \longrightarrow G, \;\; u_{\alpha}(c) = \exp(c \,e_{\alpha})
\;\;\; \mbox{and} \;\;\; u_{-\alpha}(c) = \exp(c\,e_{-\alpha}), \hs c \in \CC,
\]
be the corresponding one-parameter unipotent subgroups of $G$, and define the representative of the simple reflection
$s_\alpha$ in $N_G(T)$ by
\[
\overline{s_\alpha} = u_\alpha(-1) u_{-\al}(1) u_\al(-1).
\]
By \cite{FZ:total}, if $w \in W$ and $w = s_{\al_1}s_{\al_2} \cdots s_{\al_l}$ is a reduced
decomposition of $w$, where $\al_j \in \Gamma$ for $j = 1, \ldots, l$, the representatives
\[
\overline{w} \; \stackrel{{\rm def}}{=}\; \overline{{s}_{\al_1}} \;\,\overline{{s}_{\al_2}}\; \cdots\; \overline{{s}_{\al_l}} \hs \mbox{and}
\hs
\overline{\overline{w}} \; \stackrel{{\rm def}}{=}\; \left(\overline{w^{-1}}\right)^{-1}
\]
of $w$ in $N_G(T)$ are independent of the choices of the reduced decompositions of
$w$, and $\overline{w_1w_2} = \overline{w_1} \,\overline{w_2}$ and $\overline{\overline{w_1w_2}} =
\overline{\overline{w_1}} \,\overline{\overline{w_2}}$
if $l(w_1w_2) = l(w_1) + l(w_2)$.

Let $X(T)$ be the character group of $T$. For $\lam \in X(T)$, denote the corresponding
group homomorphism $T \to \CC^\times$ by $t \mapsto t^\lam$ for $t \in T$. Define the left action of $W$ on $X(T)$ by
\[
(t^w)^{\lam} =t^{w(\lam)}, \hs t \in T, \, w \in W, \, \lam \in X(T).
\]
For $\lam \in X(T)$, we will also use $\lam$ to denote the
corresponding element in $\t^* = \CC \otimes_{\ZZ} X(T)$. The linear pairing between $\t$ and $\t^*$
is denoted by $(\, , \, )$.

Let $\calP^+ \subset X(T)$ be the set of all dominant weights, and for $\al \in \Gamma$, let
$\omega_\al \subset \calP^+$ be the corresponding fundamental weight. Recall from \cite{FZ:total} that
corresponding to each $\al \in \Gamma$ one has the {\it principal minor} $\Delta^{\omega_\al}$, which is the
regular function on $G$ whose restriction to $B_-B$ is given by
\[
\Delta^{\omega_{\al}}(g) = [g]_0^{\omega_\al},
\]
where for $g \in B_-B = N_- T N$, we write
\begin{equation}\label{eq-ggg}
g = [g]_- [g]_0 [g]_+, \hs \mbox{where} \;\;\; [g]_- \in N_-, \, [g]_0 \in T, \, [g]_+ \in N.
\end{equation}
Recall also from \cite{FZ:total} that a {\it generalized minor} on $G$ is a regular function on $G$ of the form
\[
g \longmapsto \Delta_{w_1\omega_\al, w_2\omega_\al}(g) \;\stackrel{{\rm def}}{=}\;
\Delta^{\omega_\al} (\overline{w_1}^{\,-1}g \overline{w_2}), \hs g \in G,
\]
where $w_1, w_2 \in W$ and $\al \in \Gamma$.
Let $\lam \in \calP^+$, and write $\lam$ as $\lam = \sum_{\al} n_\al \omega_{\al}$ with $n_\al \in {\mathbb{N}}$ for each $\al \in \Gamma$.
One then has the regular function $\Delta^\lam$ on $G$ defined by
\begin{equation}\label{eq-Delta-lam}
\Delta^{\lam} = \prod_{\al \in \Gamma} (\Delta^{\omega_{\al}})^{n_\al}.
\end{equation}
For $w_1, w_2 \in W$, define the regular function $\Delta_{w_1\lambda, w_2\lambda}$ on $G$ by
\begin{equation}\label{eq-Delta-ww-lam}
\Delta_{w_1\lambda, w_2\lambda}(g) = \Delta^\lambda(\overline{w_1}^{\,-1}g \overline{w_2}), \hs g \in G.
\end{equation}
It is easy to see that $\Delta_{w_1\lambda, w_2\lambda}$ is a matrix coefficient of the highest weight representation of
$G$ with highest weight $\lambda$ (see \cite[Definition 6.2]{MR:flag}), and
\begin{equation}\label{eq-Delta-Delta}
\Delta_{w_1\lam, w_2\lam} = \prod_{\al \in \Gamma} (\Delta_{w_1\omega_\al, w_2\omega_\al})^{n_\alpha}.
\end{equation}
Thus $\Delta_{w_1\lam, w_2\lam}$ is a monomial, with non-negative exponents, of generalized minors.

\subsection{Acknowledgments}\label{subsec-acknow} Research in this paper was partially supported by
the Research Grants Council of the Hong Kong SAR, China (GRF HKU 703712 and 17304415). We are grateful to M. Yakimov
for answering our questions on double Bruhat cells and for pointing out to us certain references. We
would also like to thank
Jun Peng and Shizhuo Yu for helpful discussions, and we thank the referee for helpful comments.

\section{Complete Hamiltonian Flows and Integrable Systems on Generalized Bruhat Cells}\label{sec-GBC}

\subsection{The complex Poisson Lie group $(G, \pist)$}\label{subsec-pist}
Let $G$ be any connected and simply connected complex semisimple Lie group with Lie algebra $\g$, and let
the notation
be as in $\S$\ref{subsec-nota-intro}.
We will also fix a non-degenerate symmetric invariant bilinear form $\lara_\g$ on  $\g$.
The restriction of $\lara_\g$ to $\t$ will be denoted by $\lara$. As $\lara$ is non-degenerate, one has the
isomorphism $\#: \t^* \to \t, \lam \mapsto \lam^\#$, given by
\begin{equation}\label{eq-sharp}
\la \lam^\#, \; x\ra = \lam(x) = (\lam, x), \hs \lam \in \t^*, \; x \in \t.
\end{equation}
Let $\lara$ also denote the bilinear form on $\t^*$ given by
\[
\la \lam_1, \, \lam_2\ra = \la \lam_1^\#, \; \lam_2^\#\ra, \hs \lam_1, \lam_2 \in \t^*.
\]
The choice of the triple $(T, B, \lara_\g)$ gives rise to the
{\it standard quasi-triangular $r$-matrix} $r_{\rm st}$ on $\g$ defined by (see \cite{dr:quantum, etingof-schiffmann})
\[
r_{\rm st} = \sum_{i=1}^r h_i \otimes h_i + \sum_{\alpha \in \Delta^+} \la \al, \, \al \ra
e_{-\al} \otimes e_\al  \in \g \otimes \g,
\]
where $\{h_i\}_{i=1}^r$ is any orthonormal basis of $\h$ with respect to $\lara$, and for $\al \in \Delta^+$,
$e_{\alpha} \in \g_\al$ and $e_{-\alpha}\in \g_{-\al}$ are such that
$\alpha([e_\al, e_{-\al}]) = 2$. Correspondingly, one has the
{\it standard multiplicative} holomorphic Poisson bi-vector field  $\pist$ on $G$ given by
\begin{equation}\label{eq-pist}
\pist(g) = l_g r_{\rm st}- r_g r_{\rm st} = l_g \Lambda_{\rm st} - r_g \Lambda_{\rm st}, \hs g \in G,
\end{equation}
where
$l_g$ (resp. $r_g$) for $g \in G$ is the left (resp. right) translation on $G$ by $g$, and
\[
\Lambda_{\rm st} = \sum_{\al \in \Delta^+} \frac{\la \al, \, \al \ra}{2}
(e_{-\al} \otimes e_{\al} - e_{\al} \otimes e_{-\al}) \in
\wedge^2\g
\]
is the skew-symmetric part of $r_{\rm st}$.
We refer to \cite{chari-pressley, dr:quantum, etingof-schiffmann} for the relation between the Poisson Lie group
$(G, \pist)$ and the corresponding quantum group of $G$. By, for example, \cite{hodges, reshe-4, KZ:integrable},
the double Bruhat cells $\Guv = (BuB) \cap (B_-vB_-)$,
where  $u, v \in W$, are precisely the $T$-orbits (under left or right
translation) of symplectic leaves
of $\pist$ in $G$. In particular, $BuB$ and $B_-uB_-$, for any $u \in W$, are
Poisson submanifolds
of $(G, \pist)$. The restriction of $\pist$ to $\Guv$ will still be denoted by $\pist$.

\begin{remark}\label{rk-compare-KZ}
{\rm
The Poisson structure $\pist$ in \eqref{eq-pist} is the negative of the one defined 
in \cite[$\S$2.5]{KZ:integrable} by Kogan and Zelevinsky, where the bilinear form $\lara$ on $\t^*$ is denoted as $(\, , \, )$.
\hfill $\diamond$
}
\end{remark}

\subsection{The Poisson structure $\pi_n$ on the generalized Bruhat cell $\O^\bfu$}\label{subsec-gBC}
For an integer $n \geq 1$, let the product group $B^n$ act on $G^n$ by
\[
(g_1, \, g_2, \, \ldots, \, g_n) \cdot (b_1, \, b_2, \ldots, b_n) =
(g_1b_1, \, b_1^{-1}g_2b_2, \, \ldots, \, b_{n-1}^{-1}g_nb_n), \hs g_j\in G, \, b_j \in B,
\]
and denote the corresponding quotient space by
\begin{equation}\label{eq-Fn}
F_n = G \times_B G \times_B \cdots \times_B G/B.
\end{equation}
Let $\varpi_n: G^n \to F_n$ be the natural projection. For a sequence $\bfu = (u_1, \ldots, u_n) \in W^n$,
let
\begin{equation}\label{eq-de-Ou}
\calO^\bfu = Bu_1B\times_B \cdots \times_B Bu_nB/B
\stackrel{{\rm def}}{=} \varpi_n((Bu_1B) \times \cdots \times (Bu_nB)) \subset F_n.
\end{equation}
The Bruhat decomposition $G = \bigsqcup_{u \in W} BuB$ of $G$ then gives the decomposition
\begin{equation}\label{eq-gBD}
F_n = \bigsqcup_{\bfu \in W^n} \calO^\bfu \hs (\mbox{disjoint union}).
\end{equation}
Following \cite{LM:flags}, each $\calO^\bfu \subset F_n$ is called
a {\it generalized Bruhat cell}, and when $\bfu$ is
a sequence of {\it simple reflections}, the generalized Bruhat cell $\calO^\bfu$ is said to be
{\it of Bott-Samelson type}. Let $l: W \to {\mathbb{N}}$ be again the length function on $W$.
It is clear that
\[
\dim \O^{\bfu} = l(\bfu) = l(u_1) + \cdots + l(u_n).
\]
It is shown in \cite[$\S$7.1]{LM:mixed} (see also \cite[Theorem 1.1]{LM:flags})
that
\[
\pi_n \; \stackrel{{\rm def}}{=}\; \varpi_n(\pi_{\rm st}^n),
\]
where $\pi_{\rm st}^n$ is the
product Poisson structure on $G^n$,
is a well-defined Poisson structure on $F_n$.
As $BuB$ is a Poisson
submanifold of $G$ with respect to $\pist$ for every $u \in W$,
all the generalized
Bruhat cells $\calO^\bfu$ in $F_n$ are Poisson submanifolds with respect to $\pi_n$.
The restriction of $\pi_n$ to
each $\calO^\bfu$, still denoted by $\pi_n$, will be referred to as the {\it standard Poisson
structure} on the generalized Bruhat cell $\calO^\bfu$.

Note that the action of $T$ on $F_n$ given by
\begin{equation}\label{eq-T-Fn}
t \cdot [g_1, g_2, \ldots, g_n] = [tg_1, g_2, \ldots, g_n], \hs t\in T, \;  g_1, g_2, \ldots, g_n \in G,
\end{equation}
preserves the Poisson structure $\pi_n$ on $F_n$, where for $(g_1, \ldots, g_n) \in G^n$,
\[
[g_1,\ldots, g_n] = \varpi_n(g_1, \ldots, g_n) \in F_n.
\]
Define a {\it $T$-leaf} of $(F_n, \pi_n)$  to be the union $\cup_{t \in T} (t\cdot \Sigma)$, where $\Sigma$ is
a symplectic leaf of $(F_n, \pi_n)$. For $\bfu \in W^n$ and $w \in W$, define
\begin{equation}\label{eq-O-bfu-w}
\O^\bfu_w = \{[g_1, g_2, \ldots, g_n] \in \O^\bfu: \; g_1g_2 \cdots g_n \in B_-wB\} \subset \O^\bfu.
\end{equation}
It is shown in \cite[Theorem 1.1]{LM:flags} that the decomposition
\begin{equation}\label{eq-Fn-T-leaves}
F_n = \bigsqcup_{\bfu \in W^n, \, w \in W} \O^\bfu_w \hs (\mbox{disjoint union})
\end{equation}
is that of $F_n$ into the $T$-leaves of $\pi_n$ (see \cite[Theorem 1.1]{LM:flags} for 
a criterion for $\O^\bfu_w
\neq \emptyset$). In particular, for each $\bfu \in W^n$, one has the (unique) open $T$-leaf
\begin{equation}\label{eq-Ou-e}
\O^\bfu_e = \{[g_1, g_2, \ldots, g_n] \in \O^\bfu: \; g_1g_2 \cdots g_n \in B_-B\}
\end{equation}
of $\pi_n$ in $\O^\bfu$, which will play an important role in this paper.

\begin{remark}\label{rk-Fn-Fn}
{\rm
We remark that every generalized Bruhat cell $\calO^\bfu \subset F_n$ with the Poisson structure $\pi_n$
is Poisson isomorphic to a generalized Bruhat cell of Bott-Samelson type with
the Poisson structure $\pi_{l(\bfu)}$,
where $l(\bfu) = l(u_1) + \cdots l(u_n)$.
Indeed, if $u_j = s_{{j, 1}} s_{{j, 2}} \cdots s_{{j, l(u_j)}}$
is a reduced decomposition of $u_j$, one then has the sequence
\[
\tilde{\bfu} = (s_{{1, 1}}, \,  s_{{1, 2}}, \,  \ldots, \,  s_{{1, l(u_1)}},\; \ldots, \;
s_{{n, 1}}, \,  s_{{n, 2}}, \,  \ldots s_{{n, l(u_n)}})
\]
of simple reflections of length $l(\bfu)$, and
the map $G^{l(u_1)} \times \cdots \times G^{l(u_n)} \to G^n$ given by
\[
(g_{1, 1}, \ldots, g_{1, l(u_1)}, \; \ldots, \; g_{n, 1}, \ldots, g_{n, l(u_n)})
\longmapsto (g_{1, 1}g_{1, 2} \cdots g_{1, l(u_1)}, \; \ldots, \; g_{n, 1} g_{n, 2} \cdots g_{n, l(u_n)})
\]
induces a $T$-equivariant Poisson isomorphism  (see \cite[$\S$1.3]{LM:flags})
\begin{equation}\label{eq-Fn-Fn}
(F_{l(\bfu)}, \,\pi_{l(\bfu)}) \supset (\O^{\tilde{\bfu}}, \, \pi_{l(\bfu)})
\longrightarrow (\calO^\bfu, \, \pi_n) \subset (F_n, \, \pi_n).
\end{equation}
Consequently, to study the Poisson manifold $(\O^\bfu, \pi_n)$ for an arbitrary
generalized Bruhat cell $\O^\bfu$, it is enough to study the case when $\O^\bfu$ is of Bott-Samelson type.
\hfill $\diamond$
}
\end{remark}

\subsection{The Poisson structure $\pi_n$ in Bott-Samelson coordinates on $\O^\bfu$}\label{subsec-Zu}
Assume first that $\bfu = (s_{\al_1},  s_{\al_2}, \, \ldots, \, s_{\al_n})$ is any sequence
of simple reflections in $W$, and for notational simplicity, we will also write
\begin{equation}\label{eq-u-si}
\bfu = (s_{\al_1},  s_{\al_2}, \, \ldots, \, s_{\al_n})= (s_1, \; s_2, \, \ldots, \; s_n).
\end{equation}
Corresponding to $\bfu$, one has the
Bott-Samelson variety
\[
Z_\bfu = \varpi_n(P_1 \times P_2 \times \cdots \times P_n) \subset F_n,
\]
where $P_j = B \cup Bs_jB$ for each $j$. It is easy to see that $\Zu$ is a (smooth and projective)
Poisson submanifold of $F_n$ with respect to the Poisson structure $\pi_n$, and
$(\O^\bfu, \pi_n)$ is embedded in $(\Zu, \pi_n)$ as an open Poisson submanifold.

The choice of the pinning in $\S$\ref{subsec-nota-intro}
 gives rise \cite{Dudas, EL:BS} to the atlas
\[
\calA = \{(\phi^\gamma: \; \CC^n \lrw \phi^\gamma(\CC^n)): \gamma \in \Upu\}
\]
on $\Zu$, where $\Upu$ is the set of all  {\it subexpressions of $\bfu$}, i.e.,
the set of all sequences
\[
\gamma = (\gamma_1, \ldots, \gamma_n) \in W^n,
\]
where  $\gamma_j =e$ or $\gamma_j = s_{\al_j}$ for each $j = 1, \ldots, n$, with $e$ being the identity element of $W$.
More specifically, for $\gamma = (\gamma_1,\ldots, \gamma_n) \in \Upu$, one has
the embedding $\phi^\gamma: \CC^n \to \Zu$,
\[
\phi^\gamma(x_1, \ldots, x_n) =
[u_{-\gamma_1(\alpha_1)}(x_1) \overline{{\gamma}_1}, \, u_{-\gamma_2(\alpha_2)}(x_2) \overline{{\gamma}_2}, \, \ldots, \,
u_{-\gamma_n(\alpha_n)}(x_n) \overline{{\gamma}_n}] \in \Zu,\\
\]
where $\bar{e} = e \in G$. For each $\gamma \in \Upu$, it is shown in \cite{Balazs:thesis, EL:BS} that
the Poisson structure $\pi_n$ is algebraic in the coordinate chart $\phi^\gamma: \CC^n \to \phi^\gamma(\CC^n)$,
and the Poisson brackets among the coordinate functions are
expressed using root strings and structure constants of the Lie algebra $\g$.
Note, in particular, that when $\gamma = \bfu$, $\phi^\bfu(\CC^n) = \O^\bfu$.

\begin{definition}\label{de-BS-coor}
{\rm
For a sequence $\bfu = (s_1, s_2, \ldots, s_n)$ of simple reflections,
we call the coordinates $(x_1, \ldots, x_n)$ on $\O^\bfu$ via $\phi^\bfu: \CC^n \to \O^\bfu$ given by
\begin{equation}\label{eq-phiu}
\phi^\bfu(x_1, \ldots, x_n) =
[u_{\al_1}(x_1)\bs_{1}, \,u_{\al_2}(x_2)\bs_{2}, \, \ldots, \,u_{\al_n}(x_n)\bs_{n}],\hs (x_1, \ldots, x_n) \in \CC^n,
\end{equation}
{\it Bott-Samelson coordinates} on $\O^\bfu$, and  we call
$\piGu = (\phi^{\bfu})^{-1}(\pi_n)$
the {\it Bott-Samelson Poisson structure on $\CC^n$ associated to $(G, \bfu)$}.
When necessary, we will denote the induced Poisson polynomial algebra by $(\CC[x_1, \ldots, x_n], \piGu)$
or
$(\CC[x_1, \ldots, x_n], \{\, , \, \}_{(\sG, \bfu)})$.
\hfill $\diamond$
}
\end{definition}

Note that in the Bott-Samelson coordinates, the $T$-action on $\O^\bfu$ in \eqref{eq-T-Fn}
is given by
\begin{equation}\label{eq-phi-ga-t}
t \cdot \phi^\bfu(x_1, x_2, \ldots, x_n) = \phi^\bfu(t^{\al_1}x_1, \; t^{s_1(\al_2)}x_2, \, \ldots,
\, t^{s_1s_2\cdots s_{n-1}(\al_n)}x_n), \hs t \in T.
\end{equation}

\begin{remark}\label{rk-T-alg}
{\rm
By a {\it $T$-Poisson algebra} we mean a Poisson algebra with a $T$-action by Poisson automorphisms.
By \eqref{eq-phi-ga-t}, the Poisson polynomial algebra $(\CC[x_1, \ldots, x_n], \piGu)$ is  a $T$-Poisson algebra with
the $T$-action
\begin{equation}\label{eq-t-xj}
t \cdot_\bfu x_j = t^{s_1s_2 \cdots s_{j-1}(\al_j)} x_j, \hs t \in T, \,\; j = 1, \ldots, n.
\end{equation}
For $1 \leq i \leq k \leq n$, let $\bfu_{[i, k]}=(s_i, s_{i+1}, \ldots, s_k)$, so one also has the $T$-Poisson algebra
$(\CC[x_i, x_{i+1}, \ldots, x_k], \, \pi_{\bfu_{[i, k]}})$ with the $T$-action
\begin{equation}\label{eq-t-xj-1}
t\cdot_{\bfu_{[i,k]}} x_j = t^{s_is_{i+1} \cdots s_{j-1}(\al_j)} x_j, \hs t \in T, \,\; j = i,  \ldots, k.
\end{equation}
By \cite[Theorem 4.14]{EL:BS}, the inclusion
\[
(\CC[x_i, \ldots, x_k], \;\pi_{(\sG, \bfu_{[i, k]})}) \hookrightarrow (\CC[x_1, \ldots, x_n], \;\piGu)
\]
is a Poisson algebra embedding but not $T$-equivariant due to \eqref{eq-t-xj} and \eqref{eq-t-xj-1}.
\hfill $\diamond$
}
\end{remark}

We do not need the full strength of the
explicit formulas for $\{\, ,\, \}_{(\sG, \bfu)}$ given in \cite[Theorem 4.14]{EL:BS},
but we will need the following property.

\begin{lemma}\label{le-iterated-0} Assume that $f \in \CC[x_1, \ldots, x_n]$ is a $T$-weight vector with weight
$\lam_f \in X(T)$.

1) If $f \in \CC[x_2, \ldots, x_n]$, then
\[
\{x_1, f\}_{(\sG, \bfu)} + \la \lam_f, \, \al_1\ra \,x_1 f \in \CC[x_2, \ldots, x_n];
\]

2) If $f \in \CC[x_1, \ldots, x_{n-1}]$, then
\[
\{f, x_n\}_{(\sG, \bfu)} + \la \lam_f, \; s_1s_2 \cdots s_{n-1}(\al_n)\ra \,x_nf \in \CC[x_1, \ldots, x_{n-1}].
\]
\end{lemma}

\begin{proof}
For $\zeta \in \t$, let $\partial_\zeta$ be the derivation of $\CC[x_1, \ldots, x_n]$ given by
\[
\partial_\zeta(x_j) = (s_1s_2 \cdots s_{j-1} (\al_j), \; \zeta) \, x_j, \hs j = 1, 2, \ldots, n.
\]
By \cite[Theorem 5.12]{EL:BS}, there exist a derivation $\delta$ of $\CC[x_2, \ldots, x_n]$ and a derivation
$\delta'$ on $\CC[x_1, \ldots, x_{n-1}]$ such that
\begin{align}\label{eq-ite-1}
\{x_1, \; f\}_{(\sG, \bfu)} &= -\frac{\la \al_1, \al_1\ra}{2} x_1\partial_{h_{\al_1}}(f) + \delta(f),\hs f \in
\CC[x_2, \ldots, x_n],\\
\label{eq-ite-2}
\{f, \; x_n\}_{(\sG, \bfu)} & = -\frac{\la \al_n, \al_n\ra}{2} x_n \partial_{s_1s_2 \cdots s_{n-1}(h_{\al_n})}(f) + \delta'(f),
\hs f \in
\CC[x_1, \ldots, x_{n-1}],
\end{align}
from which both assertions of Lemma \ref{le-iterated-0} now follow.
\end{proof}

One advantage of embedding $(\O^\bfu, \pi_n)$ into $(\Zu, \pi_n)$ is that we can make use of other
coordinate charts on $\Zu$ in which the Poisson structure $\pi_n$ may be simple.
Indeed, let
\[
(e) \, \stackrel{{\rm def}}{=} \, (e, e, \ldots, e)
\]
be the subexpression of $\bfu$ in which all the entries of $(e)$ are the identity element of $W$, and
consider the coordinate chart
$\phi^{(e)}: \CC^n \to \phi^{(e)}(\CC^n) \subset \Zu$
\begin{equation}\label{eq-phi-e}
\phi^{(e)} (\varepsilon_1, \ldots, \varepsilon_n) = [u_{-\al_1}(\varepsilon_1), \; u_{-\al_2}(\varepsilon_2), \; \ldots, \;
u_{-\al_n}(\varepsilon_n)], \hs (\varepsilon_1, \ldots, \varepsilon_n) \in \CC^n.
\end{equation}

\begin{lemma}\label{le-e-u} \cite[Theorem 4.14]{EL:BS}
In the coordinates $(\varepsilon_1, \varepsilon_2, \ldots, \varepsilon_n)$ on $\phi^{(e)}(\CC^n)$, the Poisson structure $\pi_n$ is {\it log-canonical}. More precisely,
\[
\{\varepsilon_i, \; \varepsilon_j\} = \la \al_i, \, \al_j\ra \varepsilon_i\varepsilon_j, \hs 1 \leq i < j \leq n.
\]
\end{lemma}

Lemma \ref{le-e-u} will be used in the proof of Proposition \ref{pr-y-lam} to compute the log-Hamiltonian
vector fields for certain regular functions on $\O^\bfu$.

\begin{definition}\label{de-BS-coor-1}
{\rm
For an arbitrary sequence $\bfu = (u_1, \ldots, u_n) \in W^n$, choose any reduced word for each $u_j$ and
form the sequence $\tilde{\bfu}$ of simple reflections as in Remark \ref{rk-Fn-Fn}.
Through the
Poisson isomorphism
$(\O^{\tilde{\bfu}}, \pi_{l(\bfu)}) \to (\O^\bfu, \pi_n)$ in \eqref{eq-Fn-Fn} in Remark \ref{rk-Fn-Fn}, the Bott-Samelson
coordinates on $\O^{\tilde{\bfu}}$ will also be called Bott-Samelson coordinates on $\O^\bfu$.
\hfill $\diamond$
}
\end{definition}

\begin{remark}\label{rm-Fn-Fn-1}
{\rm
For the remaining of $\S$\ref{sec-GBC}, we will work with generalized Bruhat cells $\O^\bfu$ of Bott-Samelson type,
i.e.,  we will assume that $\bfu$ is a
sequence of {\it simple reflections}. The results are trivially extended to arbitrary generalized
Bruhat cells by the Poisson isomorphism \eqref{eq-Fn-Fn} in Remark \ref{rk-Fn-Fn}. See also
Remark \ref{rk-Fn-Fn-2}.
\hfill $\diamond$
}
\end{remark}

\subsection{Homogeneous Poisson regular functions on $(\O^\bfu, \pi_n)$}\label{subsec-y-lam}
Let $(X, \pi)$ be a smooth affine Poisson variety, and let $(\O_X, \{\, , \,\})$ be the corresponding
Poisson algebra of regular functions on $X$. Recall that $y \in \O_X$ is
said to be {\it Poisson}  if the principal ideal of $\O_X$ generated by $y$ is a Poisson ideal.
 For such a Poisson
element $y$, define the {\it log-Hamiltonian vector field} $H_{\log (y)}$ on $X$ by
\begin{equation}\label{eq-log-Ha}
H_{\log(y)}(f) = \frac{1}{y} \{y, \; f\}, \hs f \in \O_X.
\end{equation}
If a complex algebraic torus $T$ acts on $(X, \pi)$ preserving the Poisson structure, let $T$ act on $\O_X$ by
\[
(t \cdot f)(x) = f(t \cdot x), \hs t \in T, \,\, x \in X,
\]
so that $\O_X$ becomes a $T$-Poisson algebra.
An element $y \in \O_X$ that is a $T$-weight vector will also be
said to be {\it $T$-homogeneous}.

\begin{lemma}\label{le-homo-poi} Let $(\CC[x_1, \ldots, x_n], \{ \, , \, \})$ be a polynomial $T$-Poisson algebra for
which  each $x_j$ is $T$-homogeneous. If $y \in \CC[x_1, \ldots, x_n]$ is $T$-homogeneous and Poisson, so is
every one of its prime factors.
\end{lemma}

\begin{proof}
Let $y_1 \in \CC[x_1, \ldots, x_n]$ be a prime factor of $y$ and write
$y = y^k_1y_2$, where $k$ is a positive integer, $y_2 \in \CC[x_1, \ldots, x_n]$, and $y_1$ does not divide $y_2$. For any $f \in \CC[x_1, \ldots, x_n]$, it follows from
\[
y^k_1y_2 H_{\log(y)}(f) = \{y^k_1y_2, \, f\} = y^k_1\{y_2, \, f\} + ky^{k-1}_1y_2 \{y_1 , f\}
\]
that $y_1|\{y_1, f\}$. Thus $y_1$ is Poisson. Equip
the (algebraic) character group $X(T)$ of $T$ with
a total ordering $\preceq$
such that
\[
\lam_1 \preceq \lam_2 \;\; \mbox{and} \;\; \lam_3 \preceq \lam_4\;\; \Longrightarrow\;\;
\lam_1 + \lam_3 \preceq \lam_2 + \lam_4
\]
for any $\lam_j \in X(T)$, $j = 1, 2, 3, 4$.
For each non-zero $f \in \CC[x_1, \ldots, x_n]$, let $\lam_{f, {\rm min}}$ and $\lam_{f, {\rm max}}$ be the
respective minimal and maximal weights appearing in the decomposition of $f$ into the sum of
$T$-weight vectors. It follows from
\[
\lam_{y_1^k, {\rm min}} \lam_{y_2, {\rm min}} = \lam_{y, {\rm min}} = \lam_{y, {\rm max}} =
\lam_{y_1^k, {\rm max}}\lam_{y_2, {\rm max}}
\]
that both $y_1^k$ and $y_2$, and thus $y_1$, are $T$-weight vectors.
\end{proof}

Let again $\bfu = (s_{\al_1}, \ldots, s_{\al_n}) = (s_1, \ldots, s_n)$.
We will identify all the  regular functions on $(\O^\bfu, \pi_n)$ that
are $T$-homogeneous and Poisson,  and we will compute their log-Hamiltonian vector fields.
For $\al \in \Gamma$ and $c \in \CC$, set
\begin{equation}\label{eq-p-al-c}
p_\al(c) = u_\al(c) \bs_\al \in G.
\end{equation}
Then the parametrization $\phi^\bfu: \CC^n \to \O^\bfu$ can be written as
\begin{equation}\label{eq-phi-u-p}
\phi^\bfu(x_1, x_2, \cdots, x_n) = [p_{\al_1}(x_1), \; p_{\al_2}(x_2), \; \ldots, \; p_{\al_n}(x_n)], \hs
(x_1, \ldots, x_n) \in \CC^n.
\end{equation}

Recall from $\S$\ref{subsec-nota-intro} that $\calP^+ \subset X(T)$ is the set of all dominant weights
on $T$, and that for $\lam \in \calP^+$, $\Delta^\lam$ is the regular function on $G$ defined in \eqref{eq-Delta-lam}.


\begin{definition}\label{de-y-lam}
{\rm For $\lam \in \calP^+$, let $y^\lam$ be the regular function on $\O^\bfu$ given by
\begin{equation}\label{eq-y-lam-1}
y^\lam([p_{\al_1}(x_1), \; p_{\al_2}(x_2), \; \ldots, \; p_{\al_n}(x_n)])=
\Delta^\lam(p_{\al_1}(x_1) p_{\al_2}(x_2) \cdots p_{\al_n}(x_n)),
\end{equation}
where $(x_1, \ldots, x_n) \in \CC^n$.
\hfill $\diamond$
}
\end{definition}

\begin{lemma}\label{le-y-lam-weight}
For any $\lam \in \calP^+$, the regular function $y^\lam$ on $\O^\bfu$ is
$T$-homogeneous with weight $\lam - s_{1}s_{2}\cdots s_{{n}}(\lam)$.
\end{lemma}

\begin{proof} Let $x = (x_1, \ldots, x_n) \in \CC^n$, $q(x) = [p_{\al_1}(x_1), \; \ldots, \; p_{\al_n}(x_n)]
\in \O^\bfu$, and $t \in T$. It
follows from
\begin{equation}\label{eq-t-px}
tp_{\al_1}(x_1)\cdots p_{\al_n}(x_n) = p_{\al_1}(t^{\al_1}x_1)
\cdots p_{\al_n}(t^{s_{1}\cdots s_{{n-1}}(\al_n)}x_n) t^{s_{1}s_{2}\cdots s_{{n}}} \in G
\end{equation}
and \eqref{eq-phi-ga-t} that
\begin{align*}
(t\cdot y^\lam)(q(x)) &= \Delta^\lam(p_{\al_1}(t^{\al_1}x_1)
\cdots p_{\al_n}(t^{s_{1}\cdots s_{{n-1}}(\al_n)}x_n))\\
& = \Delta^\lam\left(t p_{\al_1}(x_1)\cdots p_{\al_n}(x_n) (t^{s_{1}s_{2}\cdots s_{{n}}})^{-1}\right)\\
&= t^{\lam - s_1s_2 \cdots s_n(\lam)} y^\lam(q(x)).
\end{align*}
Thus $y^\lam$ is
$T$-homogeneous with weight $\lam - s_{1}s_{2}\cdots s_{{n}}(\lam)$.
\end{proof}

Let $\alpha \in \Gamma$.
We make some remarks on the function $y^{\omega_\al}$ on $\O^\bfu$
that will be used in the next Proposition \ref{pr-y-lam} and in $\S$\ref{subsec-int-O}.
For the given $\bfu = (s_1, \ldots, s_n) = (s_{\al_1}, \ldots, s_{\al_n})$,
if $\alpha \in \{\al_1, \ldots, \al_n\}$, let
\[
i_\al = {\rm min}\{1 \leq k \leq n: \al_k = \al\} \hs \mbox{and} \hs
j_\al= {\rm max}\{1 \leq k \leq n: \al_k = \al\}.
\]
We will regard $y^{\omega_\al}$ as an element in $\CC[x_1, \ldots, x_n]$ via the parametrization $\phi^\bfu$ of $\O^\bfu$.

\begin{lemma}\label{le-y-xx}
Let $\al \in \Gamma$. Then $y^{\omega_\al} = 1$ if $\al \notin \{\al_1, \ldots, \al_n\}$. Otherwise,
there exist $A_\al, B_\al, C_\al, D_\al \in \CC[x_{i_\al+1}, \ldots, x_{j_\al-1}]$ and
$A_\al \neq 0$, such that
\[
y^{\omega_\al}(x) = A_\al x_{i_\al} x_{j_\al} + B_\al x_{i_\al} + C_\al x_{j_\al} + D_\al
\in \CC[x_{i_\al}, \ldots, x_{j_\al}].
\]
\end{lemma}

\begin{proof}
Let $V_{\omega_\al}$ be the irreducible representation of $G$ with highest
weight
$\omega_\al$ and let $v_0 \in V_{\omega_\al}$ be a highest weight vector. Let ${\mathcal P}^\prime(\omega_\al)$ be
the set of all weights in $V_{\omega_\al}$ that are not equal to
$\omega_\al$, and for each $\mu \in {\mathcal P}^\prime(\omega_\al)$, let $V_{\omega_\al}(\mu) \subset V_{\omega_\al}$
be the corresponding weight space. Then the
function $\Delta^{\omega_\alpha}$ on $G$ is given by
\begin{equation}\label{eq-Delta-lam-al}
gv_0 -\Delta^{\omega_\al}(g)v_0 \in \sum_{\mu \in {\mathcal P}^\prime_{\omega_\al}} V_{\omega_\al}(\mu), \hs g \in G.
\end{equation}
Let ${\mathcal P}^{\prime\prime}(\omega_\al) = \{\mu \in {\mathcal P}^\prime(\omega_\al): \mu \neq \omega_\al -\al\}$.
Then one has the direct sum decomposition
\[
V_{\omega_\al} =\Cset v_0 + \Cset \overline{s_\alpha} v_0 + \sum_{\mu_\in {\mathcal P}^{\prime\prime}(\omega_\al)} V_{\omega_\al}(\mu).
\]
Let $SL(2, \Cset)$ act on $V_{\omega_\al}$ via the  group homomorphism $SL(2, \Cset) \to G$ determined by
the choices of $e_\al \in \g_\al$ and $e_{-\al} \in \g_{-\al}$ in $\S$\ref{subsec-nota-intro}.
Then  both $\Cset v_0 + \Cset \overline{s_\alpha} v_0$ and $\sum_{\mu_\in {\mathcal P}^{\prime\prime}(\omega_\al)} V_{\omega_\al}(\mu)$ are $SL(2, \Cset)$-invariant, and the resulting representation
 of $SL(2, \CC)$ on $\Cset v_0 + \Cset \overline{s_\alpha} v_0$ is
isomorphic to the standard one on $\Cset^2$. In particular, for any $c \in \Cset$, one has
\[
u_\alpha(c) \overline{s_\alpha} v_0 = c v_0 + \overline{s_\alpha} v_0,
\]
from which it follows (see also \cite[Lemma 7]{Knutson}) that for any $c \in \CC$,
\begin{equation}\label{eq-M-al-1}
\Delta^{\omega_\alpha}(g u_\alpha(c) \overline{s_\alpha}) = c \Delta^{\omega_\alpha}(g) + \Delta^{\omega_\alpha}(g\overline{s_\alpha}), \hs
\Delta^{\omega_\alpha}(u_\alpha(c) \overline{s_\alpha} g) = c \Delta^{\omega_\alpha}(g) + \Delta^{\omega_\alpha}(\overline{s_\alpha} g).
\end{equation}
Similarly, one shows (see \cite[$\S$2.3]{FZ:total}) that if $\alpha' \in \Gamma$ and $\alpha^\prime \neq \alpha$, then
\begin{equation}\label{eq-M-al-3}
\Delta^{\omega_\alpha}(g u_{\alpha'}(c) \overline{{s}_{\alpha^\prime}}) =
\Delta^{\omega_\alpha}(u_{\alpha'}(c) \overline{{s}_{\alpha^\prime}}g)=\Delta^{\omega_\alpha}(g), \hs g \in G, \, c \in \CC.
\end{equation}
Lemma \ref{le-y-xx} now follows by applying \eqref{eq-M-al-1} and \eqref{eq-M-al-3}.
\end{proof}

Recall from \eqref{eq-Ou-e} the Zariski open subset $\O^\bfu_e$ of $\O^\bfu$ defined by
\begin{equation}\label{eq-Ofue}
\O^\bfu_e = \{[p_{\al_1}(x_1), \; p_{\al_2}(x_2), \; \ldots, \; p_{\al_n}(x_n)]:\;
p_{\al_1}(x_1) p_{\al_2}(x_2) \cdots p_{\al_n}(x_n)\in B_-B\}.
\end{equation}
Define $\tau: \O^\bfu_e \to T$ by (recall notation in \eqref{eq-ggg})
\begin{equation}\label{eq-tau}
\tau(\phi^\bfu(x))=
[p_{\al_1}(x_1) p_{\al_2}(x_2) \cdots p_{\al_n}(x_n)]_0, \hs \phi^\bfu(x) \in \O^\bfu_e.
\end{equation}
Note that, by definition, for each $\lam \in \calP^+$, one has
\begin{equation}\label{eq-y-lam-Oe}
y^\lam(\phi^\bfu(x))= (\tau(\phi^\bfu(x)))^\lam, \hs \mbox{for} \;\; \phi^\bfu(x) \in \O^\bfu_e.
\end{equation}
Let $\sigma: \t \to {\mathfrak{X}}^1(\Zu)$ be the action of $\t$ on
$\Zu$ induced by the $T$-action in \eqref{eq-T-Fn}, i.e.,
\begin{equation}\label{eq-sigma-zeta}
\sigma(\zeta) (q) = \frac{d}{ds}|_{s=0} \exp(s\zeta) \cdot q, \hs \zeta \in \t, \; q \in \Zu.
\end{equation}
Recall the map $\#: \t^* \to \t, \lam \mapsto \lam^\#,$ given in
\eqref{eq-sharp}.

\begin{proposition}\label{pr-y-lam}
The set $\{cy^\lam: c \in \CC, \, \lam \in \calP^+\}$ is the set of all regular functions on $\O^\bfu$
that are $T$-homogeneous and Poisson with respect to $\pi_n$,  and for $\lam \in \calP^+$,
\begin{equation}\label{eq-X-y-lam}
H_{\log(y^\lam)} = -\sigma(\lam^\# + s_{1} s_{2} \cdots s_{n}(\lam^\#)).
\end{equation}
\end{proposition}

\begin{proof} Recall from \eqref{eq-phi-e} the coordinate chart $\phi^{(e)}: \CC^n \to \O^{(e)}$, and note that
\[
\O^{(e)} \cap \O^\bfu = \{[u_{-\alpha_1}(\varepsilon_1), \, u_{-\alpha_2}(\varepsilon_2), \, \ldots, \,
u_{-\alpha_n}(\varepsilon_n)]: \varepsilon_i \neq 0,\, i = 1, \ldots,n\} \subset \O^\bfu_e.
\]
For $\alpha \in \Gamma$, let $\al^\vee:  \CC^\times \to T$ be the
co-character of $T$ corresponding to $h_\al \in \t$, so that
\[
(\al^\vee(c))^{\mu} = c^{(\mu, \,h_\al)}, \hs \forall \; c \in \CC^\times, \, \,\mu \in X(T),
\]
where recall that $h_\al \in [\g_\al, \g_{-\al}]$ is such that
$\al(h_\al) = 2$.
We first prove that
\begin{equation}\label{eq-tn}
\tau(\phi^{(e)}(\varepsilon)) = \prod_{i=1}^n (\al_i^\vee(\varepsilon_i^{-1}))^{s_{{i+1}}s_{i+2} \cdots s_{n}}, \hs
\varepsilon =(\varepsilon_1, \ldots, \varepsilon_n)\in (\CC^\times)^n.
\end{equation}
To this end, let $x =(x_1, \ldots, x_n)\in \CC^n$ be such that $\phi^{(e)}(\varepsilon) = \phi^\bfu(x) \in \O^{(e)} \cap \O^\bfu$.
Then for each $j = 1, 2, \ldots, n$, there exist $t_j \in T$ and $u_j \in N$ such that
\[
u_{-\alpha_1}(\varepsilon_1) \cdots u_{-\alpha_{j-1}}(\varepsilon_{j-1})
u_{-\alpha_j}(\varepsilon_j) = p_{\al_1}(x_1) \cdots p_{\al_{j-1}}(x_{j-1}) p_{\al_j}(x_j) t_ju_j,
\]
and $\tau(\phi^{(e)}(\varepsilon)) = t_n^{-1} \in T$ by definition. A calculation in $SL(2, \CC)$ shows that
\begin{equation}\label{eq-al-al}
u_{-\al}(c) = u_\al(c^{-1}) \overline{s_\al} \al^\vee(c) u_\al(c^{-1}), \hs
\forall \; \alpha \in \Gamma,  \; c \in \CC^\times.
\end{equation}
It follows that $t_1 = \al_1^\vee(\varepsilon_1)$, and
\[
t_j = (t_{j-1})^{s_{j}} \al_j^\vee(\varepsilon_j), \hs j = 2, \ldots, n.
\]
Thus \eqref{eq-tn} holds. It now follows from \eqref{eq-y-lam-Oe} and \eqref{eq-tn} that
on $\O^{(e)} \cap \O^\bfu \subset \O^{(e)}$ and in the coordinates $(\varepsilon_1, \ldots, \varepsilon_n)$ on $\O^{(e)}$,
one has
\begin{equation}\label{eq-y-ep}
y^\lam = \prod_{i=1}^n \varepsilon_i^{-(s_{{i+1}}s_{i+2}\,\cdots \,  s_{n}(\lam), \; h_{\al_i})}.
\end{equation}
To prove \eqref{eq-X-y-lam}, it is enough to show that it holds on the Zariski open subset $\O^{(e)} \cap \O^\bfu$
of $\O^\bfu$.
It is easy to see that the coordinate function $\varepsilon_j$ on $\O^{(e)}$ satisfies
\[
\sigma(\zeta) (\varepsilon_j) = -\al_j(\zeta) \varepsilon_j, \hs \;\; \zeta \in \t, \;\ j = 1, 2, \ldots, n.
\]
It is thus enough to prove that, for any $\lam \in \calP^+$,
\begin{equation}\label{eq-y-xi}
\{y^\lam, \; \varepsilon_j\} = \la \lam + s_{1} \cdots s_{{n-1}} s_{n}(\lam), \; \al_j\ra \,y^\lam
\, \varepsilon_j, \hs
j = 1, 2, \ldots, n.
\end{equation}
For $i = 1, 2, \ldots, n$, let $m_i = (s_{{i+1}}\,\cdots \,  s_{{n-1}} s_{n}(\lam), \; h_{\al_i})$.
Let $j = 1, 2, \ldots, n$. By \eqref{eq-y-ep} and Lemma \ref{le-e-u}, one has
\[
\{y^\lam, \; \varepsilon_j\} = \la \beta_j, \; \al_j\ra \,y^\lam\varepsilon_j,
\]
where $\beta_j = -\sum_{i=1}^{j-1} m_i\al_i + \sum_{i=j+1}^n m_i\al_i$. It is easy to see that
\[
\beta_j = \lam + s_1s_2 \cdots s_n(\lam) - s_js_{j+1} \cdots s_n(\lam) - s_{j+1} \cdots s_n(\lam),
\]
and as $\la s_js_{j+1} \cdots s_n(\lam) +s_{j+1} \cdots s_n(\lam), \; \al_j\ra = 0$,
\eqref{eq-y-xi} holds. It follows that
$y^\lam$ is Poisson with respect to $\pi_n$,
and that its log-Hamiltonian vector field on $\O^\bfu$ is given as in \eqref{eq-X-y-lam}.

Assume now that $y$ is any non-zero regular function on $\O^\bfu$ that is
$T$-homogeneous and Poisson with respect to $\pi_n$.
It remains to show that
$y = cy^\lam$ for some $c \in \CC^\times$ and $\lam \in \calP^+$. Let $y_1$ be a prime factor of $y$.
By Lemma \ref{le-homo-poi}, $y_1$ is $T$-homogeneous and Poisson. Let $X_1 = \{q \in \O^\bfu: y_1(q)=0\}$.
Then $X_1$ is an irreducible $T$-invariant Poisson divisor of $(\O^\bfu, \pi_n)$.
On the other hand, consider the $T$-leaf decomposition (see $\S$\ref{subsec-gBC})
\[
\O^\bfu = \bigsqcup_{w \in W} \O^\bfu_w
\]
of $(\O^\bfu, \pi_n)$, where $\O^\bfu_w$ is defined in \eqref{eq-O-bfu-w}, and note that
\[
\O^\bfu\backslash \O^\bfu_e = \bigcup_{\alpha \in \{\al_1, \al_2, \ldots, \al_n\}}
\overline{\O^\bfu_{s_\al}}
\]
is the decomposition of the divisor $\O^\bfu\backslash \O^\bfu_e$ into its irreducible components,
where $\overline{\O^\bfu_{s_\al}}$ is the Zariski closure of $\O^\bfu_{s_\al}$ in $\O^\bfu$.
If $X_1 \cap \O^\bfu_e \neq \emptyset$, as $\O^\bfu_e$ is a single $T$-leaf,
$X_1$ would contain an open subset of
$\O^\bfu$, which is not possible. Thus $X_1$ is an irreducible component of the divisor
$\O^\bfu \backslash \O^\bfu_e$, and hence $X_1 = \overline{\O^\bfu_{s_{\al_i}}}$ for some $i = 1,2,\ldots, n$.
On the other hand, by \cite[Proposition 2.4]{FZ:total}, the Zariski closure
$\overline{B_-s_{\alpha_i} B}$ of $B_-s_{\alpha_i} B$ on $G$ is precisely the zero set of $\Delta^{\omega_{\alpha_i}}$.
It follows that $\overline{\O^\bfu_{s_{\al_i}}}$ is the zero set of the function
$y^{\omega_{\al_i}}$. As $\overline{\O^\bfu_{s_{\al_i}}}$ is irreducible, $y^{\omega_{\al_i}}$
can not have more than one prime factor. By Lemma \ref{le-y-xx}, there exist $1 \leq j \leq n$, $E, F \in
\CC[x_1, \ldots, x_{j-1}]$, $E \neq 0$, such that $y^{\omega_{\alpha_i}} = x_j E + F$, which implies that
$y^{\omega_{\al_i}}$ must be prime itself.   It follows that
$y_1 = cy^{\omega_{\al_i}}$ for some $c \in \CC^\times$. This finishes the proof of
Proposition \ref{pr-y-lam}.
\end{proof}

\begin{remark}\label{rk-y-lam}
{\rm
By Proposition \ref{pr-y-lam} and its proof, the regular functions on $\O^\bfu$ that
are $T$-homogeneous, Poisson with respect to $\pi_n$ and prime, are,  up to scalar multiples, precisely the elements $y^{\omega_\al}$
with $\al \in \{\al_1, \al_2, \ldots, \al_n\}$. Moreover, one has
\begin{equation}\label{eq-O-e-y}
\calO^\bfu_e = \{q \in \calO^\bfu: \; y^{\omega_\alpha}(q) \neq 0 \; \forall \; \alpha \in \{\alpha_1, \ldots, \alpha_n\}\}.
\end{equation}
In particular, $\O^\bfu_e$ is an affine variety with an embedding
\begin{equation}\label{eq-Oe-embedding}
\O^\bfu_e \lrw \CC^{n+1}, \;\;\; q \longmapsto ((\phi^\bfu)^{-1}(q), \; 1/\varphi(q)),
\end{equation}
where $\varphi = \prod_{\alpha \in \{\al_1, \ldots, \al_n\}} y^{\omega_\al}$.
The restriction of $\pi_n$ to $\O^\bfu_e$ then makes $\O^\bfu_e$ into a smooth affine Poisson variety.
\hfill $\diamond$
}
\end{remark}

\begin{corollary}\label{co-y-f}
For any $\lam \in \calP^+$ and for any regular function $f$ on $\O^\bfu$ that is $T$-homogeneous with weight $\lam_f$,
one has
\[
\{y^\lam, \;f\} = -\la \lam + s_1s_2 \cdots s_n(\lam), \; \lam_f\ra\, y^\lam f.
\]
In particular, for $\lam, \lam' \in \calP^+$, one has
\begin{equation}\label{eq-yy-bracket}
\{y^\lam, \;y^{\lam'}\} =
(\la \lam, \; s_1s_2 \cdots s_n(\lam')\ra - \la \lam', \; s_1s_2 \cdots s_n(\lam)\ra)
\; y^\lam y^{\lam'}.
\end{equation}
\end{corollary}

\begin{remark}\label{rk-tau}
{\rm
For $w \in W$, equip $T$ with the Poisson structure $\pi_{(\sT, w)}$  given by
\[
\{t^\lam, \; t^{\lam'}\}_{(\sT, w)}  = (\la \lam, \, w(\lam')\ra - \la \lam', \, w(\lam)\ra)\, t^{\lam + \lam'},
\hs \lam, \lam' \in X(T).
\]
It follows from \eqref{eq-yy-bracket} that
$\tau: (\O^\bfu_e, \,\pi_n) \rightarrow (T, \, \pi_{(\sT, s_1s_2 \cdots s_n)})$
is a Poisson morphism, where $\tau: \O_e^\bfu \to T$ is given in \eqref{eq-tau}.
\hfill $\diamond$
}
\end{remark}

\subsection{The collection $\cYGu$ of regular functions on $(\O^\bfu, \pi_n)$}\label{subsec-Y} Let again
\[
\bfu = (s_{\al_1}, \ldots, s_{\al_n}) = (s_1, \ldots, s_n)
\]
be any sequence of simple reflections for $G$, and recall the parametrization
\[
\phi^\bfu: \;\;\CC^n \longrightarrow \O^\bfu, \;\;
\phi^\bfu(x_1, x_2, \ldots, x_n) = [p_{\al_1}(x_1), \; p_{\al_2}(x_2), \; \ldots, \; p_{\al_n}(x_n)],
\]
where recall that $p_\al(c) = u_\al(c) \bs_\al$ for $\al \in \Gamma$ and $c \in \CC$.

\begin{definition}\label{de-class-Y}
{\rm
Define, for $1 \leq i \leq n$, $0 \leq i-1 \leq j \leq n$,
and $\lam \in \calP^+$, the regular function $\yij$ on $\O^\bfu$ by
$y_{[i, j]}^\lam = 1$ when $j = i-1$, and
\begin{equation}\label{eq-y-lam-ij}
\yij([p_{\al_1}(x_1), \; \ldots, \; p_{\al_n}(x_n)]) = \Delta^\lam(p_{\al_i}(x_i) p_{\al_{i+1}}(x_{i+1}) \cdots p_{\al_j}(x_j)),
\hs i \leq j.
\end{equation}
Moreover, set
\begin{equation}\label{eq-cYGu}
\cYGu = \left\{c\, \yij: \; c \in \CC, \;  1 \leq i \leq n, \, 0 \leq i-1 \leq j \leq n,  \;
\lam \in \calP^+\right\}.
\end{equation}
Note that the constant functions and the coordinate functions $x_1, \ldots, x_n$ are all in $\cYGu$.
\hfill $\diamond$}
\end{definition}

\begin{lemma}\label{le-yij-weights}
For $1 \leq i \leq n, \, 0 \leq i-1 \leq j \leq n$,
and $\lam \in \calP^+$, $\yij$ is $T$-homogeneous with weight
$s_1s_2 \cdots s_{i-1}(\lam) - s_1s_2\cdots s_j(\lam) \in X(T)$.
\end{lemma}

\begin{proof} The statement clearly holds when $j = i-1$. Let $1 \leq i \leq j \leq n$.
For $t \in T$ and $x = (x_1, \ldots, x_n) \in \CC^n$, similar to \eqref{eq-t-px}, one has
\begin{align*}
&tp_{\al_1}(x_1)\cdots p_{\al_{i-1}}(x_{i-1}) = p_{\al_1}(t^{\al_1}x_1)
\cdots p_{\al_{i-1}}(t^{s_{1}\cdots s_{{i-2}}(\al_{i-1})}x_{i-1}) t^{s_{1}\cdots s_{{i-1}}}, \\
&tp_{\al_1}(x_1)\cdots p_{\al_j}(x_j) = p_{\al_1}(t^{\al_1}x_1)
\cdots p_{\al_j}(t^{s_{1}\cdots s_{{j-1}}(\al_j)}x_j) t^{s_{1}s_{2}\cdots s_{{j}}}.
\end{align*}
It follows that
\[
p_{\al_i}(t^{s_1\cdots s_{i-1}(\al_i)}x_i)
\cdots p_{\al_j}(t^{s_{1}\cdots s_{{j-1}}(\al_j)}x_j) =t^{s_{1}\cdots s_{{i-1}}}
p_{\al_i}(x_i)\cdots p_{\al_{j}}(x_{j}) (t^{-1})^{s_1 \cdots s_j}.
\]
Similar to the proof of Lemma \ref{le-y-lam-weight}, one sees that
 $\yij$ is $T$-homogeneous with weight $s_1s_2 \cdots s_{i-1}(\lam) - s_1s_2\cdots s_j(\lam)$.
\end{proof}


Consider now the Poisson structure $\pi_n$ on $\O^\bfu$.
By Remark \ref{rk-y-lam}, we also have the affine Poisson variety $(\O^\bfu_e, \pi_n|_{\O^\bfu_e})$.
Regard every $y \in \cYGu$ also as a regular function on $\O^\bfu_e$ by restriction, we
now prove the first main result of this paper.

\begin{theorem}\label{thm-Y-calS}
Every $y \in \cYGu$ has complete Hamiltonian flows with property $\calQ$ in both $\O^\bfu$ and $\O^\bfu_e$.
\end{theorem}

\begin{proof}
Identifying $\O^\bfu\cong \CC^n$ by $\phi^\bfu$, we  now apply Lemma \ref{le-simple}.
Let $\lam \in \calP^+$, $1 \leq i \leq j \leq n$,  and
$y = \yij \in \CC[x_1, \ldots, x_n]$. Consider the re-ordering
of the coordinates
\[
(\tilde{x}_1, \;\ldots, \; \tilde{x}_n) = (x_i, \, \ldots, x_j, \, x_{j+1}, \, \ldots, x_n, \, x_{i-1}, \, \ldots, x_1).
\]

{\bf Case 1:} $i \leq k \leq j$. Applying Corollary \ref{co-y-f} and Remark \ref{rk-T-alg} to the Poisson subalgebra
$\CC[x_i, \ldots, x_j]$, one sees
that $\{y, x_k\}$ is log-canonical;

{\bf Case 2:} $j +1 \leq k \leq n$. Applying Lemma \ref{le-yij-weights} and 2) of
Lemma \ref{le-iterated-0}  to the Poisson subalgebra
$\CC[x_i, \ldots, x_j, x_{j+1}, \ldots, x_k]$,
one sees that there exists $c \in \CC$ such that
\[
\{y, x_k\} - c y x_k \in \CC[x_i, \ldots, x_j, x_{j+1}, \ldots, x_{k-1}];
\]

{\bf Case 3:} $1 \leq k \leq i-1$. Applying Lemma \ref{le-yij-weights} and 1) of
Lemma \ref{le-iterated-0}  to the Poisson subalgebra
$\CC[x_k, \ldots, x_{i-1}, x_i, \ldots, x_j]$, one sees that there exists $c' \in \CC$ such that
\[
\{y, x_k\} - c' y x_k \in \CC[x_{k+1}, \ldots, x_i, \ldots, x_j]
\subset \CC[x_i, \, \ldots, x_j, \, x_{j+1}, \, \ldots, x_n, \, x_{i-1}, \, \ldots, x_{k+1}].
\]

\noindent
In all the three cases, the conditions in Lemma \ref{le-simple}
on the Poisson brackets between $y$ and the coordinate functions
$\tilde{x}_1, \;\ldots, \; \tilde{x}_n$ are satisfied. Thus $y$
has complete Hamiltonian flow in $\CC^n \cong \O^\bfu$ with property $\calQ$.

Let $\varphi =  \prod_{\alpha \in \{\al_1, \ldots, \al_n\}} y^{\omega_\al}$. By Corollary \ref{co-y-f},
the Poisson bracket
$\{y, \varphi\}$ is log-canonical.
It now follows from Lemma \ref{le-simple-2} that $y$, as a regular function on the affine Poisson variety
$(\O^\bfu_e, \pi_n|_{\O^\bfu_e})$,
has complete Hamiltonian flow in $\O^\bfu_e$ with property $\calQ$.
\end{proof}

\begin{notation}\label{nota-s-interval}
{\rm
Set $s_{[k, l]} = s_ks_{k+1}\cdots s_l$ for $1 \leq k \leq l \leq n$, and
$s_{[k, l]} = e$ when $k > l$.
\hfill $\diamond$
}
\end{notation}

\begin{lemma}\label{le-y-y} For $1 \leq i' \leq i \leq n, \, 0 \leq i-1  \leq j \leq j' \leq n$, and for any $\lam \in \calP^+$, one has
\[
\left\{\yij, \; \;y_{[i', j']}^{\lam'}\right\}  = \left\langle s_{[1, i-1]}(\lam) - s_{[1, j]}(\lam), \;
s_{[1,i'-1]}(\lam')+s_{[1, j']}(\lam')\right\rangle \; \yij \;y_{[i', j']}^{\lam'}.
\]
\end{lemma}

\begin{proof} The statement clearly holds when $j = i-1$. Assume that $1 \leq i' \leq i \leq j
\leq j' \leq n$.
By Corollary \ref{co-y-f} and
Lemma \ref{le-yij-weights} applied to $(\O^{(s_{i'}, s_{i'+1}, \ldots, s_{j'})}, \pi_{j'-i'+1})$, one has
\begin{align*}
\left\{\yij, \; \;y_{[i', j']}^{\lam'}\right\} & = \left\langle \lam' + s_{[i', j']}(\lam'), \;
s_{[i', i-1]}(\lam)-s_{[i', j]}(\lam)\right\rangle \; \yij \;y_{[i', j']}^{\lam'}\\
& =\left\langle s_{[1, i-1]}(\lam) - s_{[1, j]}(\lam), \;
s_{[1,i'-1]}(\lam')+s_{[1, j']}(\lam')\right\rangle \; \yij \;y_{[i', j']}^{\lam'}.
\end{align*}
\end{proof}

\subsection{The Poisson structure $0 \bowtie \piX$ on $T \times X$}\label{subsec-T-Z}
Let $T$ be a complex algebraic torus, and suppose that $(X, \piX)$ is a smooth affine Poisson variety with a $T$-action by Poisson isomorphisms. Denote again the Lie algebra of $T$ by $\t$, and let
$\theta: \t \to {\mathfrak{X}}^1(X)$ be the induced Lie algebra action of $\t$ on $X$, i.e.,
\[
\theta(\zeta)(x) = \frac{d}{ds}|_{s =0} \exp(s\zeta) x, \hs \zeta \in \t, \,\; x \in X.
\]
For
any linear map
$M: \t^* \to \t$, we define the Poisson structure $0 \bowtie_M \piX$ on $T \times X$ by
\begin{equation}\label{eq-M-piX}
0 \bowtie_M \piX = (0, \, \piX) + \sum_{i=1}^r ((M^*(\zeta_i^*))^L, 0) \wedge (0, \theta(\zeta_i)),
\end{equation}
where $\{\zeta_i\}_{i=1}^r$ is any basis of $\t$, $\{\zeta_i^*\}_{i=1}^r$ its dual basis of $\t^*$, and for $\zeta \in \t$,
$\zeta^L$ denotes the left (also right) invariant vector field on $T$ with value $\zeta$ at the identity element.
In particular,
for $\lam \in X(T)$ and $f \in \O_X$,
\[
\{t^\lam, \; f\} = t^\lam \; \theta(M(\lam))(f).
\]

\begin{notation}\label{nota-bpi}
{\rm Assume again that $\lara$ is a symmetric nondegenerate bilinear form on $\t$, and let
$\#: \t^* \to \t$ be as given in \eqref{eq-sharp}.
For  $M = -\#: \t^* \to \t$, the corresponding
Poisson structure on $T \times X$ will be denoted by $0 \bowtie \piX$. Then
\[
0 \bowtie \piX = (0, \, \piX) -\sum_{i=1}^r (h_i^L, \;0) \wedge (0, \;\theta(h_i)),
\]
where $\{h_i\}_{i=1}^r$ is any
orthonormal basis
of $\t$ with respect to $\lara$.
\hfill $\diamond$
}
\end{notation}

Recall again that for $x \in X$, the {\it $T$-leaf} of $\piX$ through $x$ is
the union
$L = \bigcup_{t \in T} t\Sigma_x$,
where $\Sigma_x$ is the symplectic leaf of $\piX$ through $x$, and the {\it leaf stabilizer}
(see \cite[$\S$2]{LM:flags}) of $\t$ at $x$
is the subspace $\t_x$ of $\t$ given by
\[
\t_x = \{\zeta \in \t: \theta(\xi)(x) \in T_x \Sigma_x\}.
\]
(Here and for the the rest of $\S$\ref{subsec-T-Z}, the $T$ in the tangent space $T_xX$ should not
be confused with the torus $T$.)
It is shown in \cite[$\S$2]{LM:flags} that
for a given $T$-leaf of $\piX$, the leaf stabilizer $\t_x$ is independent of $x \in L$
and that the codimension of any symplectic leaf of $\piX$ in $L$ is equal to the co-dimension of
$\t_x$ in $\t$.
Equip $T \times X$ with the diagonal action
\[
t_1 \cdot (t, x) = (t_1t, \, t_1x), \hs t_1, t \in T, \, x \in X.
\]
It is clear that $0 \bowtie \piX$ is $T$-invariant. In the rest of $\S$\ref{subsec-T-Z},
we determine the $T$-leaves and the leaf stabilizers
of $0 \bowtie \piX$ in $T \times X$. We first recall from \cite{Lu:Duke-Poi, LM:flags}
the construction of the {\it Drinfeld Lagrangian subspace} $\l_x$
of $\t \oplus \t$
associated to the $T$-action on $(X, \piX)$.

For any Poisson manifold $(Z, \pi)$ and $z \in Z$, let $\pi_z^\flat: T_z^*Z \to T_zZ$ be given by
\[
\pi_z^\flat(\al_z) (\beta_z) = \pi(z)(\al_z, \beta_z), \hs \al_z, \beta_z \in T_z^*Z,
\]
so that
${\rm Im}(\pi_z^\flat)$ is the tangent space at $z$ to the symplectic leaf of $\pi$ through $z$. Let now $L \subset
X$ be a $T$-leaf of $\piX$, and let $x \in X$. Consider the
linear map
\[
\theta_x: \;\; \t \lrw T_xX, \;\; \theta_x(\zeta) = \theta(\zeta)(x), \hs \zeta \in \t,
\]
and its dual map $\theta_x^*: T^*_xX \to \t^*$.
Recall from \cite{Lu:Duke-Poi} and
\cite[$\S$2]{LM:flags} that associated to the $T$-Poisson manifold $(L, \piX)$ one has the Lie algebroid
$A = (L \times \t) \bowtie T^*L$ over $L$ with the surjective anchor map $\rho =-\theta + \piX^\flat$ from $A$
to the tangent bundle of $L$.
At any $x \in L$, the kernel of $\rho_x: \t \oplus T_x^*L \to T_xL$ is then
\[
\ker \rho_x = \{(\zeta, \, \al_x) \in \t \oplus T_x^*L: \;\; \theta_x(\zeta) = \pi_{\sX, x}^\flat(\al_x)\}.
\]
Consider the vector space $\t \oplus \t^* \cong \t \oplus \t$ with the symmetric bilinear form
\begin{equation}\label{eq-t-t-lara}
\la (\zeta_1, \,\zeta_2), \; (\zeta_1^\prime, \,\zeta_2^\prime)\ra = \la \zeta_1, \, \zeta_2^\prime\ra +
\la \zeta_1^\prime,,\zeta_2\ra, \hs \zeta_1, \,\zeta_2, \,\zeta_1^\prime, \, \zeta_2^\prime \in \t.
\end{equation}
For each $x \in L$, the linear map
\[
\ker \rho_x \lrw \t \oplus \t, \;\; (\zeta, \, \al_x) \longmapsto (\zeta, \;\, (\theta_x^*(\al_x))^\#),
\]
is then injective \cite{dr:homog, Lu:Duke-Poi}, and its image, denoted by
\[
\l_x \; \stackrel{{\rm def}}{=} \; \{(\zeta, \, (\theta_x^*(\al_x))^\#): \; (\zeta, \al_x) \in \ker \rho_x\},
\]
is Lagrangian with respect to the bilinear form $\lara$ on $\t \oplus \t$
given in \eqref{eq-t-t-lara}, and it is called the {\it Drinfeld Lagrangian subspace} of $\t \oplus \t$
associated to $x \in L$ and the $T$-action on $(L, \piX)$.
 By definition, the leaf stabilizer $\t_x$ of $\t$ at $x$ is
$p_1(\l_x)$, where $p_1: \t \oplus \t \to \t$ is the projection to the first factor.
Let $p_2: \t \oplus \t \to \t$ be the projection to the second factor, and introduce
\begin{align}\label{eq-flip-lz}
&{\rm Flip}(\l_x) = \{(\zeta_1, \, \zeta_2): \; (\zeta_2, \, \zeta_1) \in \l_x\} \subset \t \oplus \t,\\
\label{eq-tilde-tz}
&\tilde{\t}_x = p_2(\l_x) = \{ (\theta_x^*(\al_x))^\#: \; \al_x \in T^*_x L, \;
\pi_{\sX, x}^\flat(\al_x) = \theta_x(\zeta)\; \mbox{for some} \; \zeta \in \t\} \subset \t.
\end{align}

\begin{lemma}\label{le-T-leaves-T-X}
With respect to the diagonal action of $T$ on $T \times X$, the $T$-leaves of the Poisson structure $0 \bowtie \piX$
in $T \times X$ are precisely all the $T \times L$, where $L$ is a $T$-leaf of $\piX$ in $X$. Moreover,
for a $T$-leaf $L$ in $X$ and $(t, x) \in T \times L$, the Drinfeld Lagrangian subspace at $(t, x)$
for the diagonal $T$-action on $(T \times L, \, 0 \bowtie \piX)$ is $\l_{(t, x)} = {\rm Flip}(\l_x)$,
the leaf-stabilizer of $\t$ at $(t, x)$ is
$\tilde{\t}_x$, and the rank of the Poisson structure $0 \bowtie \piX$ at $(t, x)$ is
equal to $\dim L + \dim \tilde{\t}_x$.
\end{lemma}

\begin{proof}
Let $\pi = 0 \bowtie \piX$ and $(t, x) \in T \times X$. For $a \in \t^*$, let $a^L$ be the left invariant $1$-form on
$T$ with value $a$ at the identity element.
It follows from the definition of $\pi$ that
\[
\pi_{(t, x)}^\flat(a^L(t), \al_x) = \left(((\theta_x^*(\al_x))^\#)^L(t), \; \pi_{\sX, x}^\flat(\al_x) -
\theta_x(a^\#)\right), \hs a \in \t^*, \, \al_x \in T_x^*X.
\]
Let $\tilde{\theta}_{(t, x)}: \t \to T_{(t, x)}(T \times X), \, \zeta \mapsto (\zeta^L(t), \theta_x(\zeta))$,
and let $L$ be the $T$-leaf of $\piX$ through $x$. It is then clear that
\[
{\rm Im}(\tilde{\theta}_{(t, x)}) + {\rm Im} (\pi_{\sX, x}^\flat) = T_{(t, x)} (T \times L),
\]
from which it follows that $T \times L$ is the $T$-leaf of $\pi = 0 \bowtie \piX$
through $(t, x)$.

Let $\l_{(t, x)} \subset \t \oplus \t$ be the Drinfeld Lagrangian subspace  for the
diagonal $T$-action on $T \times X$. By definition,
\begin{align*}
\l_{(t, x)} & = \left\{(\zeta, \, (a + \theta_x^*(\al_x))^\#): \; \zeta \in \t, \, a \in \t^*, \; \al_x \in T_x^*L, \;
\zeta = (\theta_x^*(\al_x))^\#, \; \right.\\
& \hspace{1.5in} \;\;\left.
\theta_x(\zeta) = \pi_{\sX, x}^\flat(\al_x) -\theta_x(a^\#)\right\},\\
& = \{(\zeta, \, a^\# + \zeta): \; \zeta \in \t, \, a \in \t^*, \; (a^\# + \zeta, \; \zeta) \in \l_x\} = {\rm Flip}(\l_x).
\end{align*}
The rest of statements in Lemma \ref{le-T-leaves-T-X} follows from the fact that $\l_{(t, x)} = {\rm Flip}(\l_x)$.
\end{proof}

Let now $\bfu = (u_1, u_2, \ldots, u_n) \in W^n$ be arbitrary, and consider the generalized Bruhat cell
$\O^\bfu \subset F_n$, equipped with the Poisson structure $\pi_n$ and the $T$-action given in
\eqref{eq-T-Fn}, where $T = B \cap B_-$. One then has the Poisson structure $\bpi$ on $T \times \O^\bfu$.  By definition,
the projections to the two factors
\[
(T \times \O^\bfu, \; \bpi) \lrw (T,\; 0) \hs \mbox{and} \hs
(T \times \O^\bfu, \; \bpi) \lrw (\O^\bfu, \; \pi_n)
\]
are Poisson,  where $(T, 0)$ denotes $T$ with the zero Poisson structure, and for any $\lam \in X(T)$ and any $T$-homogeneous regular function $f$ on $\O^\bfu$ with $T$-weight $\lam_f$,
\begin{equation}\label{eq-t-lam-f}
\{t^\lam, \; f\} = -\la \lam, \; \lam_f\ra\, t^\lam f.
\end{equation}
Let $1 - u_1u_2 \cdots u_n: \t \to \t$ be  given by
$\zeta \mapsto \zeta - u_1u_2 \cdots u_n(\zeta)$, $\zeta \in \t$. Recall that $l(\bfu) = l(u_1) + l(u_2) + \cdots + l(u_n)$.

\begin{proposition}\label{pr-T-O-u}
Equip $T \times \O^\bfu$ with  the diagonal $T$-action.
Then $T \times \O^\bfu_e$ is a single $T$-leaf of $(T \times \O^\bfu, \, 0 \bowtie \pi_n)$.
The leaf-stabilizer of $\t$ in $T \times \O^\bfu_e$ is ${\rm Im}(1 -u_1u_2 \ldots u_n)$,
and the symplectic leaves of $0 \bowtie \pi_n$ in $T \times \O^\bfu$ have dimension equal to
\[
l(\bfu) + \dim \, {\rm Im}(1 -u_1u_2 \ldots u_n).
\]
\end{proposition}

\begin{proof} By \cite[Theorem 1.1]{LM:flags}, $\O^\bfu_e$ is a single $T$-leaf of $\pi_n$ for the $T$-action on $\O^\bfu$.
By Lemma \ref{le-T-leaves-T-X}, $T \times \O^\bfu_e$ is a single $T$-leaf of $0 \bowtie \pi_n$
for the diagonal $T$-action on $T \times \O^\bfu$. Let $x \in \O^\bfu_e$. The rest of Proposition \ref{pr-T-O-u}
would follow from Lemma \ref{le-T-leaves-T-X} once we prove that the Drinfeld Lagrangian
subspace $\l_x \subset \t \oplus \t$ for the $T$-action on $(\O^\bfu_e, \pi_n)$ is given by
\begin{equation}\label{eq-l-x-12}
\l_x = \{(\zeta + u_1u_2 \cdots u_n(\zeta), \; \; -\zeta + u_1u_2 \cdots u_n(\zeta)): \; \zeta \in \t\}.
\end{equation}
To prove \eqref{eq-l-x-12}, note that  by Remark \ref{rk-Fn-Fn},
we may assume that $\bfu = (s_1, s_2, \ldots, s_n)$ is a sequence of simple
reflections. By \cite[Theorem 1.1]{LM:flags}, the leaf stabilizer of
$\t$ in $\O^\bfu_e$ is ${\rm Im}(1 + s_1s_2 \ldots, s_n)$.
Consider now the regular functions $y^\lam$ on $\O^\bfu$ defined in $\S$\ref{subsec-y-lam}, where $\lam \in \calP^+$.
By Proposition \ref{pr-y-lam}, for every $\lam \in \calP^+$, one has
\[
\sigma(\lam^\# + s_1s_2 \cdots s_n(\lam^\#)) = -\pi_n^\flat\left(\frac{d y^\lam}{y^\lam}\right).
\]
As the function $y^\lam$ has $T$-weight $\lam - s_1s_2\cdots s_n(\lam)$ by Lemma \ref{le-y-lam-weight},
it follows from definition that $\l_x$ is as described.
\end{proof}

\subsection{The collection $\tcYGu$ of regular functions on $(T \times \O^\bfu, \, 0 \bowtie \pi_n)$}\label{subsec-T-O}
Let again $\bfu = (s_1, \ldots,s_n)$ be any sequence of simple reflections in  $W$, and recall from Definition
\ref{de-class-Y} the
collection $\cYGu$ of regular functions on $\O^\bfu$.

\begin{definition}\label{nota-class-tY}
{\rm
Let $\tcYGu$ be the set of all regular functions on $T \times \O^\bfu$ of the form
$t^\lam \, y$ with $\lam \in X(T)$ and $y \in \cYGu$, i.e,
\[
\tcYGu = \left\{c\,t^\lam\, y_{[i,j]}^{\lam'}: \; c \in \CC, \;  1 \leq i \leq n, \, 0 \leq i-1 \leq j \leq n,  \;
\lam \in X(T), \, \lam' \in \calP^+\right\}.
\]
\hfill $\diamond$}
\end{definition}

Note that both $T \times \O^\bfu$ and $T \times \O^\bfu_e$ are principal
Zariski open subsets in affine spaces and are thus affine varieties.
Denote the restriction of the Poisson structure $\bpi$ to  $T \times \O^\bfu_e$ also  by  $\bpi$.
One then has the (smooth) Poisson affine varieties $(T \times \O^\bfu, \, \bpi)$ and
$(T \times \O^\bfu_e, \, \bpi)$. We also regard $\tilde{y} \in \tcYGu$ as a regular function on $T \times \O^\bfu_e$
by restriction.

\begin{theorem}\label{thm-tcYGu}
Every $\tilde{y} \in \tcYGu$ has complete Hamiltonian flows with property $\calQ$
in both $T \times \O^\bfu$ and $T \times \O^\bfu_e$.
\end{theorem}

\begin{proof}
Let $\beta_1, \ldots, \beta_r$ be all the simple roots. Let $(\xi_1, \ldots, \xi_r)$ be the
coordinates on $\CC^r$, and identify $T$ with $(\CC^\times)^r\subset \CC^r$ by
\[
T \ni t \longmapsto (\xi_1(t), \ldots, \xi_r(t)) \in \CC^r,
\]
where $\xi_j(t) = t^{\omega_{\beta_j}}$. Let again $(x_1, \ldots, x_n)$ be the
coordinates on $\O^\bfu$ via the parametrization $\phi^\bfu: \CC^n \to \O^\bfu$.
Note then that by \eqref{eq-t-lam-f}, the Poisson structure $\bpi$ on $T \times \O^\bfu
\cong (\CC^\times)^r \times \CC^n$ extends
to a unique algebraic Poisson structure on $\CC^r \times \CC^n$, still denoted as $\bpi$, such that
for any $T$-homogeneous regular function $f$ on $\O^\bfu$ with $T$-weight $\lam_f$,
\begin{equation}\label{eq-xi-f}
\{\xi_k, \; f\} = -\la \omega_{\beta_k}, \; \lam_f\ra \, \xi_k f, \hs k = 1, \ldots, r.
\end{equation}
Consider $\CC^{1+r+n}$ with coordinates $(\xi_0, \xi_1, \ldots, \xi_r, x_1, \ldots, x_n)$ and extend the
Poisson structure  $\bpi$ on $\CC^{r+n}$ to a Poisson structure $\pi$ on $\CC^{1+r+n}$ such that
\[
\{\xi_0,\, \xi_k\} = 0 \;\;\; \mbox{and} \;\;\; \{\xi_0, \, x_j\} = \la \omega_{\beta_1} + \cdots + \omega_{\beta_r}, \,
s_1s_2\cdots s_{j-1}(\al_j)\ra \, \xi_0 x_j, \;\;\; j = 1, \ldots, n,
\]
and such that the projection $(\CC^{1+r+n}, \pi) \to (\CC^{r+n}, \bpi)$ to the last $(r+n)$-coordinates
is Poisson. It then follows from the definition that the function $\xi_0\xi_1\cdots \xi_r$ on $\CC^{1+r+n}$ is
a Casimir function. Consider now the affine Poisson subvariety $X$ of $(\CC^{1+r+n}, \pi)$ given by
\[
X = \{(\xi_0, \xi_1, \ldots, \xi_r, x_1, \ldots, x_n): \; \xi_0\xi_1\cdots \xi_r = 1\}.
\]
The embedding
$\Phi: (T \times \O^\bfu, \; \bpi) \to (\CC^{1+r+n}, \; \pi)$
given by
\[
(t, q) \longmapsto ((\xi_1(t) \cdots \xi_r(t))^{-1}, \, \xi_1(t), \, \ldots, \, \xi_r(t), \, (\phi^\bfu)^{-1}(q)),
\hs (t, q) \in T \times \O^\bfu,
\]
is then Poisson with $X$ as its image. Let $\tilde{y} = t^\lam y^{\lam'}_{[i, j]} \in \tcYGu$ as in
Definition \ref{nota-class-tY}. Since for each $k = 1 \ldots, r$,
$\xi_k^{-1} = \Phi^*(\xi_0\xi_1 \cdots \xi_{k-1}\xi_{k+1}\cdots \xi_r)$ as regular functions on $T \times \O^\bfu$, there exists a monomial $\mu \in \CC[\xi_0, \xi_1, \ldots, \xi_r]$
such that $\Phi^*\left(\mu  y^{\lam'}_{[i, j]}\right) = \tilde{y}$. Consider now the polynomial
$y = \mu  y^{\lam'}_{[i, j]} \in \CC[\xi_0, \xi_1, \ldots, \xi_r, x_1, \ldots, x_n]$.
By considering a re-ordering $(\tilde{x}_1, \ldots, \tilde{x}_n)$ of the coordinates
$(x_1, \ldots, x_n)$ as in the proof of Theorem \ref{thm-Y-calS}, and consider the coordinates
$(\xi_0, \xi_1, \ldots, \xi_r, \tilde{x}_1, \ldots, \tilde{x}_n)$ on $\CC^{1+r+n}$, one proves
as in Theorem \ref{thm-Y-calS} that $y$ has complete
Hamiltonian flow in $\CC^{1+r+n}$ with property $\calQ$. By Lemma \ref{le-simple-2},
$\tilde{y} = \Phi^*(y)$ has complete Hamiltonian flow in $T \times \O^\bfu$ with property $\calQ$.
Note also that
$\Phi(T \times \O^\bfu_e) = \{x \in X: g(x) \neq 0\}$, where
\[
g(\xi_0, \xi_1, \ldots, \xi_r, x_1, \ldots, x_n) = \!\!\!\prod_{\alpha \in \{\al_1, \ldots, \al_n\}}\!\!\!
y^{\omega_\al}(\phi^\bfu(x_1, \ldots, x_n)) \in \CC[\xi_0, \xi_1, \ldots, \xi_r, x_1, \ldots, x_n].
\]
By Corollary \ref{co-y-f}, $\{y,g\}$ is log-canonical. By Lemma \ref{le-simple-2}, one sees that
$\tilde{y}|_{T \times \O^\bfu_e}$ has complete Hamiltonian flows in $T \times \O^\bfu_e$ with property $\calQ$.
\end{proof}

It follows from Lemma \ref{le-y-y} that for $1 \leq i' \leq i \leq n+1$, $i-1 \leq j \leq j' \leq n$, and
any $\lam_1, \lam_2, \lam_1^\prime, \lam_2^\prime \in \calP^+$,
the Poisson bracket between $t^{\lam_1}y_{[i,j]}^{\lam_2}$ and $t^{\lam^\prime_1}y_{[i',j']}^{\lam^\prime_2}$
is log-canonical. We single out the following cases, which will appear in our
application to double Bruhat cells in $\S$\ref{sec-KZ}, and
for which the Poisson brackets are specially simple.

\begin{notation}\label{nota-y-special}
{\rm
For $1 \leq i \leq n +1$, $i-1 \leq j \leq n$,  and $\lam \in \calP^+$, let
\[
\tilde{y}_{[i,j]}^\lam = t^{-s_{[1, i-1]}(\lam)} y^\lam_{[i, j]} \in \tcYGu,
\]
where $y^\lam_{[i, j]} = 1$ if $j = i-1$.
Note then that with respect to the diagonal $T$-action on $T \times \O^\bfu$,
$\tilde{y}_{[i,j]}^\lam$ is $T$-homogeneous with weight $-s_{[1, j]}(\lam)$ when $i \leq j$ (see Notation \ref{nota-s-interval}).
\hfill $\diamond$
}
\end{notation}

\begin{lemma}\label{le-y-y-special}
For any $\lam, \lam' \in \calP^+$ and $1 \leq i' \leq i \leq n+1$,  $i-1\leq j \leq j' \leq n$, and with respect to the Poisson structure
$0 \bowtie \pi_n$ on $T \times \O^\bfu$, one has
\begin{equation}\label{eq-bracket-y-y}
\left\{\tilde{y}^\lam_{[i, j]}, \, \; \tilde{y}^{\lam'}_{[i', j']}\right\}
= (\la s_{[1, i-1]}(\lam), \; s_{[1, i'-1]}(\lam')\ra - \la s_{[1, j]}(\lam), \; s_{[1, j']}(\lam')\ra)\;
\tilde{y}^\lam_{[i, j]}\,\tilde{y}^{\lam'}_{[i', j']}.
\end{equation}
\end{lemma}

\begin{proof} If $j' = i'-1$, then $j = i-1$, and both sides of \eqref{eq-bracket-y-y} are zero.
Assume that $j' \geq i'$. Set $a = s_{[1, i-1]}(\lam), b = s_{[1,j]}(\lam),
a'=s_{[1, i'-1]}(\lam')$, and $b' = s_{[1,j']}(\lam')$. It follows from the definition of $0 \bowtie \pi_n$
and Lemma \ref{le-y-y} that
\[
\left\{\tilde{y}^\lam_{[i, j]}, \, \; \tilde{y}^{\lam'}_{[i', j']}\right\} =
c \;\tilde{y}^\lam_{[i, j]}\, \tilde{y}^{\lam'}_{[i', j']},
\]
with $c = -\la a', \, a-b\ra + \la a-b, \, a'+b'\ra + \la a, \, a'-b'\ra =\la a, \, a'\ra -\la b, \, b'\ra$.
\end{proof}


\subsection{Integrable systems on $\O^{(\bfu^{-1}, \bfu)}$}\label{subsec-int-O}
Let again $\bfu = (s_{\alpha_1}, \ldots, s_{\al_n}) =(s_1, \ldots, s_n)$ be any sequence of simple reflections in $W$.
Consider the sequence
\[
(\bfu^{-1}, \bfu) = (s_n, \; \ldots, \; s_2, \; s_1, \; s_1, \; s_2, \; \ldots, s_n)
\]
and the Poisson manifold $(\O^{(\bfu^{-1}, \bfu)}, \pi_{2n})$. Recall
the parametrization $\phi^{(\bfu^{-1}, \bfu)}: \CC^{2n} \to \O^{(\bfu^{-1}, \bfu)}$ given by
\[
\phi^{(\bfu^{-1}, \bfu)}(x_1, \ldots, x_{2n}) = [p_{\al_n}(x_1), \; \ldots,
\; p_{\al_1}(x_n), \; p_{\al_1}(x_{n+1}), \; \ldots, \;
p_{\al_n}(x_{2n})].
\]
For $k = 1, \ldots, n$, define the regular functions $y_k \in \calY^{(\bfu^{-1}, \bfu)}$ on
$\O^{(\bfu^{-1}, \bfu)}$ (see notation in Definition \ref{de-class-Y})
\begin{equation}\label{eq-yk}
y_k = \Delta^{\omega_{\al_k}} (p_{\al_k}(x_{n+1-k})\, \cdots \,
 p_{\al_1}(x_n)\, p_{\al_1}(x_{n+1}) \cdots p_{\al_k}(x_{n+k})) = y^{\omega_{\al_k}}_{[n+1-k, \, n+k]}.
\end{equation}

\begin{theorem}\label{thm-integrable-O}
The functions $y_1, \ldots, y_n$ form an integrable system on
$(\O^{(\bfu^{-1}, \bfu)}, \pi_{2n})$ that has complete
Hamiltonian flows with property $\calQ$.
\end{theorem}

\begin{proof}
By Lemma \ref{le-y-y}, $\{y_k, y_{k'}\} = 0$ for $1\leq k < k' \leq n$.
By \cite[Theorem 1.1]{LM:flags}, the Poisson structure $\pi_{2n}$ is non-degenerate everywhere
in $\O^{(\bfu^{-1}, \bfu)}_e$. It follows from Lemma \ref{le-y-xx} that the functions $y_1, \ldots, y_n$ are
functionally
independent on $\O^{(\bfu^{-1}, \bfu)}$. By Theorem \ref{thm-tcYGu}, $y_1, \ldots, y_n$ form an
integrable system on $(\O^{(\bfu^{-1}, \bfu)}, \pi_{2n})$ that has complete Hamiltonian flows with
property $\calQ$.
\end{proof}

\begin{remark}\label{rk-integrable-T-O}
{\rm
Introduce the regular function $\tilde{y}_k$ on $T \times \O^{(\bfu^{-1}, \bfu)}$
by (see Notation \ref{nota-y-special})
\begin{align}\label{eq-tulde-y-k}
\tilde{y}_k & = t^{-s_n s_{n-1}\cdots s_{k}(\omega_{\al_k})}
\Delta^{\omega_{\al_k}} (p_{\al_{k-1}}(x_{n+2-k})\, \cdots \,
 p_{\al_1}(x_n)\, p_{\al_1}(x_{n+1}) \cdots p_{\al_{k-1}}(x_{n+k-1}))\\
\nonumber
&=\tilde{y}_{[n-k+2, \, n+k-1]}^{\,\omega_{\al_k}}, \hs k = 1, \ldots, n.
\end{align}
As we will see in Example \ref{ex-KZ-system},  when $\bfu$ is a reduced, the
functions $\tilde{y}_1, \ldots, \tilde{y}_n$ form an integrable system on $T \times \O^{(\bfu^{-1}, \bfu)}_e$,
which under the Fomin-Zelevinsky isomorphism (see Definition \ref{de-FZ-map}), is precisely the
Kogan-Zelevinsky integrable system on the double Bruhat cell $G^{u, u}$ defined by $\bfu$, where $u=s_1s_2 \cdots s_n$.
For an arbitrary $\bfu = (s_1, \ldots, s_n)$, not necessarily reduced, we know by
Proposition \ref{pr-T-O-u} that $T \times \O^{(\bfu^{-1}, \bfu)}_e$ is a single $T$-leaf
of the Poisson structure $0 \bowtie \pi_{2n}$ (for the diagonal $T$-action), all the symplectic leaves
have dimension $2n$, and by Lemma \ref{le-y-y-special},
the functions $\tilde{y}_1, \ldots, \tilde{y}_n$ pairwise Poisson commute.
However, we are unable to show that the functions $\tilde{y}_1, \, \ldots, \, \tilde{y}_{n}$
are functionally independent  when restricted to a symplectic leaf of $0 \bowtie \pi_{2n}$, as we need
to better understand the symplectic leaves in $T \times \O^{(\bfu^{-1}, \bfu)}_e$. We will
leave this to future research.
\hfill $\diamond$
}
\end{remark}

\begin{remark}\label{rk-Fn-Fn-2}
{\rm
For any arbitrary sequence $\bfu = (u_1, \ldots, u_n)$ in $W$, by choosing a reduced word for each $u_j$ and
by identifying $\O^\bfu$ with a $\O^{\tilde{\bfu}}$ of Bott-Samelson type as in Remark \ref{rk-Fn-Fn}, one
then has the collections of regular functions $\tilde{\calY}$ on $\O^\bfu$ and $\tcYGu$ on
$T \times \O^\bfu$, and Theorem \ref{thm-Y-calS} and Theorem \ref{thm-tcYGu} apply. Moreover, one has the
integrable system on $\O^{(\bfu^{-1}, \bfu)}$ defined in $\S$\ref{thm-integrable-O}, where
$\bfu^{-1} = (u_n^{-1}, \ldots, u_1^{-1})$, which has complete Hamiltonian flows with property $\calQ$.
\hfill $\diamond$
}
\end{remark}

\begin{example}\label{ex-ss-0}
{\rm
Consider $G = SL(2, \CC)$ and $\bfu_n = (s, s, \ldots, s)$ of length $n$, where $s = s_\alpha$, and $\alpha$
is the only simple root of $SL(2, \CC)$. Take the bilinear form $\lara_\g$ on $\g=\sl(2, \CC)$ such that
$\la \al, \al \ra = 1$. Then the Bott-Samelson Poisson structure associated to $(SL(2, \CC), \bfu_n)$
on $\CC^n$ is explicitly given by (see \cite[$\S$6.2]{EL:BS})
\begin{align*}
&\{x_i, \; x_{i+1}\} = x_ix_{i+1}-1,\hs 1 \leq i \leq n-1,\\
&\{x_i, \; x_j\} = (-1)^{j-i+1} x_ix_j, \hs 1 \leq i < j \leq n, \; j-i \geq 2.
\end{align*}
Let $\omega$ be the only fundamental weight.
To see what the polynomials $y_{[i, j]}^{\omega} \in \CC[x_i, \ldots, x_j]$
are for $1 \leq i \leq j \leq n$, for any integer $k \geq 1$, introduce $E_k \in \ZZ[x_1, \ldots, x_k]$ via the
identity
\begin{equation}\label{eq-def-Ek}
\left(\begin{array}{cc}x_1 & -1 \\ 1 & 0\end{array}\right) \left(\begin{array}{cc}x_2 & -1 \\ 1 & 0\end{array}\right)
\cdots
\left(\begin{array}{cc}x_k & -1 \\ 1 & 0\end{array}\right) =
\left(\begin{array}{cc}E_k (x_1, \ldots, x_k)& \ast \\ \ast & \ast\end{array}\right) \in SL(2, \CC).
\end{equation}
The polynomials $E_k$ are, with a variation of signs,
the well-known {\it Euler continuants}
\cite{Knuth:Art}, and are related to continued fractions via
\[
\frac{E_j(x_1,...,x_j)}{E_{j-1}(x_2,...,x_{j})} = x_1- \frac{1}{x_{2} -
\displaystyle \frac{1}{\cdots - \displaystyle \frac{1}{x_{j-1} - \displaystyle \frac{1}{x_j}}}},
\]
and can be recursively defined by $E_0 = 1$, $E_1(x_1) = x_1$, and
\[
E_j(x_1, \ldots, x_j) = x_jE_{j-1}(x_1, \ldots, x_{j-1})-E_{j-2}(x_1, \ldots, x_{j-2}), \hs j \geq 2.
\]
In our example, as $\displaystyle q_\al(c) = \left(\begin{array}{cc} c & -1 \\ 1 & 0\end{array}\right)$
for $c \in \CC$, it follows from the definition that
\[
y_{[i, j]}^{\omega} = E_{j-i+1}(x_i, x_{i+1}, \ldots, x_j) \in \ZZ[x_i, x_{i+1}, \ldots, x_j], \hs 1 \leq i \leq j \leq n.
\]
As $(\bfu_n^{-1}, \bfu_n) = \bfu_{2n}$, the polynomials $y_1, \ldots, y_n$ forming
a complete integrable system on $\CC^{2n}$ with property $\calQ$, as stated in Theorem \ref{thm-integrable-O}, are given by
\[
y_k = E_{2k}(x_{n-k+1}, \, x_{n-k+2}, \, \ldots, \, x_{n+k}), \hs k = 1, 2, \ldots, n.
\]
For example, for $n = 2$, we have $y_1 = x_2x_3-1$ and $y_2 = x_1x_2x_3x_4-x_1x_2-x_1x_4-x_3x_4+
1$.
For $p=(x_1, x_2, x_3, x_4) \in \CC^4$, the integral curve $\gamma_i: \CC \to \CC^4$ of
the Hamiltonian vector field of $y_i$ through $p$ is given by
\begin{align*}
\gamma_1(c)& = \begin{cases} (x_1+cx_3, \; x_2, \; x_3, \; x_4-cx_2), & \;\; y_1 =0,\\
\left(x_1+x_3(e^{cy_1}-1)/y_1, \;  e^{-cy_1}x_2, \; e^{cy_1}x_3, \; x_4 + x_2(e^{-cy_1}-1)/y_1\right),
& \;\; y_1 \neq 0,\end{cases}\\
\gamma_2(c) & = (x_1e^{-cy_2}, \; x_2e^{cy_2}, \; x_3e^{-cy_2}, \; x_4e^{cy_2}),
\end{align*}
where for $i = 1, 2$, $y_i$ also denotes the value $y_i(p)$ of the function $y_i$ at $p$.
\hfill $\diamond$
}
\end{example}

\section{Complete Hamiltonian flows of Generalized Minors and the Kogan-Zelevinsky integrable systems}\label{sec-KZ}
In $\S$\ref{sec-KZ}, we introduce Bott-Samelson coordinates on double Bruhat cells via
the Fomin-Zelevinsky embeddings, and we
apply results from $\S$\ref{sec-GBC} to establish the completeness of Hamiltonian flows of
all Fomin-Zelevinsky minors and Kogan-Zelevinsky integrable systems on double Bruhat cells.

\subsection{Fomin-Zelevinsky embeddings and Bott-Samelson coordinates on double Bruhat cells}\label{subsec-FZ-map}
For $w \in W$ and a representative $\dw \in N_G(T)$ of $w$, let
\[
C_\dw = (N \dw)\cap(\dw N_-) \hspace{.2in} \mbox{and} \hs
D_\dw = (N_- \dw) \cap (\dw N).
\]
The multiplication map of $G$ then gives diffeomorphisms
\[
C_\dw \times B \longrightarrow BwB, \hs
B_- \times C_\dw \longrightarrow B_-wB_-,  \hs \mbox{and} \hs
D_\dw \times B_- \longrightarrow B_-wB_-.
\]
With the representatives $\overline{w}$ and $\overline{\overline{w}}$ in $N_G(T)$
from $\S$\ref{subsec-nota-intro}
for $w \in W$, one then has
\[
C_{\overline{\overline{w}}^{\, -1}} = C_{\overline{w^{-1}}} = \{g^{-1}: \; g \in D_{\overline{\overline{w}}}\}.
\]

\begin{notation}\label{nota-decompo-g}
{\rm
For $u, v \in W$, we will write each  $g \in G^{u, v}$ uniquely  as
\begin{equation}\label{eq-g-decomp}
g =n\bu n' t' = m \bbv m' t,
\end{equation}
where $n\bu \in C_\bu, \; m\bbv \in D_{\bbv},\;
n' \in N, \;m' \in N_-, \;t, t' \in T$. With $g \in \Guv$ thus written,
\begin{align}\label{eq-g-decomp-0}
&\bu^{-1}g, \; g \bbv^{\,-1},\; (m \bbv)^{-1} n\bu \in N_-TN, \hs
[\bu^{-1}g]_0 =t', \hs [g \bbv^{\,-1}]_0^v = t, \\
\label{eq-g-decomp-2}
&[(m \bbv)^{\,-1} n\bu]_- = m', \hs [(m \bbv)^{\,-1} n\bu]_0 = t(t')^{-1},\hs [(m \bbv)^{\,-1} n\bu]_+ = (n')^{-1}.
\end{align}
\hfill $\diamond$
}
\end{notation}

Fix $u, v \in W$, and recall from $\S$\ref{subsec-gBC} the generalized Bruhat cell
\[
\O^{(v^{-1}, u)} = Bv^{-1}B\times_B BuB/B\subset F_2
\]
and its  Zariski open subset $O^{(v^{-1}, u)}_e = \{[g_1, g_2] \in \O^{(v^{-1}, u)}: \; g_1g_2 \in B_-B\}$.
Note that
\[
C_{\overline{v^{-1}}} \times C_\bu \lrw \O^{(v^{-1}, u)}, \;\; (c_1, c_2) \longmapsto [c_1, c_2],
\]
is an isomorphism. The following is a modification of \cite[Proposition 3.1]{FZ:total}.

\begin{proposition}\label{pr-Guv-Cuv}
For any $u, v \in W$, the map
\begin{equation}\label{eq-de-Fuv}
F^{u, v}: \;\; \Guv \longrightarrow  T \times \O^{(v^{-1}, u)}_e, \;\;\;
g\longmapsto (t, \; [(m \bbv)^{-1}, \; n\bu]),
\end{equation}
where $g \in \Guv$ is as in \eqref{eq-g-decomp}, is a biregular isomorphism, with its inverse given by
\begin{align}\label{eq-Uuv-Guv}
&T \times \O^{(v^{-1}, u)}_e \longrightarrow \Guv: \;\;
(t,\,[(m \bbv)^{-1}, \, n\bu]) \longmapsto \;
n\bu\,[(m \bbv)^{\,-1} n\bu]_+^{-1}[(m \bbv)^{\,-1} n\bu]_0^{-1}t \\
\nonumber &\hspace{3in} \;= m \bbv \,[(m \bbv)^{\,-1} n\bu]_-t,
\end{align}
where $m \bbv \in D_{\bbv}$, and $n\bu \in C_{\bu}$.
\end{proposition}

\begin{proof} Consider the Zariski open subset $\calU^{u, v}$ of
\[
(C_\bu \,\bu^{-1}) \times T \times (\bbv C_{\bbv^{\, -1}}) = (N \cap \bu N_- \bu^{\, -1}) \times T \times
(N_- \cap \bbv N \bbv^{\, -1})
\]
consisting of all $(n, t', m) \in (C_\bu \,\bu^{-1}) \times T \times (\bbv C_{\bbv^{\, -1}})$ such that
$m^{-1}n \in \bbv \,B_-B \bu^{-1}$.
Proposition 3.1 of \cite{FZ:total} says that, with $g \in \Guv$ as in \eqref{eq-g-decomp},
\[
\Guv \longrightarrow \calU^{u, v}, \;\; g \longmapsto
(n, \, t',\, m^{-1}), \hs g \in \Guv,
\]
is a biregular isomorphism from $\Guv$ to $\calU^{u, v}$.  It follows that the map
\[
\Guv \lrw T \times \O^{(v^{-1}, u)}_e, \;\; g \longmapsto (t', \; [(m \bbv)^{\, -1}, \; n\bu]), \hs g \in \Guv,
\]
is a biregular isomorphism.
As $t = t'[(m \bbv)^{-1} n \bu]_0$, it follows that $F^{u, v}$
is a biregular isomorphism, and the formula for its inverse
follows from \eqref{eq-g-decomp-2}.
\end{proof}

\begin{remark}\label{rk-F-equi}
{\rm
In the notation of Proposition \ref{pr-Guv-Cuv}, the isomorphism $F^{u, v}$ is $T$-equivariant
if $T$ acts on
$\Guv$ by left translation and on $T \times \O^{(v^{-1}, u)}_e$ by
\[
t_1 \cdot (t, \,q) = (t_1^vt,\, \,t_1^vq), \hs t_1, \,t \in T, \; q \in \O^{(v^{-1}, u)}_e.
\]
If $T$ acts on $G^{u, v}$ by right translation, then $F^{u, v}$ is $T$-equivariant if
$T$ acts on $T \times \O^{(v^{-1}, u)}_e$ by $(t, q) \cdot t_1 = (tt_1, q)$ for $t_1, t \in T$ and $q \in \O^{(v^{-1}, u)}_e$.
\hfill $\diamond$
}
\end{remark}
Recall from \cite{BZ} that
for $u, v \in W$, the quotient $\Guv/T \subset G/T$  of the double Bruhat cell $\Guv$ is called a {\it reduced double Bruhat cell}.

\begin{definition}\label{de-FZ-map}
{\rm
For $u, v \in W$, the biregular isomorphisms
\begin{align}\label{eq-FZ-iso}
\hat{F}^{u, v}:& \;\; \Guv/T \lrw \O^{(v^{-1}, u)}_e, \;\;\; g_\cdot T \longmapsto [(m \bbv)^{-1}, \, n\bu],\\
\label{eq-FZ-2}
F^{u, v}: & \;\; \Guv \lrw T \times \O^{(v^{-1}, u)}_e, \;\;\; g \longmapsto (t, \,[(m \bbv)^{-1}, \; n\bu]),
\end{align}
where, again, $g \in \Guv$ is decomposed as  in \eqref{eq-g-decomp}, will be called, respectively, the
{\it Fomin-Zelevinsky isomorphisms} for $\Guv/T$ and for $\Guv$.
The corresponding open embeddings via
 the embedding $\O^{(v^{-1}, u)}_e \hookrightarrow \O^{(v^{-1}, u)}$, are also denoted by
\begin{equation}\label{eq-FZ-embed}
\hat{F}^{u, v}:\;\; \Guv/T \lrw \O^{(v^{-1}, u)} \hs \mbox{and} \hs
F^{u, v}: \;\; \Guv \lrw T \times \O^{(v^{-1}, u)},
\end{equation}
and will be respectively called the
{\it Fomin-Zelevinsky embeddings} of $\Guv/T$ and $\Guv$.
\hfill $\diamond$}
\end{definition}

Let again $u, v \in W$, and let $a = l(v)$ and $b = l(u)$. Fix reduced words
\begin{equation}\label{eq-v-u-reduced}
\bfv =(s_a, \; s_{a-1}, \; \cdots, \;  s_1) \hs \mbox{and} \hs
\bfu =(s_{a+1},\;  s_{a+2}, \;  \cdots, \;  s_{a+b}),
\end{equation}
of $v$ and $u$, and consider the generalized Bruhat cell $\Ovmu \subset F_{a+b}$, where
\begin{equation}\label{eq-sequence-v-u}
(\bfv^{-1}, \, \bfu) = (s_1, \, s_2, \,\ldots, \, s_a,\,s_{a+1}, \, s_{a+2}, \, \ldots, \, s_{a+b}).
\end{equation}
Consider the isomorphism $\mu: \Ovmu\to \O^{(v^{-1}, u)}$ given by
\begin{equation}\label{eq-mu-2}
\mu([g_1, \,\ldots,\,g_a,\, g_{a+1}, \,\ldots,\, g_{a+b}])=
[g_1g_2 \cdots g_a, \; g_{a+1}g_{a+2} \cdots g_{a+b}],
\end{equation}
where $g_j \in Bs_jB$ for $j = 1, \ldots, a+b$, and define the open embeddings
\begin{align}\label{eq-bf-Fuv-01}
\hat{F}^{\bfu, \bfv}&\stackrel{{\rm def}}{=}\; \mu^{-1} \circ \hat{F}^{u, v}:  \;\; \Guv/T \lrw \Ovmu,\\
\label{eq-bf-Fuv-02}
{F}^{\bfu, \bfv} &\stackrel{{\rm def}}{=}\; \mu^{-1} \circ F^{u, v}: \;\; \Guv \lrw T \times \Ovmu.
\end{align}
Note that the images of $\hat{F}^{\bfu, \bfv}$ and ${F}^{\bfu, \bfv}$ are respectively
$\Ovmu_e$ and $T \times \Ovmu_e$.
Recall also the coordinates $(x_1, \ldots, x_{a+b})$
on $\Ovmu$ given by $\phi^{(\bfv^{-1}, \bfu)}: \CC^{a+b} \to \Ovmu$.

\begin{definition}\label{de-FZ-coordinates}
{\rm
We call $\hat{F}^{\bfu, \bfv}$ (resp. $F^{\bfu, \bfv}$)
the Fomin-Zelevinsky embedding of $\Guv/T$ (resp. of $\Guv$) associated to the
reduced words $\bfu$ and $\bfv$  of $u$ and $v$ in \eqref{eq-v-u-reduced}, and we call
$(x_1, \ldots, x_{a+b})$ (resp. $(t, x_1, \ldots, x_{a+b})$, where $t \in T$,)
the {\it Bott-Samelson coordinates} on $\Guv/T$ (resp. on $\Guv$) associated to the
reduced words $\bfu$ and $\bfv$ of $u$ and $v$ in \eqref{eq-v-u-reduced}.
\hfill $\diamond$
}
\end{definition}

\subsection{Fomin-Zelevinsky minors on $\Guv$ in Bott-Samelson coordinates}\label{subsec-minors}
Let again $u, v \in W$. Recall from \cite{FZ:total} the Fomin-Zelevinsky
twist map
\begin{equation}\label{eq-FZ-twist}
\Guv \longrightarrow G^{u^{-1}, v^{-1}}, \;\;g \longmapsto g' \;\stackrel{{\rm def}}{=} \;
\left([\bu^{-1}g]_-^{-1} \bu^{-1} \, g \bbv^{\,-1} [g\bbv^{\, -1}]_+^{-1}\right)^\theta,
\end{equation}
where $\theta$ is the involutive automorphism  of $G$ satisfying
\[
\theta(t) = t^{-1}, \hs \theta(u_\alpha(c)) = u_{-\al}(c), \hs t \in T, \; \al \in \Gamma, \;  c \in \CC.
\]
Recall from $\S$\ref{subsec-nota-intro} the regular functions $\Delta_{w_1\lam, w_2\lam}$ on $G$,
where $w_1, w_2 \in W$ and $\lam \in \calP^+$.
By \cite{FZ:total}, a {\it twisted generalized minor} on $\Guv$ is a regular
function of the form
\[
g \longmapsto \Delta_{w_1 \omega_\al, \, w_2\omega_\al}(g'), \hs g \in \Guv,
\]
where $w_1, w_2 \in W$ and $\al \in \Gamma$.

\begin{definition}\label{de-FZ-minors}
{\rm
By a {\it Fomin-Zelevinsky twisted generalized minor}, or simply a
{\it Fomin-Zelevinsky minor}, on $\Guv$, we mean a twisted generalized minor on $\Guv$ of the form
\[
g \longmapsto \Delta_{u_2^{-1} \omega_\al, \; v_1\omega_\al}(g'), \hs g \in \Guv,
\]
where $\al \in \Gamma$, and $u_2, v_1 \in W$ are such that $l(u) = l(u_1)+ l(u_2)$, and $l(v) = l(v_1) + l(v_2)$
with $u_1, v_2 \in W$ determined by
$u = u_1 u_2$, $v = v_1 v_2$.
\hfill
$\diamond$
}
\end{definition}

Fixing again the reduced words $\bfu$ and $\bfv$ of $u$ and $v$ as in \eqref{eq-v-u-reduced},
we now express the Fomin-Zelevinsky minors in the corresponding Bott-Samelson coordinates on
$\Guv$.

\begin{notation}\label{nota-Mij}
{\rm
For $i = 0, 1, \ldots, a$, $j = 1, 2, \ldots, b+1$, and $\lam \in \calP^+$, let
\begin{align}\label{eq-v-intervals}
&v_{[1, i]}=\begin{cases} e, & \;\; i = 0,\\ s_{a}s_{{a-1}} \cdots s_{{a+1-i}},& \;\; i = 1, 2, \ldots, a\end{cases},
\hs
v_{[i+1, \, a]} = (v_{[1, i]})^{-1}v,\\
\label{eq-u-intervals}
&u_{[j, \, b]} = \begin{cases}  s_{{a+j}}s_{{a+j+1}} \cdots s_{{a+b}}, & \;\;j = 1, 2, \ldots, b,\\
e, & \;\; j = b+1, \end{cases}
\end{align}
and let $M_{i, j}^\lam$ be the regular function  on $\Guv$ given by
\begin{equation}\label{eq-Mij}
M_{i, j}^\lam(g) \; \stackrel{{\rm def}}{=}\; \Delta_{(u_{[j, \, b]})^{-1}\lam, \; v_{[1, i]}\lam}(g'), \hs g \in \Guv.
\end{equation}
Note in particular that
\[
M_{0, 1}^\lam(g) = \Delta_{u^{-1}\lam, \lam}(g'), \hs
M_{a, b+1}^\lam(g) = \Delta_{\lam, v\lam}(g'), \hs g \in \Guv.
\]
By \eqref{eq-Delta-Delta} in $\S$\ref{subsec-nota-intro}, each
$M_{i,j}^\lam$ is a monomial,
with non-negative powers, of Fomin-Zelevinsky minors on $G^{u, v}$.
Recall also that for $\lam \in \calP^+$, $\Delta^\lam(g) = \Delta_{\lam, \lam}(g)$ for $g \in G$.
\hfill $\diamond$
}
\end{notation}

\begin{lemma}\label{lem-Delta-g}
Let $0\leq i \leq a, 1 \leq j \leq b+1$, and $\lam \in \calP^+$. For
$g \in \Guv$ as in \eqref{eq-g-decomp}, one has
\[
M_{i, j}^\lam(g) = t^{-(v_{[i+1, \, a]})^{-1}(\lam)}
\Delta^\lam
\left(\overline{\overline{v_{[i+1, \, a]}}} \;(m \bbv)^{-1} n \bu \; \overline{u_{[j, \, b]}}^{\, -1}\right).
\]
\end{lemma}

\begin{proof}
Introduce (another twist on $\Guv$)
\begin{equation}\label{eq-g-hat}
\Guv \longrightarrow  G^{v, u},\;\;\; g \longmapsto \hat{g} \; \stackrel{{\rm def}}{=} \;((g')^{-1})^\theta =
\left([\bu^{-1}g]_-^{-1} \bu^{-1} \, g \bbv^{\,-1} [g\bbv^{\, -1}]_+^{-1}\right)^{-1}.
\end{equation}
By \cite[(2.15)]{FZ:total}, for every $\alpha \in \Gamma$, one has
\[
\Delta^{\omega_\al}(g) = \Delta^{\omega_\al}((g^{-1})^\theta), \hs \forall \; g \in G.
\]
By \cite[Proposition 2.1]{FZ:total}, for any $w_1, w_2 \in W$ and $\alpha \in \Gamma$, one has
\begin{equation}\label{eq-g-g-1}
\Delta_{w_1 \omega_\al, \, w_2\omega_\al}(g') = \Delta^{\omega_\al}((\overline{w_2}^{\, -1})^\theta \hat{g} \,
\overline{w_1}^\theta) = \Delta^{\omega_\al}(\overline{\overline{w_2}}^{\, -1} \hat{g}\, \overline{\overline{w_1}}),
\hs g \in \Guv.
\end{equation}
Let $g \in \Guv$ be as in \eqref{eq-g-decomp} and further write
\begin{equation}\label{eq-m-m}
m' = m_1m_2, \hs \mbox{where}\;\; m_1 \in N_- \cap (\bbv^{\, -1} N_- \bbv), \;\;
m_2 \in N_- \cap (\bbv^{\, -1} N \bbv).
\end{equation}
It is then easy to see that
\[
\hat{g} = (t^{-1})^{v^{-1}} \bbv\, m_1^{-1} (m \bbv)^{-1} n\bu, \hs g \in \Guv.
\]
Using \eqref{eq-g-g-1}, one has,
\begin{align*}
M_{i, j}^\lam(g)
& = \Delta^\lam\left(\overline{\overline{v_{[1, i]}}}^{\, -1} \; \hat{g} \;
\overline{\overline{(u_{[j, \, b]})^{-1}}}\right)
= t^{-(v_{[i+1, \, a]})^{-1}(\lam)}
\Delta^\lam\left(\overline{\overline{v_{[i+1, \, a]}}} m_1^{-1}\; (m \bbv)^{-1} n\bu\;
\overline{u_{[j, \, b]}}^{\, -1}\right).
\end{align*}
As  $l(v) = l(v_{[1, i]})+ l(v_{[i+1, \, a]})$ and $m_1^{-1} \in N_-
\cap (\bbv^{\, -1} N_- \bbv)$,
$\overline{\overline{v_{[i+1, \, a]}}} m_1^{-1} \overline{\overline{v_{[i+1, \, a]}}}^{\, -1} \in N_-$ by
\cite[VI, 1.6]{Bourbaki},
from which Lemma \ref{lem-Delta-g} follows.
\end{proof}

To express the functions $M_{i, j}^\lam$ on $\Guv$ in the Bott-Samelson coordinates
$(t, x_1, \ldots, x_{a+b})$
defined by the given
reduced words $\bfu$ and $\bfv$ of $u$ and $v$, set
for $1 \leq k \leq l \leq a+b$,
\begin{equation}\label{eq-g-lk}
p_{[k, l]}(x) = p_{\al_k}(x_k) p_{\al_{k+1}}(x_{k+1}) \cdots p_{\al_l}(x_l) \in G,
\hs x =(x_1, \ldots, x_{a+b})\in \CC^{a+b}.
\end{equation}
Here recall that
$p_\al(c) = u_\al(x) \overline{{s}_\al}$ for $\al \in \Gamma$ and $c \in \CC$.
Define $p_{[k, l]}(x) = e$ if $k > l$.

\begin{proposition}\label{pr-M-ij-y}
Let $0\leq i \leq a, 1 \leq j \leq b+1$, and $\lam \in \calP^+$. In the Bott-Samelson coordinates
$(t, x_1, \ldots, x_{a+b})$ of $g \in \Guv$, one has
\[
M_{i, j}^\lam(g) \; \stackrel{{\rm def}}{=}\; \Delta_{(u_{[j, \, b]})^{-1}\lam, \; v_{[1, i]}\lam}(g')=
 t^{-s_1s_2 \cdots s_{a-i}(\lam)} \Delta^\lam(p_{[a+1-i, \, a-1+j]}(x)).
\]
Consequently, with the open Fomin-Zelevinsky embedding $F^{\bfu, \bfv}: G^{u, v} \to T \times \O^{(\bfv^{-1}, \bfu)}$
in \eqref{eq-bf-Fuv-02} and in the notation of Notation \ref{nota-y-special}, one has
\begin{equation}\label{eq-Mij-y}
M_{i, j}^\lam = (F^{\bfu, \bfv})^* \left(\tilde{y}^\lam_{[a+1-i, \, a-1+j]}\right),
\end{equation}
where $\tilde{y}^\lam_{[a+1-i, \, a-1+j]} = t^{-s_1s_2 \cdots s_{a-i}(\lam)} y^\lam_{[a+1-i, \, a-1+j]} \in \widetilde{\calY}^{(\bfv^{-1}, \bfu)}$.
\end{proposition}

\begin{proof} Let $g \in \Guv$ be as in \eqref{eq-g-decomp}.
Writing
\[
(m\bbv)^{-1}n \bu = p_{[1, \, a+b]} (x)= p_{[1, \, a-i]}(x)\,p_{[a-i+1, \, a+j-1]}(x)\,p_{[a+j, \, a+b]}(x),
\]
and noting that $\overline{\overline{v_{[i+1, \, a]}}}\;p_{[1, \, a-i]}(x) \in N_-$ and
$p_{[a+j, \, a+b]}(x) \;\overline{u_{[j, \, b]}}^{\, -1} \in N$,
it follows from Lemma \ref{lem-Delta-g} that
\begin{align*}
M_{i,j}^\lam(g) &= t^{-(v_{[i+1, \, a]})^{-1}(\lam)}
\Delta^\lam
\left(\overline{\overline{v_{[i+1, \, a]}}} \;(m \bbv)^{-1} n \bu \; \overline{u_{[j, \, b]}}^{\, -1}\right)\\
&=t^{-s_1s_2 \cdots s_{a-i}(\lam)} \Delta^\lam(p_{[a+1-i, \, a-1+j]}(x)).
\end{align*}
\end{proof}

\begin{example}\label{ex-KZ-system}
{\rm
Let $u \in W$ be arbitrary.
In \cite[$\S$2.5]{KZ:integrable}, Kogan and Zelevinsky introduced an integrable system
$M_1, M_3, \ldots, M_{2n-1}$ on
$(G^{u, u}, \pist)$ corresponding to each reduced word $\bfu = (s_1, s_2, \ldots, s_n)$ of $u$.
In terms of
our notation in $\S$\ref{subsec-minors},
\[
M_{2k-1}(g) = M_{k-1, k}^{\omega_{\al_k}}(g) = \Delta_{s_ns_{n-1}\cdots s_k\omega_{\al_k}, \,
s_1s_2 \cdots s_{k-1}\omega_{\al_k}}(g'), \hs g \in G^{u, u}, \; k = 1, \ldots, n.
\]
Using the open Fomin-Zelevinsky embedding $F^{{\bfu, \bfu}}: G^{u, u} \to T \times O^{(\bfu^{-1}, \bfu)}$, one thus has
\[
M_{2k-1} = (F^{\bfu, \bfu})^*\left(\tilde{y}_{[n+2-k, \, n+k-1]}^{\,\omega_{\al_k}}\right), \hs
k = 1, \ldots, n.
\]
A  concrete example for $G = SL(3, \CC)$ will be given in
Example \ref{ex-A2-121}.
\hfill $\diamond$
}
\end{example}

\subsection{The Fomin-Zelevinsky embeddings are Poisson}\label{subsec-FZ-Poi-1}
Turning to Poisson structures, note that
since the multiplicative Poisson structure $\pist$ on $G$ is invariant under (both left and right)
translations by elements in $T$,
\[
\hat{\pi}_{\rm st} \; \stackrel{{\rm def}}{=}\; \varpi(\pist)
\]
is a well-defined Poisson structure on $G/T$, where $\varpi: G \to G/T$ is the natural projection.
It is evident that each reduced double Bruhat cell $\Guv/T$ is a Poisson submanifold of $G/T$
with respect to $\hat{\pi}_{\rm st}$.
The restriction of $\hat{\pi}_{\rm st}$ to $\Guv/T$ will also be denoted as $\hat{\pi}_{\rm st}$.

\begin{theorem}\label{thm-FZ-Poi}
For any $u, v \in W$, the open Fomin-Zelevinsky embeddings
\begin{align*}
&\hat{F}^{u, v}: \;\;\; (\Guv/T,  \;\hat{\pi}_{\rm st}) \lrw \left(\O^{(v^{-1}, u)}, \; \pi_2\right),\\
&F^{u, v}: \;\; (\Guv, \; \pist) \lrw \left(T \times \O^{(v^{-1}, u)}, \; 0 \bowtie \pi_2\right)
\end{align*}
are Poisson, inducing biregular Poisson isomorphisms
\begin{align*}
&\hat{F}^{u, v}: \;\; (\Guv/T, \; \hat{\pi}_{\rm st}) \stackrel{\sim}{\lrw} \left(\O^{(v^{-1}, u)}_e,  \;\pi_2\right),\\
&F^{u, v}: \;\;(\Guv,  \;\pist) \stackrel{\sim}{\lrw} \left(T \times \O^{(v^{-1}, u)}_e, \; 0 \bowtie \pi_2\right).
\end{align*}
\end{theorem}

Given any reduced words $\bfu$ and $\bfv$ of $u$ and $v$ as in \eqref{eq-v-u-reduced}, recall
the isomorphism $\mu: \Ovmu \to \O^{(v^{-1}, u)}$ defined in \eqref{eq-mu-2}.
By Remark \ref{rk-Fn-Fn}, $\mu: (\Ovmu, \pi_{a+b}) \to (\O^{(v^{-1}, u)}, \pi_2)$ is Poisson.
Thus Theorem \ref{thm-FZ-Poi} is equivalent to the following Theorem \ref{thm-bf-Fuv-Poi}.

\begin{theorem}\label{thm-bf-Fuv-Poi} For any $u, v \in W$ and reduced words $\bfu$ and $\bfv$ of $u$ and $v$,
the open Fomin-Zelevinsky embeddings
\begin{align}\label{eq-bf-Fuv-1}
&\hat{F}^{\bfu, \bfv}:  \;\; (\Guv/T, \; \hat{\pi}_{\rm st}) \lrw \left(\Ovmu, \; \pi_{a+b}\right),\\
\label{eq-bf-Fuv-2}
&{F}^{\bfu, \bfv}: \;\; (\Guv, \; \pist) \lrw \left(T \times \Ovmu, \; 0 \bowtie\pi_{a+b}\right)
\end{align}
are Poisson, inducing biregular Poisson isomorphisms
\begin{align*}
&\hat{F}^{\bfu, \bfv}: \;\; (\Guv/T, \; \hat{\pi}_{\rm st}) \stackrel{\sim}{\lrw} \left(\O^{(\bfv^{-1}, \bfu)}_e,  \;
\pi_{a+b}\right),\\
&F^{\bfu, \bfv}: \;\;(\Guv,  \;\pist) \stackrel{\sim}{\lrw} \left(T \times \O^{(\bfv^{-1}, \bfu)}_e, \;
0 \bowtie \pi_{a+b}\right).
\end{align*}
\end{theorem}

We now give a proof of Theorem \ref{thm-bf-Fuv-Poi} using a result of Kogan and Zelevinsky in
\cite{KZ:integrable}  on the
Poisson brackets between certain Fomin-Zelevinsky minors on $\Guv$ and our expressions of the minors in
Bott-Samelson coordinates. A more conceptual proof of Theorem \ref{thm-FZ-Poi} without computations in coordinates
is given in the Appendix.

We first compute the Poisson brackets with respect to $0 \bowtie \pi_{a+b}$ of the functions
$\tilde{y}^\lam_{[a+1-i, \, a-1+j]}$ on $T \times
\O^{(\bfv^{-1}, \bfu)}$  appearing in \eqref{eq-Mij-y}.
\begin{lemma}\label{le-Y-Y}
For $0 \leq i \leq i' \leq a, \, 1 \leq j \leq j' \leq b+1$, and $\lam, \lam' \in \calP^+$,
\[
\left\{\tilde{y}^\lam_{[a+1-i, \, a-1+j]}, \;\; \tilde{y}^{\lam'}_{[a+1-i', \, a-1+j']}\right\} =
c \; \tilde{y}^\lam_{[a+1-i, \, a-1+j]}\;\tilde{y}^{\lam'}_{[a+1-i', \, a-1+j']},
\]
where (see Notation \ref{nota-s-interval})
\begin{align}\label{eq-c-1}
c & = \la s_{[1, a-i]}(\lam), \; s_{[1, a-i']}(\lam')\ra - \la s_{[1, a-1+j]}(\lam), \;s_{[1, a-1+j']}(\lam')\ra\\
\label{eq-c-2}& = \la v_{[1, i]}(\lam), \; v_{[1, i']}(\lam')\ra - \la (u_{[j, \, b]})^{-1}(\lam), \;
(u_{[j', b]})^{-1}(\lam')\ra.
\end{align}
\end{lemma}

\begin{proof}
By Lemma \ref{le-y-y-special}, the Poisson bracket between the two functions  is log-canonical
with the coefficient
$c$ given as in \eqref{eq-c-1}. For \eqref{eq-c-2}, we only need to note that
$s_{[1, a-i]} = v^{-1} v_{[1, i]}$, $s_{[1, a-i']} = v^{-1} v_{[1, i']}$,
$s_{[1, a-1+j]} = v^{-1}u(u_{[j, \, b]})^{-1}$, and $s_{[1, a-1+j']} = v^{-1}u(u_{[j', b]})^{-1}$.
\end{proof}

\noindent
{\it Proof of Theorem \ref{thm-bf-Fuv-Poi} by computation in coordinates.}
Associated to any double reduced word $(i_1, \ldots, i_{a+b})$ of $(u, v)$,
appended with $r$ entries $i_{a+b+1}, \ldots, i_{a+b+r}$ representing the $r$ simple reflections
in $W$,
Fomin and Zelevinsky defined in \cite{FZ:total}
a sequence $M_k$, $k = 1, \ldots, a+b+r$, of twisted generalized minors on $\Guv$, which are of the form
\[
M_k (g) = \Delta_{\gamma^k, \delta^k}(g'), \hs g \in \Guv, \; k = 1, \ldots, a+b + r,
\]
where $\gamma^k = u_{\geq k} (\omega_{|i_k|})$ and $\delta^k = v_{<k} (\omega_{|i_k|})$ (we refer to
\cite{FZ:total, KZ:integrable}
for the detail on the notation). Moreover, $\{M_k: k = 1, \ldots, a+b+r\}$ generates the
field of rational functions on $\Guv$. It is shown in  \cite[Theorem 2.6]{KZ:integrable} that with respect to the Poisson structure $\pist$ on $\Guv$, one has
\begin{equation}\label{eq-M-M-KZ}
\{M_k, \, M_{k'}\} = (\la \delta^k, \, \delta^{k'}\ra - \la \gamma^k, \, \gamma^{k'}\ra) M_k M_{k'}, \hs k \leq k'.
\end{equation}
Here recall that our Poisson structure $\pist$ on $G$ is the negative of the one used in \cite{KZ:integrable}.

Given reduced words $\bfu$ and $\bfv$  of $u$ and $v$, use any shuffle
of $\bfu$ and $\bfv$ to form a
double reduced word $(i_1, \ldots, i_{a+b})$ of $(u, v)$, and consider the sequence of
functions $M_k$, $k = 1, \ldots, a+b+r$ described above.
In our notation in \eqref{eq-Mij}, it is easy to see that for each $1 \leq k \leq a+b+r$, there exist
$0 \leq i(k) \leq a$ and $1\leq j(k) \leq b+1$ such that
\[
M_k (g)= M_{i(k), j(k)}^{\omega_{|i_k|}}(g) = \Delta_{(u_{[j(k), b]})^{-1}\omega_{|i_k|}, \,
v_{[1, i(k)]}\omega_{|i_k|}}(g'), \hs g \in \Guv,
\]
and $i(k) \leq i({k'})$ and $j(k) \leq j({k'})$ if $k \leq k'$.
 By \eqref{eq-Mij-y}, one has
\[
M_k = (F^{\bfu, \bfv})^* \left(\tilde{y}^{\omega_{|i_k|}}_{[a+1-i(k), a-1+j(k)]}\right), \hs k = 1, \ldots, a+b+r.
\]
By Lemma \ref{le-Y-Y} and comparing with \eqref{eq-M-M-KZ}, one sees immediately that
$F^{\bfu, \bfv}: (\Guv, \pist) \to (T \times \O^{(\bfv^{-1}, \bfu)}, 0 \bowtie \pi_{a+b})$ is Poisson.
The statement about $\hat{F}^{\bfu, \bfv}$ follows from Remark \ref{rk-F-equi} and the definitions of the Poisson
structures $\hat{\pi}_{\rm st}$ and $0 \bowtie \pi_{a+b}$.

\subsection{Completeness of Hamiltonian flows of Fomin-Zelevinsky minors}\label{subsec-complete-FZ-minors}
For $u, v \in W$, regarding $(G^{u, v}, \pist)$ as an affine Poisson variety, we now have one of the main results of the
paper.

\begin{corollary}\label{co-FZ-complete}
For any $u, v \in W$ and with respect to the Poisson structure $\pist$,
every Fomin-Zelevinsky minor on $\Guv$ has complete Hamiltonian flow with property $\calQ$.
\end{corollary}

\begin{proof}
Choose any reduced words $\bfu$ and $\bfv$ of $u$ and $v$ and
consider the corresponding Fomin-Zelevinsky Poisson isomorphism
$F^{\bfu, \bfv}: (\Guv, \pist) \to (T \times \O^{(\bfv^{-1}, \bfu)}_e, 0 \bowtie \pi_{a+b})$.
By Proposition \ref{pr-M-ij-y}, every Fomin-Zelevinsky minor on $\Guv$ is identified, via $F^{\bfu, \bfv}$, with
a regular function $\tilde{y} \in \widetilde{\calY}^{(\bfv^{-1}, \bfu)}$, which, by
Theorem \ref{thm-tcYGu}, has complete Hamiltonian flow in $T \times \O^{(\bfv^{-1}, \bfu)}_e$
with property $\calQ$.
\end{proof}

For $u \in W$, recall from Example \ref{ex-KZ-system} the Kogan-Zelevinsky integrable systems on $G^{u, u}$ defined
using reduced words of $u$.

\begin{corollary}\label{co-KZ-complete}
For every $u \in W$, all the Kogan-Zelevinsky integrable systems on $G^{u, u}$ have complete Hamiltonian flows
with property $\calQ$.
\end{corollary}

\subsection{Integrable systems on $(G^{u, u}/T, \hat{\pi}_{\rm st}$)}\label{subsec-KZ-system}
Let again $u \in W$ be arbitrary, and let $\bfu = (s_1, s_2, \ldots, s_n)$ be any reduced word of $u$.
Consider the open Fomin-Zelevinsky embedding
\[
\hat{F}^{(\bfu, \bfu)}: \;\; (G^{u, u}/T, \; \hat{\pi}_{\rm st}) \lrw (\O^{(\bfu^{-1}, \bfu)}, \; \pi_{2n}).
\]
Recall from $\S$\ref{subsec-int-O} the
integrable system $(y_1, \ldots, y_n)$ on $(\O^{(\bfu^{-1}, \bfu)}, \; \pi_{2n})$. Define
\[
\hat{M}_k = (\hat{F}^{(\bfu, \bfu)})^* y_k, \hs k = 1, \ldots, n,
\]
Then $(\hat{M}_1, \ldots, \hat{M}_n)$ is an integrable system on $G^{u, u}/T$ that has complete Hamiltonian flows with
property $\calQ$. To express the regular function $\hat{M}_k$ on $G^{u, u}/T$ using
the Fomin-Zelevinsky twist map $G^{u, u} \to G^{u^{-1}, u^{-1}}, g \mapsto g'$, note 
by Proposition \ref{pr-M-ij-y} that for $g \in G^{u, u}$, 
\begin{align*}
\hat{M}_k(gT) &= t^{s_n s_{n-1}\cdots s_{k+1}(\omega_{\alpha_k})} 
\Delta_{s_ns_{n-1} \cdots s_{k+1}\omega_{\alpha_k}, \; s_1s_2 \cdots s_k\omega_{\alpha_k}}(g')\\ 
& = t^{s_n s_{n-1}\cdots s_{k+1}(\omega_{\alpha_k})} M_{k, k+1}^{\omega_{\al_k}}(g),
\end{align*}
where recall from \eqref{eq-g-decomp-0} that $t = [g\overline{\overline{u}}^{\,-1}]_0^u$,
and the function $M_{k, k+1}^{\omega_{\al_k}}$ on $G^{u, u}$ is defined in Notation \ref{nota-Mij}. It is easy to see that
$t = ([\overline{\overline{u}}\,g']_0^u)^{-1}$. One thus has, for $k = 1, \ldots, n$,  
\[
\hat{M}_k(gT) = \left(\left[\overline{\overline{u}} g'\right]_0^{-1}\right)^{s_1s_2 \cdots s_{k}(\omega_{\al_k})}
\Delta_{s_ns_{n-1} \cdots s_{k+1}\omega_{\al_k}, \, s_1s_2 \cdots s_k\omega_{\al_k}}(g'), \hs g \in G^{u, u}.
\]

\subsection{Examples}\label{subsec-examples}

We give two examples to illustrate the Fomin-Zelevinsky embeddings and how
Fomin-Zelevinsky minors and their Hamiltonian flows can be expressed in Bott-Samelson coordinates.
The Poisson brackets in
Example \ref{ex-A2-121} and Example \ref{ex-G2-212} are computed using the GAP computer program written by B. Elek.

\begin{example}\label{ex-A2-121}
{\rm
Consider $G=SL(3,\mathbb{C})$ with the standard choices of maximal torus $T$
and the simple roots $\al_1$ and $\al_2$, and choose the bilinear form on $\g$ such that the
induced bilinear form $\lara$ on $\t^*$ satisfies  $\la \al_2, \al_2\ra = \la \al_1, \al_1\ra = 2$.
Let $w_0$ be the longest element in the Weyl group, and
take the reduced word $\bfu = (s_{\al_1}, s_{\al_2}, s_{\al_1})$ of $w_0$.
Let $\omega_1$ and $\omega_2$ be the fundamental weights corresponding to $\al_1$ and $\al_2$,
and consider Bott-Samelson coordinates $(a, b, x_1, \ldots,x_6)$ on $T \times \O^{(\bfu^{-1}, \bfu)}$,
where
\[
a(t) = t^{\omega_1} \hs \mbox{and} \hs b(t) =
t^{\omega_2}, \hs t \in T.
\]
For $g = (a_{ij})_{i, j = 1, 2, 3}$ and
$1 \leq i < k \leq 3$ and $1 \leq j < l \leq 3$, let $\Delta_{ik,jl}(g)$ be the
determinant of the submatrix of $g$ formed by the $i$'th and the $k$'th row and the $j$'th and the $l$'th column.
Recall that $g \in
G^{w_0, w_0}$ if and only if $a_{13}a_{31}\Delta_{12,23}\Delta_{23,12} \neq 0$.
A direct calculation shows that the open Fomin-Zelevinsky embedding
\[
F^{\bfu, \bfu}: \;\;G^{w_0, w_0} \lrw T \times \O^{(\bfu^{-1}, \bfu)} \cong (\CC^\times)^2 \times \CC^6
\]
is explicitly given by $SL(3, \CC) \ni g =(a_{ij}) \mapsto (a, b, x_1, \ldots, x_6)$, with
\begin{align*}
&a=1/\Delta_{12,23}, \hs b=1/a_{13}, \hs x_1=\Delta_{13,23}/\Delta_{12,23}, \hs x_2=a_{33}/a_{13}, \hs x_3=a_{23}/a_{13},\\ &x_4=\Delta_{13,12}/\Delta_{23,12}, \hs x_5=\Delta_{12,12}/\Delta_{23,12}, \hs x_6=a_{21}/a_{31}.
\end{align*}
The Zariski open subset $T \times \O^{(\bfu^{-1}, \bfu)}_e$ of $T \times \O^{(\bfu^{-1}, \bfu)}$ is defined in the coordinates
$(a, b, x_1, \ldots, x_6)$ by $y_1 \neq 0$ and $y_2 \neq 0$, where
\[
y_1 = x_1x_3x_4x_6-x_2x_4x_6-x_1x_6-x_1x_3x_5+x_2x_5+1  \hs \mbox{and} \hs
y_2 = x_2x_5-x_3x_4+1.
\]
One can see that $F^{\bfu, \bfu}(G^{w_0, w_0}) \subset T \times \O^{(\bfu^{-1}, \bfu)}_e$ directly by checking that
\[
y_1 (F^{\bfu, \bfu}(g)) = \frac{1}{a_{31} \Delta_{12, 23}} \hs \mbox{and} \hs
y_2 (F^{\bfu, \bfu}(g)) = \frac{1}{a_{13} \Delta_{23, 12}}, \hs g \in G^{w_0, w_0},
\]
and one can also solve for the inverse of $F^{\bfu, \bfu}: G^{w_0, w_0} \to T \times \O^{(\bfu^{-1}, \bfu)}_e$ directly.

In the coordinates $(a, b, x_1, \ldots, x_6)$, the Poisson structure
$0 \bowtie \pi_6$ on $T \times \O^{(\bfu^{-1}, \bfu)}$ has
\begin{align*}
&\{x_{1},x_{2}\} = -x_{1}x_{2}, \hhs
\{x_{1},x_{3}\} = x_{1}x_{3}-2x_2, \hhs
\{x_{1},x_{4}\} = -x_{1}x_{4}, \\
&\{x_{1},x_{5}\} = x_{1}x_{5}-2x_4, \hhs
\{x_{1},x_{6}\} = 2x_{1}x_{6}-2,\\
&\{x_{2},x_{3}\} = -x_{2}x_{3}, \hhs
\{x_{2},x_{4}\} = x_{2}x_{4}, \hhs
\{x_{2},x_{5}\} = 2x_{2}x_{5}-2x_{3}x_{4}+2,\\
&\{x_{2},x_{6}\} = x_{2}x_{6}-2x_3, \hhs
\{x_{3},x_{4}\} = 2x_{3}x_{4}-2, \hhs
\{x_{3},x_{5}\} = x_{3}x_{5},\\
&\{x_{3},x_{6}\} = -x_{3}x_{6},\hhs
\{x_{4},x_{5}\} = -x_{4}x_{5}, \hhs
\{x_{4},x_{6}\} = x_{4}x_{6}-2x_5,\hhs
\{x_{5},x_{6}\} = -x_{5}x_{6},\\
&\{a, \;x_j\} = -\la \omega_1, \lam_j\ra \,a x_j,  \hs \{b, x_j\} = -\la \omega_2, \lam_j\ra\, b x_j, \hs j = 1, \ldots, 6,
\hhs \mbox{where}\\
&\lam_1 = \al_1, \;\lam_2 = \al_1 + \al_2, \; \lam_3 = \al_2, \;
\lam_4 = -\al_2, \; \lam_5 = -\al_1-\al_2, \; \lam_6 = -\al_1.
\end{align*}
One computes directly that
the regular functions $M_1, \ldots, M_6$ on $G^{w_0, w_0}$ defined by Kogan and Zelevinsky in
\cite[Corollary 2.7]{KZ:integrable}
are given by
\begin{align*}
&M_1 = b, \hs M_2 = bx_4, \hs M_3 = a, \hs M_4 = ax_5, \\
&M_5 =ab^{-1}(x_3x_4-1), \hs M_6 = ab^{-1}(x_3x_4x_6-x_6-x_3x_5),
\end{align*}
and that $M_1, M_3, M_5$ form an integrable system on $G^{w_0, w_0}$.
For $i = 1, 3, 5$, denote by
$c \mapsto \phi_{i, c} \in {\rm Diff}((\CC^\times)^2 \times \CC^6)$, $c \in \CC$,
the Hamiltonian flow of $M_i$.  A direct calculation gives that for
$p = (a,b,x_1,x_2,x_3,x_4,x_5,x_6) \in (\CC^\times)^2 \times \CC^6$,
\begin{align*}
&\phi_{1,c}(p)=(a,\;b,\;x_1,\;x_2e^{-cM_1},\;x_3e^{-cM_1},\;x_4e^{cM_1},\; x_5e^{cM_1},\;
x_6),\\
&\phi_{3, c}(p) =
(a,\; b,\; x_1e^{-cM_3},\; x_2e^{-cM_3},\; x_3,\; x_4,\; x_5e^{cM_3},\; x_6e^{cM_3}),\\
&\phi_{5, c}(p) =
(a,\; b,\; x_1e^{-cM_5}+ab^{-1}x_2x_4(e^{cM_5}-e^{-cM_5})/M_5,\; x_2,\; x_3e^{-cM_5},\; x_4e^{cM_5},\\
&\hs \hs \hs \hs x_5,\; x_6e^{cM_5}+ab^{-1}x_3x_5(e^{-cM_5}-e^{cM_5})/M_5), \hs \;\mbox{if} \;\;  M_5 \neq 0, \\
&\phi_{5, c}(p) =
(a,\; b,\; x_1+2ab^{-1}x_2x_4c,\; x_2,\; x_3,\; x_4,\;x_5,\; x_6-2ab^{-1}x_3x_5c), \hs \;\mbox{if} \;\;  M_5 = 0,
\end{align*}
where for $i = 1, 3, 5$, $M_i$ also denotes the value $M_i(p)$ of the function $M_i$ at $p$.
\hfill $\diamond$
}
\end{example}

\begin{example}\label{ex-G2-212}
{\rm
Consider $G_2$ with the simple roots $\al_1$ and $\al_2$ and
$\la \al_2, \al_2\ra = 3\la \al_1, \al_1\ra = 6$.
Take $\bfu = (s_{\al_2}, s_{\al_1}, s_{\al_2})$ so that
$(\bfu^{-1}, \bfu) = (s_{\al_2}, s_{\al_1}, s_{\al_2}, s_{\al_2}, s_{\al_1}, s_{\al_2})$.
The Poisson structure $\pi_6$ on
$\O^{(\bfu^{-1}, \bfu)}$ in the Bott-Samelson coordinates $(x_1, \ldots, x_6)$ is given by
\begin{align*}
\{x_{1},x_{2}\} &= -3x_{1}x_{2}, \hhs
\{x_{1},x_{3}\} = -6x_{2}^3-3x_{1}x_{3}, \hhs
\{x_{1},x_{4}\} = 3x_{1}x_{4}, \\
\{x_{1},x_{5}\} &= -6x_{2}^2x_{4}+3x_{1}x_{5}, \hhs
\{x_{1},x_{6}\} = -18x_{2}^2x_{5}^2+6x_{1}x_{6}-18x_{2}x_{5}+6x_{3}x_{4}-6,\\
\{x_{2},x_{3}\} &= -3x_{2}x_{3}, \hhs
\{x_{2},x_{4}\} = 3x_{2}x_{4}, \hhs
\{x_{2},x_{5}\} = 2x_{2}x_{5}-2x_{3}x_{4}+2,\\
\{x_{2},x_{6}\} &= -6x_{3}x_{5}^2+3x_{2}x_{6}, \hhs
\{x_{3},x_{4}\} = 6x_{3}x_{4}-6, \hhs
\{x_{3},x_{5}\} = 3x_{3}x_{5},\\
\{x_{3},x_{6}\} &= 3x_{3}x_{6},\hhs
\{x_{4},x_{5}\} = -3x_{4}x_{5}, \hhs
\{x_{4},x_{6}\} = -6x_{5}^3-3x_{4}x_{6},\hhs
\{x_{5},x_{6}\} = -3x_{5}x_{6}.
\end{align*}
By Theorem \ref{thm-integrable-O}, one has the integrable system $\{y_1,y_2, y_3\}$ on $\O^{(\bfu^{-1}, \bfu)}$, where $y_1=x_3x_4-1$, $y_2=x_2x_5-x_3x_4+1$, and
\begin{align*}
&y_3=x_1x_3x_4x_6-x^3_2x_4x_6-x_1x_6
-x_1x_3x^3_5+x^3_2x^3_5+3x^2_2x^2_5-3x_2x_3x_4x_5\\
& \;\; \;\;\;\;+3x_2x_5+x^2_3x^2_4-2x_3x_4+1.
\end{align*}
For $i=1,2,3$, denote by $c \mapsto \varphi_{i, c} \in {\rm Diff}(\CC^6)$, $c \in \CC$, the Hamiltonian flow of $y_i$.
For $p = (x_1,x_2,x_3,x_4,x_5,x_6) \in \CC^6$, a direct calculation gives
\begin{align*}
&\varphi_{1,c}(p)=(x_1 + x^3_2x_4(e^{6cy_1}-1)/y_1,\;x_2,\;x_3e^{-6cy_1},\;
x_4e^{6cy_1},\; x_5,\;x_6 + x_3x^3_5(e^{-6cy_1}-1)/y_1), \\
&\hs \;\;\hs  \hs \;\mbox{if} \;\;  y_1 \neq 0,\\
&\varphi_{1,c}(p)=(x_1+6x^3_2x_4c,\;x_2,\;x_3,\;x_4,\; x_5,\;
x_6-6x_3x^3_5c), \;\; \mbox{if} \;\;  y_1 = 0,\\
&\varphi_{2, c}(p) =
(x_1,\; x_2e^{-2cy_2},\; x_3e^{-6cy_2},\; x_4e^{6cy_2},\; x_5e^{2cy_2},\; x_6),\\
&\varphi_{3, c}(p) =
(x_1e^{-6cy_3},\; x_2e^{-6cy_3},\; x_3e^{-12cy_3},\; x_4e^{12cy_3},\; x_5e^{6cy_3},\; x_6e^{6cy_3}),
\end{align*}
where for $i = 1, 2, 3$, $y_i$ also denotes the value $y_i(p)$ of the function $y_i$ at $p$.
\hfill $\diamond$
}
\end{example}

\appendix
\section{The Fomin-Zelevinsky embeddings are Poisson}\label{appendix-FZ-Poisson}
In this appendix we give a proof of Theorem \ref{thm-FZ-Poi} without computations in coordinates.
We first recall a construction of mixed product Poisson
structures given in \cite{LM:mixed} and express some Poisson structures related to
the standard Poisson Lie group $(G, \pist)$
as mixed product Poisson structures. We refer to \cite{LM:mixed} for basic facts on Poisson Lie groups and Poisson Lie group
actions needed in this appendix.

Throughout the appendix, if $(X, \pi)$ is a Poisson manifold and if $X_1$ is a Poisson submanifold with respect to
$\pi$, the restriction of $\pi$ to $X_1$ will also be denoted by $\pi$. If $A$ is a Lie group with Lie
algebra $\a$, for $\xi \in \a$, $\xi^L$ (resp. $\xi^R$) will denote the left (resp. right) invariant vector field on
$A$ with value $\xi$ at the identity of $A$.

\subsection{Quotient and mixed product Poisson structures}\label{subsec-fact-mixed}
Assume first that $(A, \piA)$ is a Poisson Lie group with a right and a left Poisson action
\[
(Z, \piZ) \times (A, \piA) \lrw (Z, \piZ) \hs \mbox{and} \hs
(A, \piA) \times (Y, \piY) \lrw (Y, \piY)
\]
on Poisson manifolds $(Z, \piZ)$ and $(Y, \piY)$. Suppose that the diagonal action
\[
(Z \times Y) \times A \lrw Z \times Y, \;\;(z, y) \cdot a = (za, \; a^{-1}y)
\]
is free and that the quotient, denoted by $Z \times_A Y$, is a manifold. Then \cite[$\S$7]{LM:mixed}
the projection of the
product Poisson structure $\piZ \times \piY = (\piZ, 0) +(0, \piY)$ on $Z \times Y$
is a well-defined Poisson structure, called the {\it quotient} Poisson structure, on $Z \times_A Y$.

Assume now that a pair of dual Poisson Lie groups $(A, \piA)$ and $(A^*, \piAs)$ act respectively
on the Poisson manifolds $(Y, \piY)$ and $(X, \piX)$ by left and right Poisson actions
\[
\lam: \;\; (A, \piA) \times (Y, \piY) \longrightarrow (Y, \piY) \hs
\mbox{and} \hs
\rho: \;\; (X, \piX) \times (A^*, \piAs) \longrightarrow (X, \piX).
\]
Let $\a$ be the Lie algebra of $A$, so that the Lie algebra of $A^*$ is identified with the dual vector space $\a^*$ of $\a$.
Denote the corresponding Lie algebra actions by
\begin{align*}
&\rho: \;\; \a^* \lrw {\mathfrak{X}}^1(X), \;\; \rho(\xi^*)(x) =  \frac{d}{dt}|_{t = 0} (x \exp(t\xi^*)), \hs x \in X, \, \xi^* \in \a^*,\\
&\lam: \;\; \a \lrw {\mathfrak{X}}^1(Y), \;\;\lam(\xi)(y) = \frac{d}{dt}|_{t = 0} (\exp(t\xi)y), \hs y \in Y, \, \xi \in \a.
\end{align*}
Then \cite[$\S$2]{LM:mixed} the bivector field $\piX \times_{(\rho, \lam)} \piY$ on the product manifold $X \times Y$ given by
\begin{equation}\label{eq-mixed}
\piX \times_{(\rho, \lam)} \piY = (\piX, 0) + (0, \piY) -\sum_{i=1}^m (\rho(\xi_i^*), \, 0) \wedge (0, \; \lam(\xi_i))
\end{equation}
is a Poisson structure on $X \times Y$, called the {\it mixed product of $\piX$ and $\piY$ associated to $(\rho, \lam)$}, where $(\xi_i)_{i=1}^m$ is any basis for $\a$ and $(\xi^*_i)_{i=1}^m$ the dual basis for $\a^*$. Here, for
$V \in {\mathfrak{X}}^1(X)$ and $U \in {\mathfrak{X}}^1(Y)$, $(V, 0) \wedge (0, U) \in {\mathfrak{X}}^2(X \times Y)$ is
such that
\[
((V, 0) \wedge (0, U))(f, g) = V(f)U(g)
\]
for functions $f$ on $X$ and $g$ on $Y$.
In the same setting, let $\lam_A$ be the left action of $A$ on itself by
left multiplication, and
equip $X \times A$ with the mixed product Poisson structure
$\pi_{\sX \times \sA} = \piX \times_{(\rho, \lam_\sA)} \piA$.
One then has the (free) right Poisson action
\[
(X \times A, \; \pi_{\sX \times \sA}) \times (A, \piA) \longrightarrow
(X \times A, \; \pi_{\sX \times \sA}), \;\; ((x, a), a') \longmapsto (x, aa'),
\]
of the Poisson Lie group $(A, \piA)$, which, together with the left Poisson action $\lam$ of $(A, \piA)$ on $(Y, \piY)$,
gives rise to the quotient Poisson structure, denoted by $\pi$, on
$(X \times A)\times_A Y$. Let $\varpi: X \times A \times Y \to (X \times A)\times_A Y$ be the quotient map,
and set $[(x, a), y] = \varpi(x, a, y)$ for $(x, a, y)  \in X \times A \times Y$.
Note  that one has the diffeomorphism
\[
\Psi: \;\; (X \times A)\times_A Y \lrw X \times Y, \;\; [(x, a), y] \longmapsto (x, \, ay), \hs x \in X, \, a \in A, \, y \in Y.
\]

\begin{lemma}\label{le-mixed-general}
As Poisson structures on $X \times Y$, one has $\Psi(\pi) = \piX \times_{(\rho, \lam)} \piY$.
\end{lemma}

\begin{proof}
Let $e$ be the identity element of $A$, and let  $(x, y) \in X \times Y$.
By definition and a straightforward calculation, one has
\begin{align*}
\Psi(\pi([(x, e), y]))& = \Psi(\varpi(\pi_{\sX \times \sA}(x, e), \, 0) + \varpi(0, \, \piY(y)))\\
& = (\piX(x), \; 0) + (0, \; \piY(y)) -\sum_{i=1}^m \Psi(\varpi(\rho(\xi^*_i), 0, 0)) \wedge \Psi(\varpi(0, \xi_i^R, 0))\\
& =(\piX \times_{(\rho, \lam)} \piY)(x, y).
\end{align*}
\end{proof}

\subsection{Drinfeld double of the Poisson Lie group $(G, \pist)$}\label{subsec-more-pist}
We now turn to a complex semi-simple Poisson Lie group
$(G, \pist)$ as defined in $\S$\ref{subsec-pist}.
As both $B$ and $B_-$ are Poisson submanifolds of $G$ with respect to $\pist$, we have
the Poisson Lie subgroups  $(B, \pist)$ and $(B_-, \pist)$ of $(G, \pist)$.

Denote the Lie algebras of $B, B_-, N$ and $N_-$ respectively
by $\b, \b_-, \n$ and $\n_-$, and
define a non-degenerate bilinear pairing between $\b_-$ and $\b$ by
\begin{equation}\label{eq-bb-pairing-0}
\la x_- + x_0, \;\, y_+ + y_0\ra_{(\b_-, \b)} = \frac{1}{2}\la x_-,\; y_+\ra_\g + \la x_0, \;y_0\ra_\g, \hs
\end{equation}
where $x_- \in \n_-,\,
\, x_0, \,y_0 \in \t, \, y_+ \in \n$. Then
$(B_-, \pist)$ and $(B, -\pist)$ form a pair of dual  Poisson Lie groups
with respect to the pairing $\lara_{(\b_-, \b)}$  (see \cite[$\S$2.3]{LM:mixed}).

The {\it Drinfeld double} of the  Poisson Lie group $(G, \pist)$  is the product Lie group $G \times G$ together with the multiplicative
Poisson structure $\Pist$ given by
\begin{equation}\label{eq-Pist}
\Pist = (\pist, 0) + (0, \pist) + \mu_1  + \mu_2,
\end{equation}
where
\begin{equation}\label{eq-mu-12}
\mu_1 = \sum_{i=1}^m (\xi_i^R, 0) \wedge (0, x_i^R), \hs \hs
\mu_2 = -\sum_{i=1}^m (\xi_i^L, 0) \wedge (0, x_i^L),
\end{equation}
and
$\{x_i\}_{i=1}^m$ is any basis of $\b_-$ and $\{\xi_i\}_{i=1}^m$ its dual basis of $\b$ under the
pairing $\lara_{(\b_-, \b)}$ in \eqref{eq-bb-pairing-0}.
See \cite{etingof-schiffmann} and \cite[$\S$6.1]{LM:flags}.
See also \cite[$\S$6.2]{LM:mixed} and \cite[$\S$4.2]{LM:groupoids} for an interpretation of
$\Pist$ as a mixed product Poisson structure on $G \times G$ defined by Poisson Lie group actions.

As a Drinfeld double of $(G, \pist)$, the Poisson Lie group $(G \times G, \Pist)$ contains
$(G, \pist)$, identified with the diagonal of $G \times G$, as a Poisson Lie subgroup
\cite[$\S$6.2]{LM:mixed}.  Moreover, any
 left $(B \times B)$ and right $(B_- \times B_-)$-invariant submanifold of $G \times G$ is
a Poisson submanifold with respect to $\Pist$.
We will use these facts in $\S$\ref{subsec-several-mixed} to decompose some Poisson structures related $\pist$ as product Poisson structures.

\subsection{Some mixed product Poisson structures associated to $(G, \pist)$}\label{subsec-several-mixed}
Consider the flag manifold $G/B$ with the Poisson structure $\pi_1$. Then each Bruhat cell
$BwB/B$, $w \in W$, is a Poisson submanifold with respect to $\pi_1$.
We will also consider the quotient manifold $B_-\backslash G$ with the quotient Poisson structure $\pi_{-1}^\prime$
of $\pist$.
Then, again, for each $w \in W$, $B_-\backslash B_-wB_-$ is a Poisson submanifold of
$(B_-\backslash G, \; \pi_{-1}^\prime)$.

Fix $w \in W$ and let $\dw \in N_G(T)$ be any representative of $w$. Recall
from $\S$\ref{subsec-FZ-map} that
\[
C_\dw = (N \dw)\cap(\dw N_-),
\]
and that the multiplication map of $G$ gives diffeomorphisms
\[
C_\dw \times B \longrightarrow BwB \hs \mbox{and} \hs B_- \times C_\dw \lrw B_-wB_-.
\]
By \cite[Lemma 14]{LM:groupoids}, one has the anti-Poisson isomorphism
\begin{equation}\label{eq-I-dw}
I_{\dw}: \;\; (B\dw B/B, \, \pi_1) \lrw (B_-\backslash B_-\dw B_-, \, \pi_{-1}^\prime), \;\;\;
c_\cdot B \longmapsto {B_-}_\cdot c, \hs c \in C_{\dw}.
\end{equation}

\begin{notation}\label{nota-rho-dw}
{\rm Let $\rho_\dw$ be the right action of $B_-$ on $B\dw B/B$ obtained via the right $B_-$ action on
$B_-\backslash B_-wB_-$ through $I_\dw$. More precisely, $\rho_\dw$ denotes the right Poisson action
\begin{align}\label{eq-rho-dw}
\rho_\dw: & \;\; (BwB/B, \;\pi_1) \times (B_-,\; -\pist) \lrw (BwB/B, \;\pi_1),\\
\nonumber & \;\; (c_\cdot B, \; b_-) \longmapsto c'_\cdot B, \hs c, c' \in C_\dw, \, b_-, \, b_-^\prime \in B_-,\;
cb_- = b_-^\prime c',
\end{align}
of the Poisson Lie group $(B_-, -\pist)$. Let $\lam_\sB$ be the left action of the Poisson Lie group
$(B, \pist)$ on itself by left translation. Note that
$(B_-, -\pist)$ and $(B, \pist)$ form a pair of dual Poisson Lie groups under the pairing
$-\la \, ,\, \ra_{(\b_-, \b)}$ in \eqref{eq-bb-pairing-0}.
\hfill $\diamond$
}
\end{notation}

Consider now the isomorphism
\begin{equation}\label{eq-Jdw}
J_{\dw}: \;\; BwB \lrw (BwB/B) \times B, \;\; \; cb \longmapsto (c_\cdot B, \; b), \hs c \in C_\dw, \; b \in B.
\end{equation}

\begin{lemma}\label{le-Jdw}
As Poisson structures on $(BwB/B) \times B$, one has
\begin{equation}\label{eq-Jdw-1}
J_{\dw}(\pist) = \pi_1 \times_{(\rho_\dw, \lam_\sB)} \pist = (\pi_1, 0) + (0, \pist) +\sum_{i=1}^m
(\rho_\dw(x_i), 0) \wedge (0, \xi_i^R),
\end{equation}
where $\{x_i\}_{i=1}^m$ is any basis of $\b_-$ and $\{\xi_i\}_{i=1}^m$ its dual of $\b$ under the
pairing in \eqref{eq-bb-pairing-0}.
\end{lemma}

\begin{proof}
We prove Lemma \ref{le-Jdw} using a similar result for $(B_-wB_-, \pist)$  in \cite[Remark 9]{LM:groupoids}.
Consider the map
\[
J_\dw^\prime: \;\; B_-wB_- \lrw B_- \times (BwB/B), \;\; b_-c \longmapsto (b_-, \; c_\cdot B), \hs b_- \in B_-, \;
c \in C_\dw.
\]
Let $\rho_{\sB_-}$ be the right Poisson action of the Poisson Lie group $(B_-, \pist)$ on itself by
right translation, and let $\lam_+$ be the left Poisson action of the Poisson Lie group
$(B, -\pist)$ on $(BwB/B, -\pi_1)$ by left translation. Recall that $(B_-, \pist)$ and $(B, -\pist)$ form
a pair of dual Poisson Lie groups under the pairing $\la \, , \, \ra_{(\b_-, \b)}$ of their Lie algebras.
Define the Poisson structure $\pi'$ on $B_- \times (BwB/B)$ by
\[
\pi' = \pist \times_{(\rho_{\sB_-}, \, \lam_+)} (-\pi_1) = (\pist, 0) + (0, -\pi_1) -\sum_{i=1}^m
(x_i^L, 0) \wedge (0, \lam_+(\xi_i)).
\]
It is shown in \cite[Remark 9]{LM:groupoids} that $J_\dw^\prime(\pist) = \pi'$ as Poisson structures on $B_- \times
(BwB/B)$. Consider (see \cite[$\S$1.5]{FZ:total}) the involutive anti-automorphism $\Theta$ of $G$ determined by
\[
\Theta(t) = t, \;\;\; t \in T, \hs \Theta(u_\al(c)) = u_{-\al}(-c), \hs \al \in \Gamma, \; c \in \CC.
\]
It is easy to see from the definition that $\Theta(\pist) = \pist$. Consider now the sequence of Poisson
isomorphisms
\begin{align*}
(BwB, \; \pist) &\stackrel{\Theta}{\lrw} (B_- w^{-1}B_-, \; \pist) \stackrel{J_{\Theta(\dw)}^\prime}{\lrw}
(B_- \times (Bw^{-1}B/B), \; \pi')\\
& \stackrel{\Theta \times \Theta}{\lrw} (B \times (B_-\backslash B_- w B_-), \; \pi^{(2)})
\stackrel{S}{\lrw}((B_-\backslash B_-wB_-) \times B, \; \pi^{(3)})\\
&\stackrel{I_\dw^{-1} \times {\rm Id}_\sB}{\lrw} ((BwB/B) \times B, \; \pi^{(4)}),
\end{align*}
where $S$ is the map that switches the two factors, and $I_\dw$ is given in \eqref{eq-I-dw}.  It is easy to
see that $J_\dw$ is the composition of the above five maps. A direct calculation and the facts that
$\Theta$ is Poisson and $I_\dw$ is anti-Poisson
 show that
\[
\pi^{(4)} = (\pi_1, 0) + (0, \pist) + \sum_{i=1}^m( \rho_\dw(x_i), 0) \wedge (0, \xi_i^R).
\]
As $\{-x_i\}_{i=1}^m$ and $\{\xi_i\}_{i=1}^m$ are bases of $\b_-$ and $\b$ that are dual to each other under the
pairing $-\la \,, \, \ra_{(\b_-, \b)}$ between $\b_-$ and $\b$, we know that
$J_\dw(\pist) = \pi^{(4)} = \pi_1 \times_{(\rho_\dw, \lam_\sB)} \pist$.
\end{proof}

Let now $(Y, \piY)$ be any Poisson manifold with a left Poisson action
\[
\lam: \;\; (B, \pist) \times (Y, \piY) \lrw (Y, \piY)
\]
by the Poisson Lie group $(B, \pist)$. Consider the quotient manifold $B w B\times_B Y$ with
the quotient Poisson structure
$\pi \stackrel{{\rm def}}{=} \varpi(\pist \times \piY)$,
where $\varpi: BwB \times Y \to BwB \times_B Y$ is the natural projection. On the other hand, note that one has the
isomorphism
\begin{equation}\label{eq-Jw}
J_{\dw, \sY}: \;\; BwB \times_B Y \lrw (BwB/B) \times Y, \;\;\; [cb, \, y] \longmapsto (c_\cdot B, \; by),
\end{equation}
where $c \in C_\dw, \, b \in B$ and $y \in Y$.

\begin{lemma}\label{le-BwB-Y}
As Poisson structures on $(BwB/B) \times Y$, one has
\[
J_{\dw, \sY}(\pi) = \pi_1 \times_{(\rho_\dw, \lam)} \piY = (\pi_1, 0) + (0, \piY) +\sum_{i=1}^m
(\rho_\dw(x_i), 0) \wedge (0, \lam(\xi_i)),
\]
where, again,  $\{x_i\}_{i=1}^m$ is any basis of $\b_-$ and $\{\xi_i\}_{i=1}^m$ its dual the basis of $\b$
under the
pairing $\lara_{(\b_-, \b)}$ in \eqref{eq-bb-pairing-0}.
\end{lemma}

\begin{proof}
The assertion follows from Lemma \ref{le-Jdw} and
Lemma \ref{le-mixed-general}.
\end{proof}

\subsection{Fomin-Zelevinsky embeddings for reduced double Bruhat cells}\label{subsec-Luv-Ouv}
Let $u, v \in W$ and recall that the reduced double Bruhat cell $\Guv/T \subset G/T$ has the Poisson
structure $\hat{\pi}_{\rm st}$, which is the projection to $G/T$ of the Poisson structure $\pist$ on $G$.
Recall from $\S$\ref{subsec-gBC} and $\S$\ref{subsec-FZ-map} the generalized Bruhat cell
$\O^{(v^{-1}, u)} \subset F_2$ and its Zariski open subset $\O^{(v^{-1}, u)}_e$, and recall from
$\S$\ref{subsec-FZ-map} the Fomin-Zelevinsky embedding
\[
\hat{F}^{u, v}: \;\;  \Guv/T \lrw \O^{(v^{-1}, u)}, \;\;
\hat{F}^{u, v}(gT) = [(m \bbv)^{-1}, \; n\bu],\hs g \in \Guv,
\]
which induces a biregular isomorphism from $\Guv/T$ to $\O^{(v^{-1}, u)}_e$. Here again
 $g \in \Guv$ is uniquely written as $g = n\bu n't' = m\bbv \, m' t$ with
$n\bu \in C_\bu, \, (m\bbv)^{-1} \in C_{\bbv^{\, -1}}, \, n' \in N, \, m' \in N_-$ and $t, t' \in T$.
The following Proposition
\ref{pr-Guv-Ouv-reduced} is the first part of Theorem \ref{thm-FZ-Poi}.

\begin{proposition}\label{pr-Guv-Ouv-reduced} For any $u, v \in W$,  the Fomin-Zelevinsky embedding
\[
\hat{F}^{u, v}: \;\;  (\Guv/T,\, \,\hat{\pi}_{\rm st}) \lrw (\O^{(v^{-1}, u)}, \; \pi_2)
\]
is Poisson.
\end{proposition}

\begin{proof} Let $(G \times G, \Pist)$ be the Drinfeld double of the Poisson Lie group $(G, \pist)$, and let
$\Pi_1$ be the quotient of $\Pist$ to $(G \times G)/(B\times B_-) = G/B \times G/B_-$. As the diagonal embedding
$(G, \pist) \to (G \times G, \Pist)$ is Poisson, the map
\[
\phi_1: \;\; (G/T, \pist) \lrw (G/B \times G/B_-, \; \Pi_1), \;\; g_\cdot T \longmapsto (g_\cdot B, g_\cdot B_-),
\]
is Poisson. Let $\pi_{-1}$ be the projection of $\pist$ to $G/B_-$, and let $\lam_+$ and $\lam_-$ be
the respective left actions of $G$ on $G/B$ and $G/B_-$ by left translation.  In the notation of \eqref{eq-Pist}, one has
\[
\Pi_1= (\pi_1, 0) + (0, \pi_{-1}) +\sum_{i=1}^m (\lam_+(\xi_i), \, 0) \wedge (0, \, \lam_-(x_i)).
\]
Let $\phi_2: G/B \times G/B_- \to (B_-\backslash G) \times (G/B), (g_\cdot B, h_\cdot B_-) \mapsto
({B_-}_\cdot h^{-1}, \, g_\cdot B)$. Then
\[
\phi_2(\Pi_1) = (-\pi_{-1}^\prime, 0) + (0, \pi_1) +\sum_{i=1}^m (\rho_-(x_i), 0)\wedge (0, \, \lam_+(\xi_i))
=(-\pi_{-1}^\prime) \times_{(\rho_-, \lambda_+)} \pi_1,
\]
where $\rho_-$ is the right Poisson Lie group action
\[
\rho_-: \;\; (B_-\backslash G, \,-\pi_{-1}^\prime) \times (B_-, \,-\pist) \lrw
(B_-\backslash G, \,-\pi_{-1}^\prime), \;\;\; ((B_-)_\cdot g, \;b_-) \longmapsto
(B_-)_\cdot gb_-.
\]
Restricting $\phi_1$ to $(\Guv/T, \hat{\pi}_{\rm st})$ and using $I_{\bbv^{\, -1}}$ in \eqref{eq-I-dw},
one has the Poisson morphism
\[
(I_{\bbv^{\, -1}})^{-1} \circ \phi_2 \circ (\phi_1|_{\Guv/T}): \;
(\Guv/T, \hat{\pi}_{\rm st}) \to \left((Bv^{-1}B/B) \times (BuB/B), \; \pi_1 \times_{(\rho_{\bbv^{\, -1}}, \lam_+)}
\pi_1\right).
\]
Taking now $w = v^{-1}, \dw = \bbv^{\, -1}$, and $(Y, \piY) = (BuB/B, \pi_1)$ in
Lemma \ref{le-BwB-Y}, one sees that
\[
\hat{F}^{u, v} = (J_{\bbv^{\, -1}})^{-1} \circ (I_{\bbv^{\, -1}})^{-1} \circ \phi_2 \circ (\phi_1|_{\Guv/T})
\]
is Poisson from $(\Guv/T,\, \hat{\pi}_{\rm st})$ to $(\O^{(v^{-1}, u)}, \pi_2)$.
\end{proof}

\subsection{Double Bruhat cells vs reduced double Bruhat cells}
Recall that the projection map $\varpi: (G, \pist) \to (G/T,
\hat{\pi}_{\rm st})$ is Poisson by definition. For $u, v \in W$, we now give another relation
between the two Poisson manifolds
$(\Guv, \pist)$ and $(\Guv/T, \, \hat{\pi}_{\rm st})$.

\begin{lemma}\label{le-bvb-T}
Equip $T$ with the zero Poisson structure. For any $v \in W$, the map
\[
(B_- v B_-, \, \pist) \longrightarrow (T, \, 0), \;\; g \longmapsto [g\bbv^{\, -1}]_0^v, \hs g \in B_- v B_-,
\]
is Poisson.
\end{lemma}

\begin{proof}
By \cite[Lemma 10]{LM:groupoids}, both $N_-$ and $C_{\bbv} = (N \bbv) \cap (\bbv N_-)$
are coisotropic submanifolds of $G$ with respect to the
Poisson structure $\pist$. It follows by the multiplicativity of $\pist$ that
$N_- \bbv N_- = N_- C_{\bbv}$  is also coisotropic with respect to $\pist$. Writing $g \in B_- v B_-$ as $g = g_1 t$, where $g_1 \in  N_- \bbv N_-$ and $t \in T$,
one has $\pist(g) = r_t \pist(g_1)$, so
$\varpi(\pist(g)) = \varpi(\pist(g_1)) = 0$.
\end{proof}

It is clear that the $T$-action on $G/T$ by left translation preserves the Poisson structure $\hat{\pi}_{\rm st}$.
Recall from $\S$\ref{subsec-T-O} that associated to any linear map $M: \t^*\to \t$ one has the Poisson structure
$0 \bowtie_M \hat{\pi}_{\rm st}$ on $T \times (G/T)$. By slight abuse of notation, for $v \in W$,  we set
\[
0 \bowtie_v \hat{\pi}_{\rm st} = 0 \bowtie_M \hat{\pi}_{\rm st}.
\]
to be the Poisson structure on $T \times (G/T)$ associated to the linear map
\[
M: \;\;\; \t^* \longrightarrow \t, \;\; M(\lam) = -v(\lam^\#), \hs \lam \in \t^*.
\]
It is clear that  $T \times ((B_-vB_-)/T)$ and $T \times (\Guv/T)$, $u \in W$,
are Poisson submanifolds of $(T \times (G/T), \;0 \bowtie_v \hat{\pi}_{\rm st})$.

\begin{proposition}\label{pr-nu}
The maps
\begin{align*}
&\nu': \;\; (B_-vB_-, \pist) \longrightarrow (T \times (B_-vB_-/T), \;
0 \bowtie_v \hat{\pi}_{\rm st}), \;\; \;g \longmapsto
\left([g\bbv^{\, -1}]_0^v, \; g_\cdot T\right),\\
&\nu: \;\; (\Guv, \; \pist) \longrightarrow (T \times (\Guv/T), \; 0 \bowtie_v \hat{\pi}_{\rm st}),\;\;\; g \longmapsto
\left([g\bbv^{\, -1}]_0^v, \; g_\cdot T\right),
\end{align*}
are Poisson isomorphism.
\end{proposition}

\begin{proof} As $\nu$ is the restriction of $\nu'$ to $\Guv/T$, it is enough to prove the assertion for $\nu'$.
Note that if $g \in B_-vB_-$ is written as $g = m \bbv m' t$, where $m, m' \in N_-$ and $t \in T$, then
$[g\bbv^{\, -1}]_0^v = t$. It follows that $\nu'$ is an isomorphism.
Consider now the map
\[
\kappa: \;\; G \times (B_-vB_-) \longrightarrow T \times (G/T), \;\; (g', g) \longmapsto
\left(([g\bbv^{\, -1}]_0^v)^{-1}, \; g^\prime_\cdot T\right),
\]
and equip $G \times (B_-v B_-)$ with the Poisson structure $\Pist$.
By Lemma \ref{le-bvb-T},
$\kappa$ maps the bi-vector field $(\pist, 0) + (0, \pist)$ on
$G \times (B_-vB_-)$ to $(0, \hat{\pi}_{\rm st})$. Moreover, it is clear that for any $x \in \n_-$,
$\kappa(0, x^R) = \kappa(0, x^L) = 0$. Thus
$\kappa(\mu_1) = \kappa(\mu_1^\prime)$ and $\kappa(\mu_2) = \kappa(\mu_2^\prime)$, where
$\mu_1$ and $\mu_2$ are given in \eqref{eq-mu-12}, and
\[
\mu_1^\prime = \sum_{i=1}^r (\xi_i^R, 0) \wedge (0, x_i^R), \hs \text{and} \hs \mu_2^\prime
 = -\sum_{i=1}^r (\xi_i^L, 0) \wedge (0, x_i^L),
\]
with $\{x_i\}_{i = 1, \ldots, r}$ and $\{\xi_i\}_{i = 1, \ldots, r}$  any bases of $\t$ such that
$\la x_i, \; \xi_j\ra = \delta_{i, j}$ for $i, j = 1, \ldots, r$.
On the other hand, it is clear that $\kappa(x^L, 0) = 0$ for every $x \in \t$. Thus $\kappa(\mu_2^\prime) = 0$,
and we only need to compute $\kappa(\mu_1^\prime)$. It is easy to see that for any $x \in \t$,
\[
\kappa(x^R, 0) = (0, \theta(x)) \hs \mbox{and} \hs
\kappa(0, x^R) = ((v^{-1}(x))^R, 0),
\]
where $\theta(x) \in {\mathfrak{X}}^1(G/T)$ is given by
\[
\theta(x)(g_\cdot T) = \frac{d}{ds}|_{s = 0} (\exp(sx)g_\cdot T), \hs g \in G,
\]
and recall that $x^R$ is the right (also left) invariant vector field on $T$ with value $x$ at the
identity element of $T$. Thus
\[
\kappa(\mu_1^\prime) = -\sum_{i=1}^r ((v^{-1}(x_i))^L, 0) \wedge (0, \theta(\xi_i)).
\]
Let $\{h_i\}$ be an orthonormal basis of $\t$ with respect to $\lara$. Then
\[
\kappa(\Pist) = (0, \hat{\pi}_{\rm st}) - \sum_{i=1}^r (h_i^R, 0) \wedge (0, \theta(v(h_i))) = 0 \bowtie_v \hat{\pi}_{\rm st}.
\]
It follows that $\kappa: (G \times (B_-vB_-), \, \Pist) \to (T \times (G/T), \, 0 \bowtie_v \hat{\pi}_{\rm st})$
is a Poisson map. As
\[
\iota: \;\; (B_-vB_-, \pist) \longrightarrow (G \times (B_-vB_-), \, \Pist), \;\; g \longmapsto (g, \, g),
\hs g \in B_-vB_-,
\]
is Poisson, $\nu' = \kappa \circ \iota:  (B_-vB_- , \pist)\to (T \times (B_-vB_-/T), 0 \bowtie_v \hat{\pi}_{\rm st})$
is Poisson.
\end{proof}

\subsection{Fomin-Zelevinsky embeddings for double Bruhat cells}\label{subsec-proof-Poi}
Let again $u, v \in W$ and recall the Fomin-Zelevinsky embedding
\[
F^{u, v}: \;\; \Guv \lrw T \times \O^{(v^{-1}, u)}, \;\; F^{u, v}(g) = ([g \bbv^{\, -1}]_0^v, \; \hat{F}^{u, v}(g_\cdot T)),
\]
where $\hat{F}^{u, v}: \Guv/T \to \O^{(v^{-1}, u)}$ is the Fomin-Zelevinsky embedding for $\Guv/T$.
Recall from Notation \ref{nota-bpi} the definition of the Poisson structure $0 \bowtie \pi_2$ on $T \times \O^{(v^{-1}, u)}$.
The following Proposition
\ref{pr-Guv-Ouv} is the second part of Theorem \ref{thm-FZ-Poi}.

\begin{proposition}\label{pr-Guv-Ouv} For any $u, v \in W$,  the Fomin-Zelevinsky embedding
\[
{F}^{u, v}: \;\;  (\Guv,\, \, \pist) \lrw (\O^{(v^{-1}, u)}, \; 0 \bowtie \pi_2)
\]
is Poisson.
\end{proposition}

\begin{proof} It is easy to see that
the map $\hat{F}^{u, v}: \Guv/T \to \O^{(v^{-1}, u)}$ is $T$-equivariant, where $\Guv/T$ has the $T$-action by
\[
t \circ (g_\cdot T) = (t^{v^{-1}}g)_\cdot T, \hs t \in T, \; g \in \Guv,
\]
and $\O^{(v^{-1}, u)} \subset F_2$ has the $T$-action given in \eqref{eq-T-Fn}. It now follows from Proposition \ref{pr-nu}
and the definitions of the Poisson structures $0 \bowtie_v \hat{\pi}_{\rm st}$ and $0 \bowtie \pi_{2}$ that
$F^{u, v}$ is Poisson.
\end{proof}

\end{document}